\documentclass[reqno]{amsart}
\usepackage{amssymb}
\usepackage{hyperref}
\usepackage{latexsym}
\usepackage{amsmath}
\usepackage[usenames,dvipsnames,svgnames,table]{xcolor}

 \usepackage{amsfonts, amsmath, bbm, amssymb, amsthm}
\newcommand{\R}{\mathbb{R}}
\newcommand{\C}{\mathbb{C}}
\newcommand{\N}{\mathbb{N}}

\newcommand{\D}{\mathcal{D}}

\newcommand{\tr}{\boldsymbol{t}_\Sigma}

\newcommand{\dom}{\textup{dom }}
\newcommand{\ran}{\textup{ran }}
\newcommand{\sign}{\textup{sign}}

\newcommand{\supp}{ \textup{supp }}

\newcommand{\abs}[1]{\left\lvert{#1}\right\rvert}

\newcommand{\norm}[1]{{\left\lVert{#1}\right\rVert}}
\newcommand{\spv}[2]{{\left({#1},{#2}\right)}}
\newcommand{\spf}[2]{{\left\langle{#1},{#2}\right\rangle}}

\theoremstyle{plain}
\newtheorem{theorem}{Theorem}[section]
\newtheorem{lemma}[theorem]{Lemma}
\newtheorem{corollary}[theorem]{Corollary}
\newtheorem{proposition}[theorem]{Proposition}

\theoremstyle{definition}
\newtheorem{definition}[theorem]{Definition}

\newtheorem{hypothesis}[theorem]{Hypothesis}

\theoremstyle{remark}
\newtheorem{remark}[theorem]{Remark}

\numberwithin{equation}{section}

\begin{document}

\title[Approximation of Dirac operators with $\boldsymbol{\delta}$-potentials]{Approximation of Dirac operators with $\boldsymbol{\delta}$-shell potentials in the norm resolvent sense, I. Qualitative results}

\author[J. Behrndt]{Jussi Behrndt}
 \address{Technische Universit\"{a}t Graz\\
 Institut f\"ur Angewandte Mathematik\\
 Steyrergasse 30\\
 8010 Graz, Austria}
 \email{behrndt@tugraz.at}
% \urladdr{\url{https://www.math.tugraz.at/~behrndt/}}
% 

\author[M. Holzmann]{Markus Holzmann}
 \address{Technische Universit\"{a}t Graz\\
 Institut f\"ur Angewandte Mathematik\\
 Steyrergasse 30\\
 8010 Graz, Austria}
 \email{holzmann@math.tugraz.at}
% \urladdr{\url{https://www.math.tugraz.at/~behrndt/}}

\author[C. Stelzer]{Christian Stelzer}
 \address{Technische Universit\"{a}t Graz\\
 Institut f\"ur Angewandte Mathematik\\
 Steyrergasse 30\\
 8010 Graz, Austria}
 \email{christian.stelzer@tugraz.at}
% \urladdr{\url{https://www.math.tugraz.at/~behrndt/}}

\maketitle
% 
% \title{On the approximation of Dirac operators with $\boldsymbol{\delta}$-shell potentials in the norm resolvent sense}
% 
% \author{Jussi~Behrndt\footnote{
% Institut f\"{u}r Angewandte Mathematik,
% Technische Universit\"{a}t Graz,
% Steyrergasse 30, A 8010 Graz, Austria,
% \texttt{behrndt@tugraz.at, holzmann@math.tugraz.at, christian.stelzer@tugraz.at}} \and Markus Holzmann$^*$ \and
% Christian Stelzer$^*$}

% \address{Institut f\"ur Angewandte Mathematik\\
% Technische Universit\"at Graz \\
% Steyrergasse 30\\
% 8010 Graz \\
% Austria}
% \email{behrndt@tugraz.at, holzmann@math.tugraz.at, christian.stelzer@tugraz.at}

% 
% 
% \maketitle
\begin{abstract}
 	In this paper the approximation of Dirac operators with general $\delta$-shell potentials supported on 
	 	$C^2$-curves in $\R^2$ or $C^2$-surfaces in $\R^3$, which may be bounded or unbounded, is studied. It is shown under suitable conditions on the weight of the $\delta$-interaction that a family of Dirac operators with regular, squeezed potentials converges in the norm resolvent sense to the Dirac operator with the $\delta$-shell interaction.
\end{abstract}

% \tableofcontents
%------------------------------------------------------------------------

%%%%%%%%%%%%%%%%%%%%%%%%%%%%%%%%%%%%%%%%%%%%%%%%%%%%%%%%%%%%%%%%%%%%%%%%%%
%%%%%%%%%%%%%%%%%%%%%%%%%%%%%%%%%%%%%%%%%%%%%%%%%%%%%%%%%%%%%%%%%%%%%%%%%%

%\cs{$f,g,h$: Funktionen in Bochner Räumen; $u,v,w$: Funktionen auf $\Omega_\pm$ und $\R$; $\varphi,\psi$: Funktionen am Rand}

\section{Introduction}\label{sec_intro}

In mathematical physics, singular potentials supported on a set $\Sigma$ of measure zero are often used as replacements of potentials that have large values in a vicinity of $\Sigma$ and small values elsewhere, assuming that such idealized models have similar (spectral) properties as the original ones. Nonrelativistic Schr\"odinger operators with $\delta$-potentials supported on points were already considered in the early days of quantum mechanics \cite{dKP31, T35}, see also the monograph \cite{AGHH05}, and then later also Schr\"odinger operators with $\delta$-potentials supported on curves in $\mathbb{R}^2$ or $\mathbb{R}^3$, surfaces in $\mathbb{R}^3$, and similar structures in higher dimensions were investigated, see, e.g., \cite{BLL13, BEKS94, E08} and the references therein. Concerning relativistic Dirac operators with a singular $\delta$-potential, the one-dimensional case was first studied in the 1980s \cite{AGHH05, GS87},  and further investigated in \cite{BD94,BMP17,CMP13,H99,PR14}. Recently also two and three-dimensional Dirac operators coupled with $\delta$-shell potentials supported on general curves in $\mathbb{R}^2$ and surfaces in~$\mathbb{R}^3$ were treated; cf. \cite{AMV14, AMV15,BHOP20, BHSS22,CLMT21} and the discussion below.

In all of the above situations, it is necessary to justify that the differential operators with the singular interactions have indeed similar properties as the original models with regular potentials. One way to do this is to prove suitable approximation results: One constructs a family of potentials that converge to the idealized singular interaction and shows that the associated differential operators also converge in the norm resolvent sense to the idealized operator -- then also the spectral properties of the approximating and the idealized models are approximately the same \cite{kato, RS72}. For Schr\"odinger operators with singular potentials this approximation problem has been solved under the assumption that the interaction support is a sufficiently smooth hypersurface in $\mathbb{R}^\theta$, $\theta \geq 1$; cf. \cite{BEHL17} and the references therein.

The literature on the approximation of Dirac operators with singular potentials is less complete than for their nonrelativistic counterparts. Choose units such that the speed of light $c$ and the reduced Planck constant $\hbar$ are both equal to one. Then, the Dirac operator with a singular $\delta$-potential in $\mathbb{R}^\theta$, $\theta \in \{ 1, 2, 3 \}$, is formally given by
\begin{equation}\label{dirac111}
-i \sum_{j=1}^\theta \alpha_j \partial_j + m \beta + \widetilde{V} \delta_\Sigma, 
\end{equation}
where $\delta_\Sigma$ denotes the $\delta$-distribution supported on a point in $\mathbb{R}$ or the boundary $\Sigma$ of a bounded or unbounded domain in $\R^2$ or $\R^3$,
$\widetilde V$ is a matrix-valued function on $\Sigma$ modeling the position-dependent strength of the singular interaction, $m\in\R$,
and $\alpha_1, \dots ,\alpha_\theta,$ and $\beta$ denote the Dirac matrices defined in~\eqref{def_Dirac_matrices_2d}--\eqref{def_Dirac_matrices_3d} below (in dimension $\theta=1$ one can use the same choice as for $\theta = 2$).  
% The spectral analysis of Dirac operators coupled with singular $\delta$-shell potentials has attracted a lot of attention in the recent past; we refer the reader to \cite{AMV14, BEHL19, BHOP20, Ben21, BD94, BMP18,CMP13, CLMT21,GS87, H99, OV16} and the references therein. Such type of operators are formally given by 
% \begin{equation}\label{dirac111}
% -i(\alpha \cdot \nabla) + m \beta + \widetilde{V} \delta_\Sigma, 
% \end{equation}
% where $\delta_\Sigma$ denotes the $\delta$-distribution supported on the boundary $\Sigma$ of a bounded or unbounded domain in $\R^2$ or $\R^3$,
% $\widetilde V$ is a matrix-valued function on $\Sigma$ modeling the position-dependent strength of the singular interaction, $m\in\R$,
% % and $\alpha,\beta$ denote the Dirac matrices. The self-adjoint operator $H_{\widetilde V}$ in \eqref{def_H_Vtilde} associated with \eqref{dirac111} in 
% $L^2$ is regarded as an idealization of a Dirac operator coupled with a strongly localized potential. However, in 
% order to justify this viewpoint and the applicability of the model it is necessary to prove corresponding approximation results. In fact,
% the natural goal is to show that the singular Dirac operator $H_{\widetilde V}$ can be
% approximated in the norm resolvent sense by regular Dirac operators $H_\varepsilon=
% % -i(\alpha \cdot \nabla) + m \beta + V_\varepsilon$ with strongly localized potentials $V_\varepsilon$ converging (up to a certain rescaling due to  
% Klein's paradox) to $\widetilde V$ for $\varepsilon\rightarrow 0$.
The first approximation result in dimension one goes back to \v{S}eba \cite{S89}, where the norm resolvent convergence for so-called electrostatic (i.e. $\widetilde{V} = \widetilde{\eta} I_2$) and Lorentz scalar (i.e. $\widetilde{V} = \widetilde{\tau} \beta$) interactions supported on a point was proved. 
In \cite{S89} it already turned out that one has to renormalize the interaction strength in order to obtain convergence, i.e. when $v_\varepsilon \rightarrow \eta \delta_\Sigma$ in the distributional sense, then $-i \alpha_1 \partial_x + m \beta + v_\varepsilon I_2$ does not converge to $-i \alpha_1 \partial_x + m \beta + \eta I_2 \delta_\Sigma$, but to a similar operator, where the coefficient of $\delta_\Sigma$ depends in a nonlinear way on $\eta$; for Lorentz scalar potentials a similar effect appears as well. \v{S}eba suggested that this renormalization is related to Klein's paradox for the Dirac equation. 
Later Hughes proved in \cite{H97,H99} strong resolvent convergence for general point interactions and also determined the exact form of the renormalization
in the one-dimensional situation. Recently, Tu\v{s}ek proved norm resolvent convergence in the one-dimensional case for Dirac operators with general point interactions in \cite{T20}.
In the two and three-dimensional setting the approximation problem is not as well-studied as in dimension one and so far 
there exist only results on strong resolvent convergence. In dimension three the approximation of Dirac operators with purely electrostatic or Lorentz scalar $\delta$-shell potentials that fulfil a certain smallness condition supported on boundaries of bounded domains was investigated in \cite{MP17,MP18}. The approximation of two-dimensional Dirac operators with electrostatic, Lorentz scalar, and anomalous magnetic $\delta$-shell potentials supported on closed bounded curves and on a straight line was considered in \cite{CLMT21} and \cite{BHT22b}, respectively. We note that also in the two and three-dimensional setting a renormalization of the interaction strength was observed in \cite{BHT22b, CLMT21,MP18}.

In this paper we study the approximation problem for two and three-dimensional Dirac operators with $\delta$-shell potentials and improve the present
state of art in the following three ways: Under suitable conditions we show (i) norm resolvent convergence on (ii) bounded and unbounded supports $\Sigma$ with (iii) general position-dependent interaction strengths $\widetilde V$. More precisely:
\begin{itemize}
 \item [{\rm (i)}] Instead of strong resolvent convergence we prove the norm resolvent convergence of the approximating family, which has not been 
 established in the multidimensional situation so far. This type of convergence ensures that the spectrum of the limit operator $H_{\widetilde V}$ can be completely  characterized by the spectra of the approximating operators and it also implies the convergence of the related spectral projections.
 \item [{\rm (ii)}] Instead of bounded curves in $\R^2$, the straight line in $\R^2$, or bounded surfaces in $\R^3$, we treat a general class 
 of bounded and unbounded interaction supports $\Sigma$ which can be described by finitely many rotated graphs of $C^2$-functions with bounded
 derivatives
 (see Hypothesis~\ref{hypothesis_Sigma}) and which includes, in particular, graphs of $C^2$-functions with bounded derivatives and boundaries of bounded  $C^2$-domains.
 \item [{\rm (iii)}] Instead of considering only electrostatic, Lorentz scalar, and anomalous magnetic interactions (which can be described by three real-valued functions) we allow general symmetric $2\times2$ or $4 \times 4$ matrix-valued functions as interaction strengths in dimensions two or three, respectively, and  provide an explicit formula for the nonlinear renormalization when passing to the limit.
\end{itemize}

In the following we explain the approximation procedure and our main result in more detail. For this some notation needs to be fixed. 
We denote the space dimension by $\theta \in \{ 2, 3\}$ and set $N = 2$ for $\theta=2$ and $N = 4$ for $\theta = 3$. Let $\Sigma \subset \mathbb{R}^\theta$ be the boundary of a bounded or unbounded $C^2$-domain $\Omega_+ \subset \R^{\theta}$ that satisfies 
Hypothesis~\ref{hypothesis_Sigma}. The curve or surface $\Sigma$ splits $\mathbb{R}^\theta$ into two disjoint parts $\Omega_+$ and $\Omega_- := \R^{\theta} \setminus \overline{\Omega_+}$, and $\nu$ is the unit normal vector field at $\Sigma$ pointing outwards of $\Omega_+$. 
For  $u: \mathbb{R}^\theta \rightarrow \mathbb{C}^N$ we write $u_\pm := u \upharpoonright \Omega_\pm$ and we denote the Dirichlet trace operator by $\tr^\pm: H^1(\Omega_\pm; \mathbb{C}^N) \rightarrow H^{1/2}(\Sigma; \mathbb{C}^N)$, where $H^s$ are the $L^2$-based Sobolev spaces. Moreover, $\alpha_1, \dots, \alpha_\theta, \beta \in \mathbb{C}^{N \times N}$ are the Dirac matrices defined in~\eqref{def_Dirac_matrices_2d}--\eqref{def_Dirac_matrices_3d} and we make use of the notations
\begin{equation*}
  \alpha \cdot \nabla := \sum_{j = 1}^\theta \alpha_j \partial_j \quad \text{and} \quad \alpha \cdot x := \sum_{j = 1}^\theta \alpha_j x_j, \quad x = (x_1, \dots, x_\theta) \in \C^\theta.
\end{equation*}
We introduce for $m \in \R$ and $\widetilde{V} \in L^\infty(\Sigma;\C^{N\times N})$  such that $\widetilde{V}(x_\Sigma) = (\widetilde{V}(x_\Sigma))^*$ for $\sigma$-a.e. $x_\Sigma \in \Sigma$ in $L^2(\mathbb{R}^\theta; \mathbb{C}^N)$ the operator
\begin{equation} \label{def_H_Vtilde}
	\begin{split}
        H_{\widetilde{V}}u & := (-i(\alpha \cdot \nabla) + m \beta) u_+ \oplus (-i(\alpha \cdot \nabla ) + m \beta) u_-, \\
		\dom H_{\widetilde{V}}&:= \biggl\{ u \in H^1 (\Omega_+;\C^N) \oplus H^1(\Omega_-; \mathbb{C}^N): \\
		&\qquad \qquad i(\alpha \cdot \nu )(\tr^+ u_+  - \tr^- u_-) +  \frac{\widetilde{V}}{2}(\tr^+ u_+  + \tr^- u_-) =0 \biggr\}.
	\end{split}
\end{equation}
This  is a rigorous definition of  the formal operator \eqref{dirac111}, whose self-adjointness and spectrum  was studied under various assumptions on the coefficients $\widetilde{V}$ and the interaction support $\Sigma$ in \cite{AMV14, AMV15,BEHL18, BEHL19, BHOP20, BHSS22, Ben21,  CLMT21} by means of boundary triples and potential operators. Furthermore, in \cite{R21, R22a}  a sufficient condition for the self-adjointness of $H_{\widetilde{V}}$ was derived for a wide class of bounded and unbounded $C^2$-smooth interaction supports and bounded $C^1$-smooth interaction matrices. Under the general assumptions that we are going to make in this paper even the self-adjointness of $H_{\widetilde{V}}$ is not yet established, but
will follow as a byproduct of our main result below as $H_{\widetilde{V}}$ is the norm resolvent limit of a sequence of self-adjoint operators.

It is the main goal of the present paper to show that the operator $H_{\widetilde{V}}$ in~\eqref{def_H_Vtilde} can be approximated in the norm resolvent sense by Dirac operators with strongly localized potentials. To introduce the latter operators, define the map 
\begin{equation}\label{eq_def_iota}
	\begin{split}
	\iota: \Sigma \times \R \to \R^{\theta}, \qquad \iota (x_\Sigma,t) := x_\Sigma + t\nu(x_\Sigma),
	\end{split}
\end{equation}
and for $\varepsilon > 0$ the set
\begin{equation} \label{def_Omega_eps}
 \Omega_\varepsilon:= \iota(\Sigma \times (-\varepsilon,\varepsilon)), \quad \varepsilon>0.
\end{equation} 
We call $\Omega_\varepsilon$ \textit{tubular neighbourhood of $\Sigma$}. Hypothesis~\ref{hypothesis_Sigma} and Proposition \ref{prop_tubular_neighbourhood} imply that there exists $\varepsilon_1 > 0$ such that $\iota\vert_{\mathbb{R} \times (-\varepsilon_1, \varepsilon_1)} $ is injective and hence, all points in  $\Omega_{\varepsilon_1}$ can be uniquely described by the map $\iota$. 
To define the above mentioned strongly localized potentials choose 
\begin{equation}\label{eq_q}
	q \in L^{\infty}((-1,1);\R) \quad \text{with} \quad  \int_{-1}^{1} q(s)  \, ds = 1
\end{equation}
and
\begin{equation}\label{eq_V}
V \in W^1_\infty(\Sigma;\C^{N \times N}) \quad  \text{such that} \quad  V(x_\Sigma) = (V(x_\Sigma))^* \text{ for } \sigma \text{-a.e. } x_\Sigma \in \Sigma,
\end{equation} 
where $W^1_\infty$ denotes the first order $L^\infty$-based Sobolev space and $\sigma$ is the Hausdorff measure on $\Sigma$.
Since $\iota\vert_{\mathbb{R} \times (-\varepsilon_1, \varepsilon_1)}$ is injective, we can define for $\varepsilon \in (0, \varepsilon_1)$
\begin{equation}\label{eq_short_range_potential}  
	V_\varepsilon(x) := 
	\begin{cases}
		\frac{1}{\varepsilon}V(x_\Sigma)q\bigl(\frac{t}{\varepsilon}\bigr),& \text{if } x = \iota(x_\Sigma,t) \in \Omega_\varepsilon,\\
		0,& \text{if } x \notin \Omega_\varepsilon,
	\end{cases}
\end{equation} 
and for $m \in \R$ and $\varepsilon \in (0, \varepsilon_1)$ the operator
\begin{equation}\label{eq_H_eps}
	\begin{split}
		H_\varepsilon u &:= -i(\alpha \cdot \nabla) u + m \beta u + V_\varepsilon  u, \quad \dom H_\varepsilon :=  H^1(\R^\theta;\C^N).
	\end{split}
\end{equation}
Note that $H_\varepsilon$ is self-adjoint in $L^2(\mathbb{R}^\theta; \mathbb{C}^N)$, as the function  $V_\varepsilon \in L^\infty(\mathbb{R}^\theta; \mathbb{C}^{N \times N})$ is symmetric and the unperturbed Dirac operator $H = H_\varepsilon - V_\varepsilon$ defined on $H^1(\R^\theta;\C^N)$ is self-adjoint in
$L^2(\mathbb{R}^\theta; \mathbb{C}^N)$; cf. Section~\ref{sec_free_Dirac}.

The sequence $V_\varepsilon$ converges to $V \delta_\Sigma $ in the sense of distributions. However, as in \cite{BHT22b,CLMT21,H97,H99,MP18,S89,T20} the sequence $H_\varepsilon$ does not converge to $H_V$, but to a similar operator with a renormalized interaction strength $\widetilde{V}$. In the present situation this renormalization is expressed with the help of the matrix-valued function
\begin{equation}\label{eq_scaling_matrix}
S = \textup{sinc}\bigl(\tfrac{1}{2}(\alpha \cdot \nu)V\bigr)\cos\bigl(\tfrac{1}{2}(\alpha \cdot \nu)V\bigr)^{-1},
\end{equation}
where analytic functions of matrices are defined via the corresponding power series, or equivalently via the Riesz-Dunford functional calculus. 
Now we are prepared to formulate the main result of this paper.

\begin{theorem}\label{THEO_MAIN}
	Let  $q$, $V$, and $\varepsilon_2>0$ be as in \eqref{eq_q}, \eqref{eq_V}, and \eqref{eq_eps_2},  and assume that for some $z \in \mathbb{C} \setminus ((-\infty, -|m|] \cup [|m|, \infty))$ the condition
	\begin{equation}\label{cond123}
	\norm{V}_{W^1_\infty(\Sigma;\C^{N\times N})}   \norm{q}_{L^\infty(\R;\R)} < X_z
	\end{equation}
	holds with $X_z$ given by~\eqref{def_X_z}. Furthermore, assume that
	\begin{equation}\label{cond234}
	\cos\bigl(\tfrac{1}{2}(\alpha \cdot \nu)V\bigr)^{-1} \in W^1_\infty(\Sigma;\C^{N\times N})
	\end{equation}
	and set $\widetilde V =  VS$, where $S$ is defined as in~\eqref{eq_scaling_matrix}. Then, the operator 
	$H_{\widetilde{V}}$ in \eqref{def_H_Vtilde} is self-adjoint in $L^2(\mathbb{R}^\theta; \mathbb{C}^N)$, 
	$z \in \rho(H_\varepsilon) \cap \rho(H_{\widetilde{V}})$ for all $\varepsilon \in (0,\varepsilon_2)$, and
	for any $r \in (0,\tfrac{1}{2})$ there exists $C>0$ such that 
	\begin{equation*}
		\norm{(H_\varepsilon-z)^{-1} - (H_{\widetilde{V}}-z)^{-1}}_{L^2(\R^\theta;\C^N) \to L^2(\R^\theta;\C^N)} \leq C \varepsilon^{1/2-r}, \qquad  \varepsilon \in (0, \varepsilon_2).
	\end{equation*} 
	In particular, $H_{\varepsilon}$ converges to $H_{\widetilde{V}}$ in the norm resolvent sense as $\varepsilon \to 0+$. 
\end{theorem}

The main assumption in Theorem~\ref{THEO_MAIN} is condition \eqref{cond123}, which means that $V$ has to be small in a suitable sense and which also appears in a similar form in \cite{MP18}, where the strong resolvent convergence of $H_\varepsilon$ for some special potentials $V$ was proved. The condition for $V$ can be optimized by choosing $q = \tfrac{1}{2} \chi_{(-1,1)}$, where $\chi_{(-1,1)}$ denotes the characteristic function of the interval $(-1,1)$, but since for some applications also more general $q$ can be useful, we keep the more general formulation of ~\eqref{cond123}; a more detailed discussion of \eqref{cond123} 
can be found in Remark~\ref{remark_kleinheit}. We shall show later in Lemma~\ref{lemma_smallness_condition} that 
$X_z \leq \frac{\pi}{4}$, if $\Sigma$ is compact and smooth. Let us also mention that one can replace \eqref{cond123} by the more abstract conditions (i) and (ii) from Proposition~\ref{prop_I+B_conv}; cf. Remark~\ref{remark_smallnes_assumption}.

In the following, let us consider the special choice
\begin{equation}\label{specialv}
V = \eta I_N + \tau \beta + \lambda i(\alpha\cdot\nu)\beta,
\end{equation}
where $\eta,\tau,\lambda \in W^1_\infty(\Sigma;\R)$ are real-valued functions (or simply real constants) that model electrostatic, Lorentz scalar, and anomalous magnetic interactions, respectively. Such potentials $V$ are of special physical interest and the existing literature on Dirac operators with $\delta$-shell potentials mostly focuses on this case.
It turns out that for such $V$ 
the renormalization matrix in \eqref{eq_scaling_matrix} has the simple form 
$$S = \frac{2}{\sqrt{d}}\tan\Bigl(\tfrac{\sqrt{d}}{2}\Bigr)I_N,$$ 
where $d = \eta^2-\tau^2-\lambda^2$, and \eqref{cond234} simplifies, see \eqref{eq_scaling_condition} below. Note that analogous conditions and renormalizations  are given in \cite[Theorem 2.6]{CLMT21}, \cite[Theorem IV.2]{MP17} and \cite[Remarks]{S89}. 
In the present situation Theorem~\ref{THEO_MAIN} reduces to the following corollary.

\begin{corollary}\label{cor_main}
	Let  $q$, $V$, and $\varepsilon_2>0$ be as in \eqref{eq_q}, \eqref{specialv}, and  \eqref{eq_eps_2}, respectively, with $\eta,\tau,\lambda \in W^1_\infty(\Sigma; \mathbb{R})$, $d = \eta^2-\tau^2-\lambda^2$, and assume that for some $z \in \mathbb{C} \setminus ((-\infty, -|m|] \cup [|m|, \infty))$ condition \eqref{cond123}
	holds with $X_z$ given by~\eqref{def_X_z}. Furthermore, assume 
	\begin{equation}\label{eq_scaling_condition}
		\inf_{x_\Sigma \in \Sigma, k \in \N_0} \lvert(2k+1)^2\pi^2- d(x_\Sigma)\rvert>0
	\end{equation} 
	and let $\widetilde{V} =  \frac{2}{\sqrt{d}}\tan(\sqrt{d}/2)V$. Then, 
	$H_{\widetilde{V}}$ in \eqref{def_H_Vtilde} is self-adjoint in $L^2(\mathbb{R}^\theta; \mathbb{C}^N)$, 
	$z \in \rho(H_\varepsilon) \cap \rho(H_{\widetilde{V}})$ for all $\varepsilon \in (0,\varepsilon_2)$, and
	for any $r \in (0,\tfrac{1}{2})$ there exists $C>0$ such that 
	\begin{equation*}
		\norm{(H_\varepsilon-z)^{-1} - (H_{\widetilde{V}}-z)^{-1}}_{L^2(\R^\theta;\C^N) \to L^2(\R^\theta;\C^N)} \leq C \varepsilon^{1/2-r}, \qquad  \varepsilon \in (0, \varepsilon_2).
	\end{equation*} 
	In particular, $H_{\varepsilon}$ converges to $H_{\widetilde{V}}$ in the norm resolvent sense as $\varepsilon \to 0+$. 
\end{corollary}

In the previous two results we stated conditions on a given $V$ ensuring that $H_\varepsilon$ converges in the norm resolvent sense to $H_{\widetilde{V}}$. In the following corollary, we address the inverse question and state conditions on $\widetilde{V}$ such that $H_{\widetilde{V}}$ is the limit of the operators $H_\varepsilon$.

\begin{corollary}\label{cor_main_inv}
	Let $\widetilde{V} \in W_\infty^1(\Sigma;\C^{N \times N})$ such that $\widetilde{V}(x_\Sigma) = (\widetilde{V}(x_\Sigma))^*$ for  $\sigma$-a.e. $x_\Sigma \in \Sigma$, $\varepsilon_2>0$ be as in \eqref{eq_eps_2}, and	assume that for some given $z \in \mathbb{C} \setminus ((-\infty, -|m|] \cup [|m|, \infty))$
	\begin{equation}\label{cond567}
		\|(\alpha \cdot \nu) \widetilde{V}\|_{W^1_\infty(\Sigma;\C^{N \times N})} < 2\tanh\biggl( \frac{X_z}{\norm{\alpha \cdot \nu}_{W_\infty^1(\Sigma;C^{N \times N})}}\biggr)
	\end{equation}
	holds.
	Moreover, set $q = \tfrac{1}{2} \chi_{(-1,1)}$ and
	\begin{equation*}
		V = 2 (\alpha \cdot \nu )\arctan\bigl(\tfrac{1}{2}(\alpha \cdot \nu)\widetilde{V}\bigr) = 2(\alpha \cdot \nu ) \sum_{n = 0}^{\infty}  (-1)^n\frac{((\alpha \cdot \nu) \widetilde{V})^{2n +1}}{2^{2n+1} (2n+1)},
	\end{equation*}
	and let $H_\varepsilon$ be the operator defined in~\eqref{eq_H_eps} with these specific  choices of $q$ and $V$.
	Then, the operator 
	$H_{\widetilde{V}}$ in \eqref{def_H_Vtilde} is self-adjoint in $L^2(\mathbb{R}^\theta; \mathbb{C}^N)$, 
	$z \in \rho(H_\varepsilon) \cap \rho(H_{\widetilde{V}})$ for all $\varepsilon \in (0,\varepsilon_2)$, and
	for any $r \in (0,\tfrac{1}{2})$ there exists $C>0$ such that 
	\begin{equation*}
		\norm{(H_\varepsilon-z)^{-1} - (H_{\widetilde{V}}-z)^{-1}}_{L^2(\R^\theta;\C^N) \to L^2(\R^\theta;\C^N)} \leq C \varepsilon^{1/2-r}, \qquad  \varepsilon \in (0, \varepsilon_2).
	\end{equation*} 
	In particular, $H_{\varepsilon}$ converges to $H_{\widetilde{V}}$ in the norm resolvent sense as $\varepsilon \to 0+$.
\end{corollary}

    Let us briefly explain the strategy to prove Theorem~\ref{THEO_MAIN} and the structure of this paper. 
First, in Section~\ref{sec_preliminaries} we collect some preliminary results regarding the geometry of $\Sigma$, the free Dirac operator, and some associated potential and boundary integral operators (see also Appendix~\ref{appendix_Phi_C}).    
    In 
    Proposition~\ref{prop_resolvent_formula} we then prove the resolvent formula  
\begin{equation*}
  (H_\varepsilon - z)^{-1} = (H-z)^{-1} -  A_\varepsilon( z) Vq(I + B_\varepsilon( z) Vq)^{-1}C_\varepsilon( z),
\end{equation*}
where $H$ is the free Dirac operator (see \eqref{def_free_op}) and $A_\varepsilon(z), B_\varepsilon(z), C_\varepsilon(z)$ are suitable integral operators defined in~\eqref{def_ABC_operators} that are bounded between various $L^2$-spaces; cf. Proposition~\ref{prop_ABC_operators}. 
This is a standard approach that has been applied in other 
approximation problems for differential operators with singular potentials, see, e.g., \cite{AGHH05, BEHL17, MP18, S89}. 
Typically, the convergence of $A_\varepsilon(z), B_\varepsilon(z), C_\varepsilon(z)$ is analyzed in the corresponding $L^2$-setting. In the situation of Dirac operators with singular potentials, $B_\varepsilon(z)$ converges only strongly in $L^2$ and hence, via this approach only the strong resolvent convergence of $H_\varepsilon$ can be shown; cf. \cite{MP18}. Contrary to that, we investigate the convergence of $A_\varepsilon(z), B_\varepsilon(z), C_\varepsilon(z)$ between Sobolev type spaces. For this, a certain 
shift operator, which allows to compare the values of functions in $x_\Sigma \in \Sigma$ and in $x_\Sigma + \varepsilon s \nu(x_\Sigma)$, 
plays a crucial role; cf. Proposition~\ref{prop_shift_operator}. 
The condition \eqref{cond123} in Theorem~\ref{THEO_MAIN} is needed to ensure that $(I + B_\varepsilon( z) Vq)^{-1}$ is uniformly bounded in $\varepsilon$.
Finally, using the limiting behavior of $A_\varepsilon(z), B_\varepsilon(z), C_\varepsilon(z)$ we compute in Section~\ref{sec_main} the limit of $(H_\varepsilon - z)^{-1}$ for $\varepsilon \to 0+$ and show that it indeed coincides with the resolvent of $H_{\widetilde{V}}$.
In Appendices~\ref{sec_appendix_iota} and~\ref{sec_appendix_diff_B} we provide some technical calculations that are necessary for the definition of the tubular neighbourhood of $\Sigma$ and  the convergence analysis.

\subsection*{Notations}\label{sec_notations}

By $\theta \in \{ 2, 3 \}$ we denote the space dimension and we set $N=2$ for $\theta=2$ and $N=4$ for $\theta = 3$. Recall that the Pauli spin matrices are 
	\begin{equation*}
		\sigma_1 = \begin{pmatrix} 0 & 1\\ 1 & 0 \end{pmatrix}, \quad\sigma_2 = \begin{pmatrix} 0 & -i\\ i & 0 \end{pmatrix}, \quad \text{and} \quad \sigma_3 =\begin{pmatrix} 1 & 0\\ 0 & -1 \end{pmatrix}.
	\end{equation*} 
	With their help the Dirac matrices $\alpha_1, \dots, \alpha_\theta, \beta \in \mathbb{C}^{N \times N}$ are defined for $\theta=2$ by
	\begin{equation} \label{def_Dirac_matrices_2d}
		\alpha_1 := \sigma_1, \quad \alpha_2 := \sigma_2, \quad \text{and} \quad \beta := \sigma_3,
	\end{equation}
	and  for $\theta = 3$ by
	\begin{equation} \label{def_Dirac_matrices_3d}
		\alpha_j := \begin{pmatrix} 0 &
			\sigma_j \\ \sigma_j & 0
		\end{pmatrix}\text{ for } j=1,2,3  \quad \text{and} \quad \beta := \begin{pmatrix}
			I_2 & 0 \\ 0 & -I_2
		\end{pmatrix},
	\end{equation}
	where $I_2$ is the $2 \times 2$-identity matrix.
	We will often make use of the notations
	\begin{equation*}
		\alpha \cdot \nabla := \sum_{j = 1}^\theta \alpha_j \partial_j \quad \text{and} \quad \alpha \cdot x := \sum_{j = 1}^\theta \alpha_j x_j, \quad x = (x_1, \dots, x_\theta) \in \C^\theta.
	\end{equation*}
	The letter $C>0$ always denotes a constant which may change in-between lines and does not depend on the space variables or on $\varepsilon$. 
The symbol $| \cdot|$ is used for the absolute value, the Euclidean vector norm, or the Frobenius norm of a number, a vector, or a matrix, respectively.
We write $\spf{\cdot}{\cdot}$ for the Euclidean scalar product  in $\C^n$, $n \in \N$, which is anti-linear in the second argument. Eventually, the application of an analytic function to a matrix $A$ is defined via the associated power series, whenever it converges. This implies, in particular, for two holomorphic functions $f, g$ that $f(A) g(A) = (f g)(A)$. We will also make use of the equivalent definition of $f(A)$ via the Riesz-Dunford calculus 
\begin{equation} \label{Riesz_Dunford}
  f(A) = - \frac{1}{2 \pi i} \int_{\partial \Omega} f(z) (A-z)^{-1} d z,
\end{equation}
where $\Omega \subset \mathbb{C}$ such that all eigenvalues of $A$ belong to $\Omega$, $f$ is holomorphic in a neighborhood of $\Omega$, and the integral is understood as a complex line integral; cf. \cite[Chapter~VII.~\S4]{C90} and \cite[Chapter~VIII.3.1]{DS58}.

Next, let $\mathcal{H}, \mathcal{G}$ be Hilbert spaces and let $A$ be a linear operator from $\mathcal{H}$ to $\mathcal{G}$. The domain, kernel, and range of $A$ are denoted by $\dom A$, $\ker A$, and $\ran A$, respectively. If $A$ is bounded and everywhere defined, then we write $\| A \|_{\mathcal{H} \rightarrow \mathcal{G}}$ for its operator norm. If $\mathcal{H} = \mathcal{G}$ and $A$ is closed, then the resolvent set and the spectrum $A$ are denoted by $\rho(A)$ and $\sigma(A)$, respectively.

Following \cite[Appendix~B]{M00}, we call  $(\mathcal{H}_0,\mathcal{H}_1)$ a compatible pair, if $\mathcal{H}_0$ and $\mathcal{H}_1$ are two Hilbert spaces which are continuously embedded in a bigger Hausdorff topological vector space. In this situation, one can construct with the K-method (or various other methods, see \cite{CHM15,CHM22,HNVW16,M00}, which yield the same spaces with equivalent norms) a family of Hilbert spaces
$[\mathcal{H}_0,\mathcal{H}_1]_{\tau}$, $\tau \in (0,1)$, such that $\mathcal{H}_0 \cap \mathcal{H}_1	\subset [\mathcal{H}_0,\mathcal{H}_1]_{\tau} \subset \mathcal{H}_0 + \mathcal{H}_1$ for all $\tau \in (0,1)$. Assume that $(\mathcal{G}_0,\mathcal{G}_1)$ is another compatible pair of Hilbert spaces. Recall that for two bounded operators $A_0: \mathcal{H}_0 \to \mathcal{G}_0$ and $A_1: \mathcal{H}_1 \to \mathcal{G}_1$ such that 
$A_0 u = A_1 u$ for all $u \in \mathcal{H}_0 \cap \mathcal{H}_1$, there exists by \cite[Theorem B.2]{M00} a unique bounded linear operator $A_\tau: [\mathcal{H}_0,\mathcal{H}_1]_{\tau} \to [\mathcal{G}_0,\mathcal{G}_1]_{\tau}$ such that $	A_0 u = A_1 u = A_\tau u $ for all  $u \in  \mathcal{H}_0 \cap \mathcal{H}_1$  and 
\begin{equation}\label{eq_interpolation_inequality}
	\norm{A_\tau}_{[\mathcal{H}_0,\mathcal{H}_1]_{\tau} \to [\mathcal{G}_0,\mathcal{G}_1]_{\tau}} \leq \norm{A_0}_{\mathcal{H}_0 \to \mathcal{G}_0}^{1-\tau} \norm{A_1}_{\mathcal{H}_1 \to \mathcal{G}_1}^{\tau}.
\end{equation}

Finally, if $k,n \in \N$, $U \subset \R^{n}$ is an open set and $\mathcal{V}$ denotes the space of real or complex scalars, vectors, or matrices, then we write $C_b^k(U;\mathcal{V})$ for the space of all $f \in C^k(U;\mathcal{V})$ such that $f$ and all partial derivatives of $f$ up to order $k$ are bounded. Moreover, $\mathcal{D}(U;\mathcal{V})$ is the set of all $f \in  C^\infty(U;\mathcal{V})$ with  compact support. If $f \in C^1(U; \mathbb{C}^k)$, then the symbol $D f$ is used for the Jacobi matrix of $f$.
The usual Bochner Lebesgue space of $\mathcal{H}$-valued
functions are denoted by $L^2((-1,1);\mathcal{H})$, cf. Section \ref{sec_Bochner}. For $\mathcal{H}=H^r(\Sigma;\C^N)$ we use the symbol
$\norm{\cdot}_r$ for the norm in $L^2((-1,1);H^r(\Sigma;\C^N))$. In a similar way, we write
$\norm{\cdot}_{r \to r'} :=\norm{\cdot}_{L^2((-1,1);H^{r}(\Sigma;\C^N)) \to L^2((-1,1);H^{r'}(\Sigma;\C^N))}$,	$\norm{\cdot}_{r \to \mathcal{H}} := \norm{\cdot}_{L^2((-1,1);H^r(\Sigma;\C^N)) \to \mathcal{H}}$ as well as $\norm{\cdot}_{\mathcal{H} \to r'} := \norm{\cdot}_{\mathcal{H} \to L^2((-1,1);H^{r'}(\Sigma;\C^N)) }$.

% If $c \in \C\setminus \R_+$, we choose the square root such that $\textup{Im } \sqrt{c} >0$.

% % If a densely defined  continuous function has a continuous extension, then we denote its extension with the same expression.

% The letter $ z$ denotes a fixed complex number with $\textup{Im }  z \neq 0$.

\section{Preliminaries}\label{sec_preliminaries}
In this section we provide some preliminary material that is needed in the main part of this paper to prove Theorem~\ref{THEO_MAIN}. First, in Section~\ref{sec_geometry} we formulate suitable assumptions on the hypersurface $\Sigma$ that are convenient for our approximation procedure, then in  Section~\ref{sec_Bochner} we provide some definitions and fundamental results for Bochner Lebesgue spaces that are used throughout this paper, and in Section~\ref{sec_free_Dirac} we introduce the free Dirac operator and a family of associated potential and boundary integral operators.

\subsection[Description of the boundary ${\Sigma}$ and its tubular neighbourhood]{Description of the boundary $\boldsymbol{\Sigma}$ and its tubular neighbourhood}\label{sec_geometry}

In this section we formulate assumptions on the geometry of $\Sigma$ which will allow us to construct the tubular neighbourhood of $\Sigma$ in~\eqref{def_Omega_eps} in a similar manner as in \cite{BEHL17, MP18}. The presentation here differs slightly from \cite{BEHL17}, as in the present paper we need better properties of the trace map and the extension operator. The following assumption will be made throughout the paper:

\begin{hypothesis} \label{hypothesis_Sigma}
  Assume that $\Omega_+ \subset \mathbb{R}^\theta$, $\theta \in \{ 2, 3 \}$, is an open set with boundary $\Sigma := \partial \Omega_+$ and that there exist open sets $W_1, \dots ,W_p \subset \mathbb{R}^\theta$, mappings $\zeta_1,\dots,\zeta_p \in C^{2}_b(\R^{\theta-1}; \mathbb{R})$, rotation matrices $\kappa_1, \dots ,\kappa_p \in \R^{\theta \times \theta}$, and $\varepsilon_0 >0$ such that
\begin{itemize}
	\item[(i)] $\Sigma \subset \bigcup_{l=1}^p W_l$;
 	\item[(ii)] if $x \in \partial \Omega_+ = \Sigma$, then there exists $l \in \{1, \dots ,p\}$ such that $B(x,\varepsilon_0) \subset W_l$;
	\item[(iii)] $W_l \cap \Omega_+
	= W_l \cap \Omega_l$, where $\Omega_l = \{ \kappa_l (x', x_\theta): x_\theta < \zeta(x'), \, (x',x_\theta) \in \R^\theta \}$, for $l \in \{1,\dots ,p\}$.
\end{itemize}
Furthermore, we set $\Sigma_l := \partial \Omega_l = \{ \kappa_l (x', \zeta_l(x')): x' \in \mathbb{R}^{\theta-1} \}$,  $\Omega_- := \R^\theta \setminus \overline{\Omega_+}$, and denote the unit normal vector field at $\Sigma$ that is pointing outwards of $\Omega_+$ by $\nu$. 
\end{hypothesis}

We note that  $C_b^2$-hypographs and the boundaries of compact $C^2$-domains satisfy Hypothesis~\ref{hypothesis_Sigma}.
Moreover, the sets $\Sigma_l$ are a cover of $\Sigma$, but in general one has $\Sigma_l \not\subset \Sigma$.
In the following we fix notations regarding points and functions on $\Sigma$ satisfying Hypothesis~\ref{hypothesis_Sigma} that will be used in the entire paper and recall the definition of Sobolev spaces on $\Sigma$ and a convenient variant of the trace theorem; cf. \cite{BMMM14,M87, M00}.
First, if $U \subset \mathbb{R}^{\theta-1}$ is open, then for $r \in [0,2]$ the space $H^r(U; \mathbb{C}^N)$ and also $W^1_\infty(U; \mathcal{V})$ with $\mathcal{V} \in \{\C,\C^{N \times N}\}$ are defined as in \cite{M00}. To transfer these definitions to $\Sigma$,
we choose a partition of unity $\varphi_1, \dots, \varphi_p \in C^{\infty}(\R^\theta)$ subordinate to $W_1,  \dots, W_p$ such that each $\varphi_l$, $l \in \{ 1, \dots, p\}$, has uniformly bounded derivatives (see Lemma~\ref{lem_part_of_unity}). Next, we define
\begin{equation} \label{def_x_Sigma_l}
  x_{\Sigma_l}: \mathbb{R}^{\theta-1} \rightarrow \Sigma_l , \qquad x_{\Sigma_l}(x') := \kappa_l (x', \zeta_l(x')),
\end{equation}
and for $\psi \in L^2(\Sigma;\C^N)$ we write
\begin{equation*}
	\psi_{\Sigma_l}(x') := 
	\begin{cases}
	\varphi_l(x_{\Sigma_l}(x'))\psi(x_{\Sigma_l}(x')), & \text{if } x_{\Sigma_l}(x')  \in \Sigma, \\
	0, &\text{if } x_{\Sigma_l}(x')  \notin \Sigma.
	\end{cases}
\end{equation*} 
Then,  $ \psi_{\Sigma_l} \in L^2(\R^{\theta-1};\C^N)$ and $\psi(x_\Sigma)=\sum_{l \in \{1,\dots,p\}, x_\Sigma \in \Sigma_l} \psi_{\Sigma_l}(x_{\Sigma_l}^{-1}(x_\Sigma))$ for  $x_\Sigma \in \Sigma $. As usual, we introduce for $r \in [0,2]$
\begin{equation*}
  H^{r}(\Sigma;\C^N) := \big\{ \psi \in L^2(\Sigma; \mathbb{C}^N): \psi_{\Sigma_l} \in H^{r}(\R^{\theta-1};\C^N) \text{ for all } l = 1, \dots, p \big\}
\end{equation*}
and endow this space with the  inner product 
\begin{equation*}
	\spv{\phi}{\psi}_{H^{r}(\Sigma;\C^N)} = \sum_{l=1}^{p} \spv{\phi_{\Sigma_l}}{\psi_{\Sigma_l}}_{H^{r}(\R^{\theta-1};\C^N)}, \qquad \phi, \psi \in H^{r}(\Sigma;\C^N).
	\end{equation*}
Sobolev spaces $H^r(\Sigma; \mathbb{C}^N)$ with $r \in [-2,0]$  are defined, as usual, by duality. One can prove in the same manner as in \cite[Theorem B.11]{M00} that for $r_1, r_2 \in [-2,2]$ and $r = (1-\tau) r_1 + \tau r_2$, $0<\tau<1$, the interpolation property
\begin{equation}\label{eq_interpolation}
	H^{r}(\Sigma;\C^N) = \left[H^{r_1}(\Sigma;\C^N),H^{r_2}(\Sigma;\C^N) \right]_{\tau}
	\end{equation}
holds with equivalent norms. Eventually, we set $U_l := x_{\Sigma_l}^{-1}(\Sigma \cap W_l)$ and define for $\mathcal{V} \in \{\C;\C^{N \times N}\}$
\begin{equation*}
	 W^1_\infty(\Sigma;\mathcal{V}) := \big\{ F \in L^\infty(\Sigma;\mathcal{V}):(F \circ x_{\Sigma_l}) \upharpoonright U_l \in W^1_\infty(U_l;\mathcal{V})\text{ for all } l = 1,\dots,p \big\}
\end{equation*}
equipped with the norm
\begin{equation*}
	\norm{F}_{W^1_\infty(\Sigma;\mathcal{V})} := \max_{l \in \{1,\dots,p\}} \norm{(F \circ x_{\Sigma_l}) \upharpoonright U_l}_{W^1_\infty(U_l;\mathcal{V})}, \quad F \in W^1_\infty(\Sigma;\mathcal{V}).
\end{equation*}
Note that if $F \in W^1_\infty(\Sigma;\mathcal{V})$ and $r \in [0,1]$, then by direct calculation ($r \in \{0,1\}$) and interpolation ($r \in (0,1)$) one obtains for all $\phi \in H^r(\Sigma;\C^N)$
\begin{equation} \label{W_1_infty_mult_op}
	F \phi \in H^r(\Sigma;\C^N) \quad \textup{and} \quad  	\norm{F \phi}_{H^r(\Sigma;\C^N)} \leq C \norm{F}_{W^1_\infty(\Sigma;\mathcal{V})} \norm{\phi}_{H^r(\Sigma;\C^N)},
\end{equation}
 where $C>0$ does not depend on $F$. In particular, $\phi \mapsto F\phi$ induces a bounded operator in $H^r(\Sigma;\C^N)$ and the operator norm is bounded by $C\norm{F}_{W^1_\infty(\Sigma;\mathcal{V})} $. We identify $F$ with this induced operator.

In the next proposition, we state a variant of the trace theorem suitable for our situation; cf. \cite[Theorem~2]{M87} and \cite[Theorem~8.7]{BMMM14}. 

\begin{proposition} \label{proposition_trace_theorem}
	Assume that $\Omega_\pm$ and $\Sigma$ satisfy Hypothesis~\ref{hypothesis_Sigma} and let $r \in (\tfrac{1}{2}, 1]$.
	Then, there exists a unique bounded and surjective operator $$\tr^\pm: H^r(\Omega_\pm; \mathbb{C}^N) \rightarrow H^{r-1/2}(\Sigma; \mathbb{C}^N)$$ such that $\tr^\pm u = u|_\Sigma$ for all $u \in H^r(\Omega_\pm; \mathbb{C}^N) \cap C(\overline{\Omega_\pm}; \mathbb{C}^N)$. Furthermore, there
	exists a unique bounded and surjective operator $$\tr: H^r(\mathbb{R}^\theta; \mathbb{C}^N) \rightarrow H^{r-1/2}(\Sigma; \mathbb{C}^N)$$ such that $\tr u = u|_\Sigma$ for all $u \in H^r(\mathbb{R}^\theta; \mathbb{C}^N) \cap C(\mathbb{R}^\theta; \mathbb{C}^N)$.
\end{proposition} 

In the rest of this subsection we summarize some statements on tubular neighbourhoods of $\Sigma$ given by~\eqref{def_Omega_eps} by making use of the results in \cite[Section~2.1]{BEHL17}. Recall that the embedding $\iota$ is defined by~\eqref{eq_def_iota} and that  the set $\Omega_\varepsilon$ for $\varepsilon > 0$ is introduced in~\eqref{def_Omega_eps}. In the following definition we introduce the Weingarten map (or shape operator) on $\Sigma$:
\begin{definition}\label{def_Weingarten}
	Let $\Omega_\pm, \Sigma$
 satisfy Hypothesis~\ref{hypothesis_Sigma} and denote for $x_\Sigma = x_{\Sigma_l}(x') \in \Sigma $, $l \in \{ 1, \dots, p \}$, the tangent hyperplane of $\Sigma$ in the point $x_\Sigma$ by the symbol $T_{x_\Sigma} := \textup{span} \, \{  \partial_j x_{\Sigma_l}(x'): j = 1, \dots, \theta-1 \}$. The Weingarten map is the linear operator $W(x_\Sigma): T_{x_\Sigma} \to T_{x_\Sigma}$ defined by
	\begin{equation*}
	W(x_\Sigma) \partial_j x_{\Sigma_l}(x') =  - \partial_j \nu(x_{\Sigma_l}(x')), \qquad j \in \{1,\dots, \theta-1\}.
	\end{equation*}
\end{definition}

As mentioned in \cite[Definition 2.2]{BEHL17} and the text below, $W(x_\Sigma)$ is well-defined (i.e. it is independent of the parametrization of $\Sigma$ and the index $l$). Furthermore, in the following we denote the matrix representation of $W(x_\Sigma)$ corresponding to the  basis $\{ \partial_j x_{\Sigma_l}(x'): j=1, \dots, \theta-1 \}$ of $T_{x_\Sigma}$ by $L_l(x')$. 
Moreover, note that if $x_\Sigma = x_{\Sigma_l}(x') = x_{\Sigma_{k}}(y') \in \Sigma$ with $l,k \in \{1,\dots,p\}$ and $x',y' \in \R^{\theta-1}$, then the eigenvalues of $L_l(x')$ and $L_{k}(y')$ coincide, see \cite[Definition 2.2]{BEHL17} and the text below. Therefore, the expression $\det(I-tW(x_\Sigma)) := \det(I_{\theta-1} - tL_l(x'))$ for $t \in \R$ is well-defined.
In the next proposition we state important properties of $\iota$ and $W$, 
and identify $L^1(\Omega_\varepsilon)$ with $L^1(\Sigma \times (-\varepsilon,\varepsilon))$.

\begin{proposition}\label{prop_tubular_neighbourhood} 
	Let $\Omega_\pm,\Sigma \subset \mathbb{R}^\theta$, $\theta \in \{ 2, 3 \}$, satisfy Hypothesis~\ref{hypothesis_Sigma}. Then, there exists $\varepsilon_1 >0$ such that the following is true:
	\begin{itemize}
	\item[$\textup{(i)}$] $\iota \vert_{\Sigma \times (-\varepsilon_1,\varepsilon_1)}$ is injective.
	\item[$\textup{(ii)}$] There exists $C>0$ such that $\abs{1-\det(I-\varepsilon W(x_\Sigma))} \leq C \varepsilon <1/2$ for all $x_\Sigma \in \Sigma$ and $\varepsilon \in (-\varepsilon_1,\varepsilon_1)$.
	\item[$\textup{(iii)}$] For $\varepsilon \in (-\varepsilon_1,\varepsilon_1)$ one has $u \in L^1(\Omega_\varepsilon)$ if and only if  $u \circ \iota \vert_{\Sigma \times (-\varepsilon,\varepsilon)} \in L^1(\Sigma\times(-\varepsilon,\varepsilon))$  and in this case
	\begin{equation*}
		\int_{\Omega_{\varepsilon}} u(y) \,dy = \int_{-\varepsilon}^{\varepsilon} \int_\Sigma u(y_\Sigma +s\nu(y_\Sigma))  \det(I- sW(y_\Sigma))\,d\sigma(y_\Sigma)\, ds.
	\end{equation*}
	\end{itemize}
\end{proposition}
\begin{proof}
	Let $\varepsilon_A$  be the number specified in Proposition~\ref{prop_iota} and let $\varepsilon_1 \in (0, \varepsilon_{A})$. Then, by Proposition~\ref{prop_iota}~(ii)  there exists $C_{A,2}> 0$ such that
	\begin{equation*}
		\abs{\iota(x_\Sigma,t) -  \iota(y_\Sigma,s)} \geq C_{A,2}^{-1}  (\abs{x_\Sigma - y_\Sigma} + \abs{t-s}), \quad  (x_\Sigma,t),(y_\Sigma,s) \in \Sigma \times (-\varepsilon_1,\varepsilon_1).
	\end{equation*} 
	Hence, $\iota \vert_{\Sigma \times(-\varepsilon_1,\varepsilon_1)}$ is injective, i.e. item~(i) is true. 
	
	Next, we show assertion~(ii). For this, we  introduce for $l \in \{1, \dots, p\}$ and $x' \in \mathbb{R}^{\theta-1}$ such that $x_{\Sigma_l}(x')  \in \Sigma$ the coordinate matrices $M_l(x')$ and $H_l(x')$ of the first and second fundamental form, respectively. Their entries are given by 
	\begin{equation*}
			M_l(x')[j,k] := \langle \partial_j x_{\Sigma_l}(x'), \partial_k x_{\Sigma_l}(x')\rangle
			\end{equation*}
			and
			\begin{equation*}
			H_l(x')[j,k] :=  \langle \partial_j x_{\Sigma_l}(x'), -\partial_k \nu(x_{\Sigma_l}(x')) \rangle
	\end{equation*}
	for $j,k \in \{1,\dots,\theta-1\}$.
	It is well-known, see for instance \cite[Chapter 3B]{K02}, that these matrices are related to the coordinate matrix of the Weingarten map at $x_{\Sigma_l}(x') =x_\Sigma \in \Sigma$  via the formula
	\begin{equation*}
		\begin{aligned}
			H_l(x')[j,k] &= \langle   \partial_j x_{\Sigma_l}(x') , W(x_\Sigma) \partial_k x_{\Sigma_l}(x')\rangle \\
			&= \sum_{n=1}^{\theta-1}  \langle  \partial_j x_{\Sigma_l}(x'),  L_l(x')[n,k] \partial_n x_{\Sigma_l}(x')\rangle \\
			&= \sum_{n = 1}^{\theta -1} M_l(x')[j,n] L_l(x')[n,k] \\
			&= (M_l(x') L_l(x'))[j,k], \qquad j,k \in \{1,\dots,\theta-1\}. 
		\end{aligned} 
	\end{equation*}
	Furthermore, using the definition of $x_{\Sigma_l}(x')$ one concludes  
	\begin{equation*}
	  M_l(x') = I_{\theta-1} +  \nabla \zeta_l (x') (\nabla \zeta_l(x'))^T.
	\end{equation*}
	The inverse of $M_l(x')$ is given by $ I_{\theta-1} - (1+ |\nabla \zeta(x')|^2)^{-1}  \nabla \zeta_l (x') (\nabla \zeta_l(x'))^T$. Hence,
	\begin{equation}\label{eq_L_l}
		L_l(x') = \bigl( I_{\theta-1} - (1+ |\nabla \zeta_l(x')|^2)^{-1}  \nabla \zeta_l (x') (\nabla \zeta_l(x'))^T \bigr)H_l(x').
	\end{equation}
	Since $\zeta_l \in C^2_b(\R^{\theta-1};\R)$ by  Hypothesis~\ref{hypothesis_Sigma}, equation~\eqref{eq_L_l} and the definition of $H_l(x')$ imply  $\sup_{l \in \{1,\dots,p\}, x' \in x_{\Sigma_l}^{-1}(\Sigma)} \abs{L_l(x')} < \infty.$ Now, recall that $\det(I-\varepsilon W(x_\Sigma)) = \det(I_{\theta-1} -\varepsilon L_l(x'))$  for  $x_\Sigma = x_{\Sigma_l}(x') \in \Sigma $.  Moreover, expressing the determinant as the product of the eigenvalues one verifies the equation
	$1-\det(I_{\theta-1} -\varepsilon L_l(x')) = \varepsilon P_l(\varepsilon)$, where $P_l$ is a polynomial in $\varepsilon$ with coefficients depending continuously on the entries of $L_l(x')$.  This shows that (ii) holds if $\varepsilon_{1}>0$ is chosen sufficiently small.

	Finally, the claim in~(iii) follows from  \cite[Proposition 2.6]{BEHL17}, since (i), (ii), and Proposition \ref{prop_iota} (i) show that  $\Sigma$ satisfying Hypothesis~\ref{hypothesis_Sigma} also fulfils  \cite[Hypothesis~2.3]{BEHL17}.
\end{proof}

\subsection{Bochner Lebesgue spaces} \label{sec_Bochner}

In this subsection we summarize some results on Bochner Lebesgue spaces that will be used throughout this paper; for details we refer to \cite[Chapter~1]{HNVW16}. We always assume that $\mathcal{H}$ and $\mathcal{G}$ are separable Hilbert spaces, $\mathcal{O}, \mathcal{O}_1,\mathcal{O}_2$ are Borel sets in  $\R^{n}$, $n \in \N$, and $\mathcal{L}(\mathcal{H},\mathcal{G})$ is the set of bounded linear operators from $\mathcal{H}$ to $\mathcal{G}$.

\begin{definition}\label{def_Bochner_measurable}
	We call $f : \mathcal{O} \to \mathcal{H}$ (weakly) measurable, if for all $\varphi \in \mathcal{H}$ the mapping $\mathcal{O} \ni t \mapsto \spv{f(t)}{\varphi}_{\mathcal{H}}$ is measurable with respect to the Lebesgue measure on $\mathcal{O}$. Furthermore, we call $F:\mathcal{O} \to \mathcal{L}(\mathcal{H},\mathcal{G})$ measurable, if $\mathcal{O} \ni t \mapsto F(t) h$ is measurable for all $h \in \mathcal{H}$. 
\end{definition}

We also recall that a function $f: \mathcal{O} \to \mathcal{H}$ is (strongly) measurable, if $f$ is the pointwise limit of simple functions,
and that in the present situation both notions of measurability coincide due to Pettis theorem, see \cite[Theorem~1.1.20]{HNVW16}.  Moreover, if $f : \mathcal{O} \to \mathcal{H}$ and  $F: \mathcal{O} \to \mathcal{L}(\mathcal{H}, \mathcal{G})$ are measurable, then the function $\mathcal{O} \ni t \mapsto F(t)f(t) \in \mathcal{G}$ is measurable, see \cite[Proposition~1.1.28]{HNVW16}.

For a measurable function $f : \mathcal{O} \to \mathcal{H}$ such that $t \mapsto \norm{f(t)}_{\mathcal H}$ is integrable, the Bochner integral $\int_\mathcal{O} f(t) \,dt \in \mathcal{H}$ is defined in the standard way, see \cite[Definition~1.2.1 and Proposition~1.2.2]{HNVW16}. Many standard results for 
(usual) Lebesgue integrals admit natural generalizations to Bochner integrals. In particular, we will make use of Fubini's theorem for Bochner integrals, see \cite[Proposition 1.2.7]{HNVW16}: If $f:\mathcal{O}_1 \times \mathcal{O}_2 \to \mathcal{H}$ is measurable and the integral $\int_{\mathcal{O}_1 \times \mathcal{O}_2} \norm{f(t,s)}_{\mathcal{H}} \,dt ds$ is finite, then the Bochner integral $\int_{\mathcal{O}_2} f(t,s) \,ds$ exists for a.e. $t \in \mathcal{O}_1$ and the function $\mathcal{O}_1 \ni t \mapsto \int_{\mathcal{O}_2} f(t,s) \,ds$ is measurable and Bochner integrable. 

 Next, we introduce Bochner Lebesgue spaces.

\begin{definition}
	We define $L^2(\mathcal{O};\mathcal{H})$ as the space which contains all (equivalence classes of) measurable functions $f:\mathcal{O} \to \mathcal{H}$ such that 
	\begin{equation*}
		\int_\mathcal{O} \norm{f(t)}_{\mathcal{H}}^2 \,dt < \infty.
	\end{equation*}
	Furthermore, we equip this space with the scalar product
	\begin{equation*}
		\int_\mathcal{O} \spv{f(t)}{g(t)}_{\mathcal{H}} \,dt,  \qquad f,g \in L^2(\mathcal{O};\mathcal{H}).
	\end{equation*}
\end{definition}

It is not difficult to show that $L^2(\mathcal{O};\mathcal{H})$ is a Hilbert space; cf. the comments below \cite[Definition 1.2.15]{HNVW16}. Next, duality and interpolation properties of Bochner Lebesgue spaces are discussed. According to  \cite[Corollary 1.3.13 and Theorem 1.3.21]{HNVW16} the  identification of $f \in L^2(\mathcal{O};\mathcal{H}^*)$ with the functional defined by
\begin{equation*}
	L^2(\mathcal{O};\mathcal{H}) \ni g \mapsto \int_{\mathcal{O}} {_{\mathcal{H}^*}}\spf{f(t)}{g(t)}_{\mathcal{H}} \, dt,
\end{equation*} 
where ${_{\mathcal{H}^*}}\spf{\cdot}{\cdot}_{\mathcal{H}}$ denotes the bilinear dual pairing on $\mathcal{H}^* \times \mathcal{H}$,
induces an isometric isomorphism between $ L^2(\mathcal{O};\mathcal{H}^*)$ and the dual  space of  $L^2(\mathcal{O};\mathcal{H})$, i.e. 
\begin{equation} \label{BL_dual_space}
  L^2(\mathcal{O};\mathcal{H})^* \simeq L^2(\mathcal{O};\mathcal{H}^*).
\end{equation}
If $\mathcal{G}$ is a Hilbert space such that $(\mathcal{H},\mathcal{G})$ is a compatible pair (cf. the notations section), then also $(L^2(\mathcal{O}; \mathcal{H}), L^2(\mathcal{O}; \mathcal{G}))$ is a compatible pair and
	\begin{equation} \label{BL_interpolation}
			L^2(\mathcal{O}; [\mathcal{H},\mathcal{G}]_\tau)=\bigl[L^2(\mathcal{O}; \mathcal{H}),L^2(\mathcal O; \mathcal{G})\bigr]_\tau,  \quad \tau \in (0,1),
	\end{equation}
with equivalent norms, see \cite[Theorem 2.2.6 and Corollary C.4.2]{HNVW16}.
 
In this paper we are particularly interested in the Bochner Lebesgue spaces $L^2((-1,1);H^r(\Sigma;\C^N))$, $r \in [-2,2]$, where $\Sigma \subset \mathbb{R}^\theta$, $\theta \in \{ 2, 3 \}$, is a hypersurface satisfying Hypothesis~\ref{hypothesis_Sigma} and $H^r(\Sigma;\C^N)$ is defined as in Section~\ref{sec_geometry}.  We summarize important properties of these spaces in  Proposition~\ref{prop_Bochner}, where  (i) and (ii) follow from \eqref{BL_interpolation} and~\eqref{BL_dual_space} and the properties of $H^{r}(\Sigma;\C^N)$ in Section~\ref{sec_geometry} and  \eqref{eq_interpolation}. Assertion~(iii) follows from \cite[Proposition~1.2.24 and Proposition~1.2.25]{HNVW16}.
\begin{proposition}\label{prop_Bochner}
		Let $\Omega_\pm,\Sigma \subset \mathbb{R}^\theta$, $\theta \in \{ 2, 3 \}$, satisfy Hypothesis~\ref{hypothesis_Sigma}. Then, the following is true:
	\begin{itemize}
		\item[$\textup{(i)}$]
		If $\tau \in (0,1)$, $r_1,r_2 \in [-2,2]$ and $r = (1-\tau)r_1 + \tau r_2$, then 
		\begin{equation*}%\label{eq_bochner interpolation}
			\begin{split}
			&L^2((-1,1);H^{r}(\Sigma;\C^N) ) =  \bigl[L^2((-1,1);H^{r_1}(\Sigma;\C^N) ),L^2((-1,1);H^{r_2}(\Sigma;\C^N) )\bigr]_\tau 
			\end{split}
		\end{equation*}
		and the corresponding norms are equivalent.
		\item[$\textup{(ii)}$]
		For $r \in [0,2]$ there exists an isometric isomorphism between the dual space of $L^2((-1,1);H^r(\Sigma;\C^N))$ and $L^2((-1,1);H^{-r}(\Sigma;\C^N))$, i.e.
		\begin{equation*}%\label{eq_Bochner_dual}
			L^2((-1,1);H^r(\Sigma;\C^N))^* \simeq L^2((-1,1);H^{-r}(\Sigma;\C^N)).
		\end{equation*} 
		\item[$\textup{(iii)}$] For any $a > 0$ the identification of $F \in L^2(\Sigma \times (-a,a);\C^N)$ with the function $f: t \mapsto F(\cdot,t)$  induces an isometric isomorphism between the spaces $L^2(\Sigma \times (-a,a);\C^N)$ and $L^2((-a,a);L^2(\Sigma;\C^N))$ and
		\begin{equation*}
			\bigg(\int_{-a}^{a} f(t) \,dt \bigg)(x_\Sigma) = \int_{-a}^a f(t)(x_\Sigma) \,dt \quad \textup{ for } \sigma\textup{-a.e. } x_\Sigma \in \Sigma.
		\end{equation*} 
	\end{itemize}
\end{proposition}

The following identification turns out to be useful in our considerations: For 
$\mathcal{Q} \in L^\infty((-1,1);\C)$ and a bounded operator $\mathcal{A}$ in $H^r(\Sigma;\C^N)$, $r \in [-2,2]$, we identify 
\begin{equation*}
	\mathcal{M}_{\mathcal{Q}}: L^2((-1,1);H^r(\Sigma;\C^N)) \to L^2((-1,1);H^r(\Sigma;\C^N)),\quad 
	(\mathcal{M}_{\mathcal{Q}}f)(t) := \mathcal{Q}(t) f(t), 
\end{equation*}
and 
\begin{equation*}
	\mathcal{M}_{\mathcal{A}} : L^2((-1,1);H^r(\Sigma;\C^N)) \to L^2((-1,1);H^r(\Sigma;\C^N)),\quad
	(\mathcal{M}_{\mathcal{A}}f)(t) :=  \mathcal{A}(f(t)), 
\end{equation*}
with $\mathcal{Q}$ and $\mathcal{A}$, respectively. Note  that the norms  $\norm{\mathcal{M}_{\mathcal{Q}}}_{r \to r}$ and $\norm{\mathcal{M}_{\mathcal{A}}}_{r \to r}$ are bounded by $\norm{\mathcal{Q}}_{L^\infty((-1,1);\C)}$ and $\norm{\mathcal{A}}_{H^r(\Sigma;\C^N)\to H^r(\Sigma;\C^N)}$, respectively. 
We  will also  use  the bounded embedding
\begin{equation}\label{eq_def_embedding_op}
	\mathfrak{J}:H^r(\Sigma;\C^N) \to L^2((-1,1);H^r(\Sigma;\C^N)),\quad
	(\mathfrak{J}\varphi)(t) := \varphi,
\end{equation}
and its adjoint
\begin{equation*}
	\mathfrak{J}^* :L^2((-1,1);H^r(\Sigma;\C^N)) \to H^r(\Sigma;\C^N),\quad
	\mathfrak{J}^*f = \int_{-1}^1 f(t) \,dt.
\end{equation*}

\subsection{The free Dirac operator and associated integral operators}\label{sec_free_Dirac}

Let $m \in \mathbb{R}$ and recall that the Dirac matrices $\alpha_1, \dots, \alpha_\theta, \beta \in \mathbb{C}^{N \times N}$ are given by~\eqref{def_Dirac_matrices_2d}--\eqref{def_Dirac_matrices_3d}. Then, the free Dirac operator $H$ is the differential operator in $L^2(\mathbb{R}^\theta; \mathbb{C}^N)$ given by
	\begin{equation} \label{def_free_op}
		\begin{split}
		H u &:= - i (\alpha \cdot \nabla) u + m \beta u, \qquad \dom H := H^1(\R^\theta;\C^N).
		\end{split}
	\end{equation}
With the help of the Fourier transform one gets that $H$ is self-adjoint in $L^2(\R^\theta;\C^N)$ and $\sigma(H) = \left(-\infty, -\abs{m}\right] \cup \left[\abs{m}, \infty \right)$, see for instance \cite[Section~2]{BHT22b} for $\theta=2$ and \cite[Theorem 1.1]{T92} for $\theta =3$. For $z \notin \sigma(H)$ the resolvent $R_z$ is 
\begin{equation} \label{resolvent_free_op}
  R_z u(x) := (H-z)^{-1} u(x) = \int_{\mathbb{R}^\theta} G_z(x-y) u(y) \,d y, \qquad u \in L^2(\mathbb{R}^\theta; \mathbb{C}^N), ~x \in \mathbb{R}^\theta,
\end{equation}
where  $G_z$ is given for $\theta=2$ and $x \in \mathbb{R}^2 \setminus \{ 0 \}$ by
\begin{equation}\label{eq_G_z_2D}
\begin{split}
G_z(x) = \frac{\sqrt{ z^2-m^2}}{2\pi} &K_1\Bigl(-i \sqrt{ z^2-m^2}\abs{x}\Bigr)\frac{\alpha \cdot x}{\abs{x}} \\
		& \qquad +\frac{1}{2\pi} K_0\Bigl(-i \sqrt{ z^2-m^2}\abs{x}\Bigr)(m\beta +  z I_2)
		\end{split}
\end{equation}	
and for $\theta=3$ and $x \in \mathbb{R}^3 \setminus \{ 0 \}$ by
\begin{equation}\label{eq_G_z_3D}
	G_ z(x) = \biggl(  z I_4 + m \beta + i\Bigl( 1 - i \sqrt{ z^2 -m^2}\abs{x} \Bigr) \frac{ \alpha \cdot x }{\abs{x}^2}\biggr)\frac{e^{i\sqrt{ z^2 -m^2} \abs{x}}}{4 \pi \abs{x}};
\end{equation}
see, e.g., \cite{BHOP20, BHSS22, T92}.
Here, $K_0$ and $K_1$ denote the modified Bessel functions of the second kind of order zero and one, respectively, and the branch of the square root is fixed by $\text{Im}\, \sqrt{w} > 0$ for $w \in \C \setminus [0, \infty)$.
Note that $R_z$ is bounded in $L^2(\mathbb{R}^\theta; \mathbb{C}^N)$ and it can also be viewed as a bounded operator from $L^2(\R^\theta;\C^N)$ to $H^1(\R^\theta;\C^N)$. In fact, with the help of the Fourier transform, it is not difficult  to show that $R_z$ gives rise to a bounded operator from $H^s(\mathbb{R}^\theta; \mathbb{C}^N)$ to $H^{s+1}(\mathbb{R}^\theta; \mathbb{C}^N)$ for all $s \in \R$.

We move on to the discussion of potential and boundary integral operators associated with the free Dirac operator. In the following, let $z \in \rho(H) = \mathbb{C} \setminus ((-\infty, -|m|] \cup [|m|, \infty))$ be fixed and let $\Omega_\pm$ and $\Sigma \subset \mathbb{R}^\theta$ satisfy Hypothesis~\ref{hypothesis_Sigma}. First, we introduce the potential operator $\Phi_z: L^2(\Sigma; \mathbb{C}^N) \rightarrow L^2(\mathbb{R}^\theta; \mathbb{C}^N)$ by
\begin{equation} \label{def_Phi_z}
	\begin{aligned}
	\Phi_ z   \varphi(x) &:= \int_\Sigma G_ z (x-y_\Sigma) \varphi(y_\Sigma) \,d\sigma(y_\Sigma),  \qquad \varphi \in L^2(\Sigma; \mathbb{C}^N),~x \in \R^\theta.
	\end{aligned}
\end{equation}
We note that $\Phi_z$ is indeed well-defined and bounded, see \cite[Lemma 2.1]{AMV14}. 
Further properties of $\Phi_z$ are summarized in the following proposition. For compact hypersurfaces $\Sigma$, these results are well known, see, e.g., \cite[Theorem~4.3]{BHSS22}, but for unbounded $\Sigma$ they were not treated in the literature so far. For completeness, we give a proof of these results in Appendix~\ref{appendix_Phi_C}.

\begin{proposition}\label{prop_Phi_z}
	Let $z \in \rho(H) = \mathbb{C} \setminus ((-\infty, -|m|] \cup [|m|, \infty))$ and let $\Phi_z$ be given by~\eqref{def_Phi_z}. Then, the following is true:
	\begin{itemize}
	  \item[$\textup{(i)}$] For any $ r \in [0, \tfrac{1}{2}]$ the operator $\Phi_z$ gives rise to a bounded operator
	\begin{equation*}
	\begin{split}
	\Phi_z :H^r(\Sigma;\C^N) \to H^{r+1/2}(\Omega_+;\C^N) \oplus H^{r +1/2}(\Omega_-;\C^N).
	\end{split}
	\end{equation*}
	\item[$\textup{(ii)}$] For $\varphi \in H^{1/2}(\Sigma; \mathbb{C}^N)$ one has $[(-i (\alpha \cdot \nabla) + m \beta - zI_N) \Phi_z \varphi]_\pm = 0$.
	\item[$\textup{(iii)}$] The adjoint $\Phi_z^*: L^2(\mathbb{R}^\theta; \mathbb{C}^N) \to  L^2(\Sigma; \mathbb{C}^N)$ of $\Phi_z$ acts on $u \in L^2(\mathbb{R}^\theta; \mathbb{C}^N)$ as
	\begin{equation} \label{equation_Phi_star}
	  \Phi_{z}^* u(x_\Sigma) = \int_{\R^\theta} G_{\overline{ z}}(x_\Sigma -y) u(y) \, dy = \tr R_{\overline{z}} u(x_\Sigma), \qquad x_\Sigma \in \Sigma,
	\end{equation}
	and $\Phi_z^*$ gives rise to a bounded operator $\Phi_z^*: L^2(\mathbb{R}^\theta; \mathbb{C}^N) \rightarrow H^{1/2}(\Sigma; \mathbb{C}^N)$.
	\end{itemize}
\end{proposition}

Finally, we introduce a family of boundary integral operators. Let $z \in \rho(H) = \mathbb{C} \setminus ((-\infty, -|m|] \cup [|m|, \infty))$. Then, we define the map $\mathcal{C}_z: H^{1/2}(\Sigma; \mathbb{C}^N) \rightarrow H^{1/2}(\Sigma; \mathbb{C}^N)$ by
\begin{equation} \label{def_C_z}
  \mathcal{C}_z \varphi := \frac{1}{2} ( \tr^+ (\Phi_z \varphi)_+ + \tr^- (\Phi_z \varphi)_- ), \qquad \varphi \in H^{1/2}(\Sigma; \mathbb{C}^N).
\end{equation}
We remark that the operator $\mathcal{C}_z$ can be represented as a strongly singular boundary integral operator, see for instance \cite[equation (4.5) and Proposition 4.4 (ii)]{BHSS22} for the case that $\Omega_+$ is bounded. However, for our purposes the representation in \eqref{def_C_z} is more convenient.
The basic properties of $\mathcal{C}_z$ are stated in the following proposition. Again, for compact hypersurfaces $\Sigma$ they are well-known, see, e.g., \cite[Theorem~4.3, Proposition~4.4, and Corollary~4.5]{BHSS22}. For general (also unbounded) hypersurfaces $\Sigma$ satisfying Hypothesis~\ref{hypothesis_Sigma} we give a proof in Appendix~\ref{appendix_Phi_C}.
 
\begin{proposition} \label{proposition_C_z}
  Let $z \in \rho(H) = \mathbb{C} \setminus ((-\infty, -|m|] \cup [|m|, \infty))$ and let $\mathcal{C}_z$ be given by~\eqref{def_C_z}. Then, the following is true:
  \begin{itemize}
	\item[$\textup{(i)}$] For any $r \in [-\frac{1}{2},\frac{1}{2}]$ the map $\mathcal{C}_z$ has a bounded extension $\mathcal{C}_z: H^r(\Sigma;\C^N) \rightarrow H^r(\Sigma;\C^N)$.
	\item[$\textup{(ii)}$] For any $r \in(0,\frac{1}{2}]$ and $\varphi \in H^r(\Sigma;\C^N)$ one has
	\begin{equation*}%\label{eq_C_z_formula}
	\mathcal{C}_ z \varphi =  \pm \frac{i}{2}  (\alpha \cdot \nu) \varphi +  \tr^\pm (	\Phi_ z   \varphi)_\pm. 
	\end{equation*}
  \end{itemize}
\end{proposition}

	We note that item~(ii) of the previous proposition is a version of the well-known Plemelj-Sokhotski formula, see for instance \cite[Lemma 3.3 (i)]{AMV14}, where $\mathcal{C}_z$ is represented as a singular boundary integral operator.

\section{Resolvent formula for $H_\varepsilon$ and analysis of the associated integral operators}\label{sec_shifted_operators}

Throughout this section let $\varepsilon_1$ be the number specified in Proposition~\ref{prop_tubular_neighbourhood}, so that $\iota$ acts  as a bijective map from $\Sigma \times (-\varepsilon_1,\varepsilon_1)$ to $\Omega_{\varepsilon_1}$. Moreover, let $\varepsilon \in (0, \varepsilon_1)$, $z \in \rho(H) = \mathbb{C} \setminus ((-\infty, -|m|] \cup[|m|, \infty))$,  $G_z$ be the function defined in~\eqref{eq_G_z_2D}--\eqref{eq_G_z_3D}, and $\Sigma$ be a hypersurface satisfying Hypothesis~\ref{hypothesis_Sigma} with associated Weingarten map $W$ (see Definition~\ref{def_Weingarten}). This section is devoted to the study of the integral operators which formally act on $f \in L^2((-1,1);L^2(\Sigma;\C^N))$ and $u \in L^2(\R^\theta;\C^N)$ as
\begin{subequations} \label{eq_ABC_integral_rep} 
\begin{align}
  \label{def_A_eps}
		A_\varepsilon( z)f(x) &= \int_{-1}^1\int_{\Sigma} G_ z(x-y_\Sigma -\varepsilon s\nu(y_\Sigma))   f(s)(y_\Sigma)\det(I-\varepsilon sW(y_\Sigma)) \, d\sigma(y_\Sigma) \,  ds,\\
		\label{def_B_eps}
		B_\varepsilon( z)f(t)(x_\Sigma) &=   \int_{-1}^1 \int_{\Sigma} G_ z(x_\Sigma +\varepsilon t\nu(x_\Sigma) -y_\Sigma - \varepsilon s\nu(y_\Sigma))   f(s)(y_\Sigma) \\
		&\qquad \qquad \qquad \qquad \qquad \qquad \cdot \det(I- \varepsilon sW(y_\Sigma)) \, d\sigma(y_\Sigma) \, ds, \notag 		\\
		\label{def_C_eps}
		C_\varepsilon( z)u(t)(x_\Sigma) &= \int_{\R^\theta} G_ z(x_\Sigma +  \varepsilon t\nu(x_\Sigma) -y) u(y)\, dy,
\end{align}
\end{subequations}
	for a.e. $x \in \R^\theta$,  a.e. $t \in (-1,1)$, and $\sigma$-a.e. $x_\Sigma \in \Sigma$. First, in Section~\ref{section_def_ABC} we define these operators rigorously and show their relation to the resolvent of the operator $H_\varepsilon$ in~\eqref{eq_H_eps}. Then, in Section~\ref{sec_shift_operator} we introduce and investigate a shift operator which plays an important role in the convergence analysis of $A_\varepsilon(z), B_\varepsilon(z)$, and $C_\varepsilon(z)$ in Section~\ref{sec_conv_analyis}. Finally, in Section~\ref{sec_inverse} we discuss the convergence of $(I + {B_\varepsilon( z)}Vq)^{-1}$.

\subsection{Definition and elementary results on $A_\varepsilon(z), B_\varepsilon(z)$, and $C_\varepsilon(z)$} \label{section_def_ABC}
 
First, we rigorously define the operators $A_\varepsilon(z), B_\varepsilon(z)$, and $C_\varepsilon(z)$ formally given by~\eqref{eq_ABC_integral_rep}. Recall that $\Omega_\varepsilon$ was defined in~\eqref{def_Omega_eps} and introduce the mappings  
\begin{equation}\label{eq_I_eps}
	\begin{aligned}
		&\mathcal{I}_\varepsilon : L^2((-\varepsilon,\varepsilon);L^2(\Sigma;\C^N))  \to L^2(\Omega_\varepsilon;\C^N), &
		&\mathcal{I}_\varepsilon f (x_\Sigma+t \nu(x_\Sigma)) := f(t)(x_\Sigma),\\
		&\mathcal{I}_\varepsilon^{-1} : L^2(\Omega_\varepsilon;\C^N) \to L^2((-\varepsilon,\varepsilon);L^2(\Sigma;\C^N)) , &
		&\mathcal{I}_\varepsilon^{-1} u (t)(x_\Sigma) := u(x_\Sigma+t \nu(x_\Sigma)),
	\end{aligned}
\end{equation}
 and 
\begin{equation}\label{eq_S_eps}
	\begin{aligned}
		&\mathcal{S}_\varepsilon : L^2((-1,1);L^2(\Sigma;\C^N)) \to L^2((-\varepsilon,\varepsilon);L^2(\Sigma;\C^N)), &&
		\mathcal{S}_\varepsilon g (t) := \frac{1}{\sqrt{\varepsilon}}g\Bigl(\frac{t}{\varepsilon}\Bigr),\\
		&\mathcal{S}_\varepsilon^{-1} : L^2((-\varepsilon,\varepsilon);L^2(\Sigma;\C^N)) \to L^2((-1,1);L^2(\Sigma;\C^N)), &&
		\mathcal{S}_\varepsilon^{-1} g (t) :=\sqrt{\varepsilon}g(\varepsilon t).
	\end{aligned}
\end{equation}
According to Proposition~\ref{prop_tubular_neighbourhood} and Proposition~\ref{prop_Bochner}~(iii) for any $\varepsilon \in (0,\varepsilon_1)$ these mappings are well-defined, bounded, invertible, and their inverses have the claimed form, see also  \cite[equations~(3.6) and~(3.7)]{BEHL17}.
Moreover, set $u_\varepsilon:= \frac{\chi_{\Omega_\varepsilon}}{\sqrt{\varepsilon}}$, where $\chi_{\Omega_\varepsilon}$ is the characteristic function for $\Omega_\varepsilon$, and define the operators 
\begin{equation*}
U_\varepsilon: L^2(\mathbb{R}^\theta; \mathbb{C}^N) \to L^2(\Omega_\varepsilon; \mathbb{C}^N)\quad\text{and}\quad 
U_\varepsilon^*: L^2(\Omega_\varepsilon; \mathbb{C}^N) \to L^2(\mathbb{R}^\theta; \mathbb{C}^N)
\end{equation*}
acting on $u \in L^2(\mathbb{R}^\theta; \mathbb{C}^N)$ and $v \in L^2(\Omega_\varepsilon; \mathbb{C}^N)$ as
\begin{equation*}
  U_\varepsilon u = (u_\varepsilon u)\upharpoonright \Omega_\varepsilon \quad \text{and} \quad U_\varepsilon^* v = \begin{cases} u_\varepsilon v &\text{ in } \Omega_\varepsilon, \\ 0 & \text{ in } \mathbb{R}^\theta \setminus \Omega_\varepsilon. \end{cases}
\end{equation*}
Recall that we use the notation $R_z = (H-z)^{-1}$ for the resolvent of the free Dirac operator $H$ given by~\eqref{def_free_op}. Then, we define  
\begin{equation} \label{def_ABC_operators}
	\begin{aligned}	
		A_\varepsilon( z) :=R_ z U_\varepsilon^* \mathcal{I}_\varepsilon \mathcal{S}_\varepsilon&:   L^2((-1,1);L^2(\Sigma;\C^N)) \to L^2(\R^\theta;\C^N), \\
		B_\varepsilon( z) := \mathcal{S}_\varepsilon^{-1}\mathcal{I}_\varepsilon^{-1} U_\varepsilon R_ z U_\varepsilon^* \mathcal{I}_\varepsilon \mathcal{S}_\varepsilon&:  L^2((-1,1);L^2(\Sigma;\C^N)) \to L^2((-1,1);L^2(\Sigma;\C^N)),  \\
		C_\varepsilon( z) := \mathcal{S}_\varepsilon^{-1}\mathcal{I}_\varepsilon^{-1} U_\varepsilon R_ z&: L^2(\R^\theta;\C^N) \to L^2((-1,1);L^2(\Sigma;\C^N)). 
	\end{aligned}
\end{equation}
Note that by definition these operators are well-defined and bounded. In the next proposition we show that these operators coincide with the formal expressions in~\eqref{eq_ABC_integral_rep}.

\begin{proposition}\label{prop_ABC_operators}
	For the operators $A_\varepsilon(z)$, $B_\varepsilon(z)$, and $C_\varepsilon(z)$ defined in~\eqref{def_ABC_operators} the representations in~\eqref{eq_ABC_integral_rep} hold.
\end{proposition}
\begin{proof}
	First, we show the claim for $A_\varepsilon(z)$. Let $f \in L^2((-1,1);L^2(\Sigma;\C^N))$. Using~\eqref{resolvent_free_op}, Proposition~\ref{prop_tubular_neighbourhood}~(iii), \eqref{eq_I_eps}, and~\eqref{eq_S_eps} we find
	\begin{equation*}
		\begin{split}
		A_\varepsilon( z) f(x) &= (R_ z U_\varepsilon^* \mathcal{I}_\varepsilon \mathcal{S}_\varepsilon f)(x) = \int_{\R^\theta} G_ z(x-y)(U_\varepsilon^* \mathcal{I}_\varepsilon \mathcal{S}_\varepsilon f)(y) \,dy \\
		&=\int_{\Omega_\varepsilon} G_ z(x-y)  u_\varepsilon(y) (\mathcal{I}_\varepsilon\mathcal{S}_\varepsilon f)(y) \,dy\\
		&=\int_{-\varepsilon}^\varepsilon\int_{\Sigma} G_ z(x-(y_\Sigma +s\nu(y_\Sigma)))  \frac{1}{\sqrt{\varepsilon}}  (\mathcal{I}_\varepsilon\mathcal{S}_\varepsilon f)(y_\Sigma +s\nu(y_\Sigma))\\
		&\qquad \qquad \qquad \qquad \qquad \qquad \qquad \cdot \det(I-sW(y_\Sigma)) \, d\sigma(y_\Sigma) \, ds\\
		%&=\int_{-\varepsilon}^\varepsilon\int_{\Sigma} G_ z(x-y_\Sigma -s\nu(y_\Sigma))  \frac{1}{\sqrt{\varepsilon}} (\mathcal{S}_\varepsilon f)(s)(y_\Sigma)\det(I-sW(y_\Sigma)) \, d\sigma(y_\Sigma) \, ds\\
		&=\int_{-\varepsilon}^\varepsilon \int_{\Sigma} G_ z(x-y_\Sigma -s\nu(y_\Sigma))   \frac{1}{{\varepsilon}} f\Bigl(\frac{s}{\varepsilon}\Bigr)(y_\Sigma) \det(I-sW(y_\Sigma)) \, d\sigma(y_\Sigma) \, ds\\
		&=\int_{-1}^1\int_{\Sigma} G_ z(x-y_\Sigma -\varepsilon s\nu(y_\Sigma))   f(s)(y_\Sigma)\det(I-\varepsilon sW(y_\Sigma)) \, d\sigma(y_\Sigma)   \, ds 
		\end{split}
	\end{equation*}
	for a.e. $x \in \R^\theta$, which is the claimed identity.
	Next, to prove the claim for $C_\varepsilon(z)$ we note for $v \in L^2(\mathbb{R}^\theta; \mathbb{C}^N)$, a.e. $t \in (-1,1)$, and $\sigma$-a.e. $x_\Sigma \in \Sigma$ that
	\begin{equation*}
		\begin{split}
		C_\varepsilon( z) v(t)(x_\Sigma) &= (\mathcal{S}_\varepsilon^{-1}\mathcal{I}_\varepsilon^{-1} U_\varepsilon R_ z  v )(t)(x_\Sigma) \\
		&=   \sqrt{\varepsilon}(U_\varepsilon R_ z v)(x_\Sigma  + \varepsilon t\nu(x_\Sigma)) = \int_{\R^\theta}G_z(x_\Sigma + \varepsilon t\nu(x_\Sigma) -y) v(y)\, dy.
		\end{split}
	\end{equation*}
	The representation for $B_\varepsilon(z)$ follows by combining the last two calculations.
\end{proof}

Next, we show a resolvent formula for $H_\varepsilon$ involving the operators $A_\varepsilon(z), B_\varepsilon(z)$, and $C_\varepsilon(z)$ which will be useful for the convergence analysis. Using the identifications in the end of Section~\ref{sec_Bochner} we regard $Vq$ in the following as a 
multiplication operator in $L^2((-1,1);L^2(\Sigma;\C^N))$. A similar formula is shown in \cite[Lemma~3.1]{MP18}.

\begin{proposition}\label{prop_resolvent_formula}
	Let $A_\varepsilon(z)$, $B_\varepsilon(z)$, and $C_\varepsilon(z)$ be given by~\eqref{def_ABC_operators}, let $q$  and $V$ be given by  \eqref{eq_q} and \eqref{eq_V}, respectively, and let $H_\varepsilon$ be given by~\eqref{eq_H_eps}. If $-1 \in \rho(B_\varepsilon( z)Vq)$, then $z \in \rho(H_\varepsilon)$ and
	\begin{equation*}%\label{eq_H_eps_resolvent}
		(H_\varepsilon -  z)^{-1} = R_ z - A_\varepsilon( z) Vq(I + B_\varepsilon( z) Vq)^{-1}C_\varepsilon( z).
	\end{equation*}
\end{proposition}
\begin{proof}
	Let $V_\varepsilon$ be given by \eqref{eq_short_range_potential} and set $v_\varepsilon:= \sqrt{\varepsilon} V_\varepsilon$. Then, $H_\varepsilon = H + v_\varepsilon u_\varepsilon$ and $v_\varepsilon \mathcal{I}_\varepsilon \mathcal{S}_\varepsilon = u_\varepsilon \mathcal{I}_\varepsilon \mathcal{S}_\varepsilon Vq$. Hence, due to the  invertibility of  $I + B_\varepsilon(z)Vq$ and~\eqref{def_ABC_operators} we obtain
	\begin{equation*}
		\begin{split}
		(&H_\varepsilon -  z)( R_ z - A_\varepsilon(z) Vq(I + B_\varepsilon(z)Vq )^{-1} C_\varepsilon(z)) \\
		&= (H-  z+   v_\varepsilon u_\varepsilon)(R_ z - R_ z v_\varepsilon(I + u_\varepsilon R_ z v_\varepsilon)^{-1}u_\varepsilon R_ z)\\
		& = I + v_\varepsilon u_\varepsilon R_ z - v_\varepsilon(I + u_\varepsilon R_ z v_\varepsilon)^{-1}u_\varepsilon R_ z
		-v_\varepsilon u_\varepsilon R_ z v_\varepsilon(I + u_\varepsilon R_ z v_\varepsilon)^{-1}u_\varepsilon R_ z \\
		& = I + v_\varepsilon u_\varepsilon R_ z - v_\varepsilon(I +u_\varepsilon R_ z v_\varepsilon)^{-1}u_\varepsilon R_ z
		+v_\varepsilon (I + u_\varepsilon R_ z v_\varepsilon)^{-1}u_\varepsilon R_ z -v_\varepsilon u_\varepsilon R_ z\\
		&=I.
		\end{split}
	\end{equation*}
    A similar calculation shows that
	$	( R_ z - A_\varepsilon(z) Vq(I + B_\varepsilon(z)Vq )^{-1} C_\varepsilon(z)) (H_\varepsilon -  z) =I$
	is true. The latter two equations imply the claim of this proposition.
\end{proof}

\subsection{The shift operator}\label{sec_shift_operator}
In this subsection we introduce and study a shift operator which turns out to be useful in the convergence analysis of the maps $A_\varepsilon(z), B_\varepsilon(z)$, and $C_\varepsilon(z)$ in~\eqref{def_ABC_operators}. For that, we first fix a $C^1_b$-extension of the normal vector field $\nu$ to $\R^{\theta}$, which we also denote by $\nu$. In the following, we show how one possible choice of this extension can be constructed, which also proves that such an extension exists. Choose  $\varphi_\nu \in C^1(\R;\R)$ with $\varphi_\nu(0) = 1$ and compact support in $(-\varepsilon_1, \varepsilon_1)$, where $\varepsilon_1$ is the number specified in Proposition~\ref{prop_tubular_neighbourhood}. Since $\Sigma$ is assumed to satisfy Hypothesis~\ref{hypothesis_Sigma}, it is not difficult to show that the vector field defined by
\begin{equation*}
	\R^{\theta} \ni x \mapsto 
	\begin{cases}
	\nu(x_\Sigma)\varphi_\nu(t), & \text{if } x = x_\Sigma + t \nu(x_\Sigma) \in \Omega_{\varepsilon_1} \text{ with } (x_\Sigma,t) \in \Sigma \times (-\varepsilon_1,\varepsilon_1),\\
	0,& \text{if } x \notin \Omega_{\varepsilon_1},
	\end{cases}
\end{equation*}
is a $C^1_b$-extension of $\nu$ which is supported in $\Omega_{\varepsilon_1}$.  Next, we define for $\delta \in \mathbb{R}$ the shift operator $\tau_\delta : L^2(\R^\theta;\C^N) \to L^2(\R^\theta;\C^N)$ by
	\begin{equation} \label{def_tau_delta}
		\tau_\delta u(x) : = u\left(x + \delta \nu(x)\right), \quad  x \in \R^\theta.
	\end{equation}

\begin{proposition}\label{prop_shift_operator} 
	Let $D \nu$ be the Jacobi matrix of $\nu$ and $\delta_0 \in (0, \norm{D \nu}_{L^\infty(\R^\theta;\R^{\theta\times \theta})}^{-1})$. Then, for any $r \in  [0,1]$ the operators $\tau_{\delta}$, $\delta \in [-\delta_0,\delta_0]$, are uniformly bounded in $H^{r}(\R^\theta;\C^N)$ and for $r' \in [0,r]$ 
	\begin{equation} \label{estimate_tau_delta}
		\norm{\tau_{\delta} - I}_{H^r(\R^\theta;\C^N) \to H^{r'}(\R^\theta;\C^N)} \leq C \abs{\delta}^{r-r'}   
	\end{equation}
	holds for all $\delta \in [-\delta_0,\delta_0]$, where $C>0$ is independent of $\delta$.
\end{proposition}
\begin{proof}
	Fix $\delta\in [-\delta_0,\delta_0]$ and observe first that $I_\theta + \delta D \nu(x)$ is invertible for all $x\in \R^\theta$ and 
	the norm of the inverse is bounded by $(1- \abs{\delta_0}\norm{D \nu}_{L^\infty(\R^\theta;\R^{\theta\times \theta})})^{-1}$. 
	The same bound holds for the modulus of the eigenvalues of $(I_\theta + \delta D \nu(x))^{-1}$ and hence we conclude
	\begin{equation}\label{detdet}
		\abs{\det((I_\theta + \delta D \nu(x))^{-1})} \leq \frac{1}{(1- \abs{\delta_0}\norm{D \nu}_{L^\infty(\R^\theta;\R^{\theta\times \theta})} )^{\theta}},\quad x\in \R^\theta.
	\end{equation}
	We start by showing the uniform boundedness of $\tau_\delta$ for $r=0$. Let $u \in L^2(\R^\theta;\C^N)$. Then, a change of variables and \eqref{detdet} 
	lead to
	\begin{equation}\label{abc1}
		\begin{split}
			\int_{\R^\theta} \abs{\tau_\delta u(x)}^2\,dx &= \int_{\R^\theta} \abs{u(x + \delta\nu(x))}^2\,dx \\
%			\\&= \int_{\R^\theta} \abs{u(x + \delta\nu(x))}^2 \frac{\abs{ \det\left( I_\theta + \delta D \nu(x) \right)}}{\abs{ \det\left( I_\theta + \delta	D \nu(x)\right)}}\,dx \\
			&= \int_{\R^\theta} \abs{u(x + \delta\nu(x))}^2 \abs{ \det( I_\theta + \delta D \nu(x) )\det(( I_\theta + \delta	D \nu(x))^{-1})}\,dx\\
			& \leq\frac{1}{{(1- \abs{\delta_0}\norm{D \nu}_{L^\infty(\R^\theta;\R^{\theta\times \theta})})^\theta}} \int_{\R^\theta}  \abs{u(x)}^2\,dx,
		\end{split}
	\end{equation}
	and it follows that the operators $\tau_{\delta}$, $\delta \in [-\delta_0,\delta_0]$, are uniformly bounded
	 in $L^2(\mathbb{R}^\theta; \mathbb{C}^N)$. To see the uniform boundedness of the operators $\tau_\delta$ in $H^1(\mathbb{R}^\theta; \mathbb{C}^N)$, let $u \in \mathcal{D}(\R^{\theta};\C^N)$ and compute in a similar way as above 
	\begin{equation}\label{abc2}
	\begin{split}
			\int_{\R^\theta} \abs{D (\tau_\delta u)(x)}^2\,dx &= \int_{\R^\theta} \abs{(D u)(x + \delta\nu(x))(I_\theta +\delta D \nu (x))}^2\,dx \\
			&\leq \frac{(1+\delta_0\norm{D \nu}_{L^\infty(\R^\theta;\R^{\theta\times \theta})})^2}{{(1- \abs{\delta_0}\norm{D \nu}_{L^\infty(\R^\theta;\R^{\theta\times \theta})})^\theta}} \int_{\R^\theta}  \abs{D u(x)}^2\,dx.
			\end{split}
	\end{equation}
	By density this estimate remains valid for $u \in H^1(\R^{\theta};\C^N)$.
	Therefore, the uniform boundedness of the operators $\tau_\delta$ in $H^1(\mathbb{R}^\theta; \mathbb{C}^N)$ follows from \eqref{abc1} and \eqref{abc2}. Eventually, using interpolation one concludes that $\tau_\delta$ is uniformly bounded in $H^r(\mathbb{R}^\theta; \mathbb{C}^N)$ for any $r \in [0, 1]$.
	
	It remains to prove~\eqref{estimate_tau_delta}. Since we already have shown that $\tau_\delta$ is uniformly bounded in $H^r(\mathbb{R}^\theta; \mathbb{C}^N)$, the claim in~\eqref{estimate_tau_delta} holds for $r=r' \in [0,1]$. Next, we show~\eqref{estimate_tau_delta} for  $r'=0$ and $r=1$. With the main theorem of calculus and the chain rule we find for  $u \in \mathcal{D}(\R^{\theta};\C^N)$
	\begin{equation*}
		\begin{split}
			\int_{\R^\theta} \abs{\tau_\delta u(x)-u(x)}^2\,dx 
			%&= \int_{\R^\theta} \Bigl|\int_{0}^\delta \frac{d}{d t} u(x +t\nu(x)) dt \Bigr|^2\,dx 
			&= \int_{\R^\theta} \biggl|\int_{0}^\delta D u(x +t\nu(x)) \nu(x) dt \biggr|^2\,dx \\
			&\leq \int_{\R^\theta} \biggl( \int_{0}^\delta \abs{(\tau_{t} D u)(x)}^2 \,dt\biggr) \biggl(\int_0^\delta \abs{\nu(x)}^2 \,dt\biggr)\, dx\\
			&\leq \vert\delta\vert \norm{\nu}_{L^\infty(\R^\theta;\R^{\theta})}^2 \int_{0}^\delta \norm{\tau_t D u}_{L^2(\R^\theta;\C^{N\times\theta})}^2 \,dt \\
			&\leq   C\vert\delta\vert  \int_{0}^\delta \norm{ D u}_{L^2(\R^\theta;\C^{N\times\theta})}^2 \,dt \leq C\abs{\delta }^2\norm{u}_{H^1(\R^\theta;\C^N)}^2,
		\end{split}
	\end{equation*}
	where $\tau_{t} D u$ is understood column-wise. By density this estimate remains valid for $u \in H^1(\R^{\theta};\C^N)$ and hence
	$\norm{\tau_{\delta} - I}_{H^1(\R^\theta;\C^N) \to L^2(\R^\theta;\C^N)} \leq C \abs{\delta}   $.
	It remains to prove the claim in the case  $0\leq r'< r \leq1$ with $(r',r) \neq (0,1)$. We set $\mu = r-r' \in (0,1)$ and $\upsilon =\frac{r'}{1-(r-r')} \in [0,1]$. Then,  
	\begin{equation*}
		r' = (1-\mu)\upsilon + \mu 0 \quad\text{and}\quad r = (1-\mu)\upsilon + \mu 1
	\end{equation*}
	and consequently \cite[Theorem B.7]{M00} implies 
	\begin{equation*}
			H^{r'}(\R^\theta;\C^N) = [H^\upsilon(\R^\theta;\C^N),H^0(\R^\theta;\C^N)]_{\mu} =[H^\upsilon(\R^\theta;\C^N),L^2(\R^\theta;\C^N)]_{\mu}
	\end{equation*}
	and
	\begin{equation*}
		H^{r}(\R^\theta;\C^N) = [H^\upsilon(\R^\theta;\C^N),H^1(\R^\theta;\C^N)]_{\mu}.
	\end{equation*}
	Applying~\eqref{eq_interpolation_inequality} yields
	\begin{equation*}
		\begin{split}
			&\norm{I - \tau_{\delta}}_{H^r(\R^\theta;\C^N) \to H^{r'}(\R^\theta;\C^N)} \\
			&\qquad\leq C \norm{I - \tau_{\delta}}_{[H^\upsilon(\R^\theta;\C^N),H^1(\R^\theta;\C^N)]_{\mu} \to [H^\upsilon(\R^\theta;\C^N),L^2(\R^\theta;\C^N)]_{\mu}}\\
			&\qquad\leq C 	\norm{I - \tau_{\delta}}_{H^\upsilon(\R^\theta;\C^N) \to H^{\upsilon}(\R^\theta;\C^N)}^{1-\mu} \norm{I - \tau_{\delta}}_{H^1(\R^\theta;\C^N) \to L^2(\R^\theta;\C^N)}^\mu\\
			&\qquad =C \abs{\delta}^{r-r'},
		\end{split}
	\end{equation*}
	which is exactly~\eqref{estimate_tau_delta}. This finishes the proof of this proposition.
\end{proof}

We will also need a variant of the shift operator $\tau_\delta$ that acts on functions defined on $\Omega_\pm$. 
Since $\Omega_\pm$ satisfy Hypothesis~\ref{hypothesis_Sigma}
we can make use of Stein's extension operator $E: L^2(\Omega_\pm; \mathbb{C}^N) \rightarrow L^2(\mathbb{R}^\theta; \mathbb{C}^N)$ which has the properties $(E f)_\pm = f$ for $f \in L^2(\Omega_\pm; \mathbb{C}^N)$ and which has a continuous restriction $E: H^r(\Omega_\pm; \mathbb{C}^N) \rightarrow H^r(\mathbb{R}^\theta; \mathbb{C}^N)$ for any $r \geq 0$, 
see \cite[Chapter~6, Section~3, Theorem~5]{S70}. 
We then define the shift operator for functions on $\Omega_\pm$ by
	\begin{equation} \label{def_tau_delta_pm}
			\tau_\delta^{\Omega_\pm} := (\tau_{\delta} E(\cdot))_\pm :L^2(\Omega_\pm;\C^N) \to L^2(\Omega_\pm;\C^N).
	\end{equation}
	The following properties of $\tau_\delta^{\Omega_\pm}$ follow immediately from the properties of $E$ and Proposition~\ref{prop_shift_operator}.

\begin{corollary}\label{cor_shift_operator}
    Let $D \nu$ be the Jacobi matrix of $\nu$ and $\delta_0 \in (0, \norm{D \nu}_{L^\infty(\R^\theta;\R^{\theta\times \theta})}^{-1})$. Then, for any $r \in  [0,1]$ the operators $\tau_{\delta}^{\Omega_\pm}$, $\delta \in [-\delta_0,\delta_0]$, are uniformly bounded in $H^{r}(\Omega_\pm;\C^N)$ and for $r' \in [0,r]$ 
	\begin{equation*} 
		\norm{\tau_{\delta}^{\Omega_\pm} - I}_{H^r(\Omega_\pm;\C^N) \to H^{r'}(\Omega_\pm;\C^N)} \leq C \abs{\delta}^{r-r'}   
	\end{equation*}
	holds for all $\delta \in [-\delta_0,\delta_0]$, where $C>0$ is independent of $\delta$.
\end{corollary}

Eventually, we show that the map $t \mapsto \tau_{t \delta }u$ has a useful continuity property.

\begin{proposition}\label{prop_continous_shift}
	Let $D \nu$ be the Jacobi matrix of $\nu$, $\delta_0 \in (0, \norm{D \nu}_{L^\infty(\R^\theta;\R^{\theta\times \theta})}^{-1})$, $\delta \in [-\delta_0,\delta_0]$, $r \in [0,1]$, $u \in H^r(\R^\theta;\C^N)$ and $v \in H^r(\Omega_\pm;\C^N)$. Then, the functions
	\begin{equation*}
			f_u : (-1,1) \to H^r(\R^\theta;\C^N), \quad
			t \mapsto \tau_{t \delta } u, 
	\end{equation*}
	and
	\begin{equation*}
			f_{v}^{\pm}: (-1,1) \to H^r(\Omega_\pm\;\C^N), \quad 
			t \mapsto \tau_{t \delta }^{\Omega_\pm} v,
	\end{equation*} 
	are continuous.
\end{proposition}
\begin{proof}
	First, consider $u \in \mathcal{D}(\R^{\theta};\C^N)$ and let $t_n, t \in (-1,1)$ such that $t_n \rightarrow t$ as $n \rightarrow \infty$.  Then, with dominated convergence one gets
	\begin{equation*}
		\lim_{n \to \infty}  f_u(t_n) = \lim_{ n \to \infty} u( (\cdot) + \delta t_n  \nu)  =   u( (\cdot) + \delta t  \nu) = f_u(t) \quad \textup{in } H^1(\R^\theta;\C^N).
	\end{equation*}
	Since $H^1(\mathbb{R}^\theta; \mathbb{C}^N)$ is continuously embedded in $H^r(\mathbb{R}^\theta; \mathbb{C}^N)$, the assertion follows for $u \in \mathcal{D}(\R^{\theta};\C^N)$. If $u \in H^r(\R^\theta;\C^N)$, then there exists a sequence $(u_n)_{n \in \N}\subset  \D(\R^\theta;\C^N)$ such that $u_n \to u$ in $H^r(\R^\theta;\C^N)$ as $n \to \infty$.  Applying Proposition \ref{prop_shift_operator} yields
	\begin{equation*}
		\begin{split}
			\norm{f_{u}(t) - f_{u_n}(t)}_{H^r(\R^\theta;\C^N)}  &=  \norm{ \tau_{\delta t   } (u- u_n)}_{H^r(\R^\theta;\C^N)} \leq C \norm{u- u_n}_{H^r(\R^\theta;\C^N)}
		\end{split}
	\end{equation*}
	for all $n \in \N$ and $t \in (-1,1)$. Hence, $f_{u_n}(t) \to f_{u}(t)$ uniformly with respect to $t$ in $H^r(\R^\theta;\C^N)$ as $n \to \infty$.  Thus, $f_{u}$ is also  continuous. 
	
	It remains to verify the claim for $f_v^\pm$. Let $t_n, t \in (-1,1)$ such that $t_n \rightarrow t$ as $n \rightarrow \infty$.  Using the properties of Stein's extension operator $E$ and the above observations, we get that $f_{Ev}(t_n) \rightarrow f_{Ev}(t)$ in $H^r(\R^\theta;\C^N)$. Moreover, the boundedness of the restriction mapping gives us that
	$f_v^\pm(t_n)  = (\tau_{\delta t_n} E v)_\pm = (f_{Ev}(t_n))_\pm$ converges  to $(f_{Ev}(t))_\pm = f_v^{\pm}(t)$ in $H^r(\Omega_\pm;\C^N)$. This shows the continuity of $f_v^\pm$.
\end{proof}

\subsection{Convergence of $A_\varepsilon(z), B_\varepsilon(z)$, and $C_\varepsilon(z)$}\label{sec_conv_analyis}

This section is devoted to the convergence analysis of the operators $A_\varepsilon(z), B_\varepsilon(z)$, and $C_\varepsilon(z)$ introduced in~\eqref{def_ABC_operators} for $\varepsilon \rightarrow 0+$. First, in Proposition~\ref{prop_C_eps} we study the convergence of $C_\varepsilon(z)$, which allows us with a duality argument to investigate the convergence of $A_\varepsilon(z)$ in Proposition~\ref{prop_A_eps}. Eventually, in Proposition~\ref{prop_B_conv} we consider the  convergence of $B_\varepsilon(z)$.

We define
\begin{equation}\label{eq_eps_2}
  \varepsilon_2 := \min \left\{\frac{\varepsilon_1}{2}, \frac{1}{2\norm{D \nu}_{L^{\infty}(\R^\theta;\R^{\theta \times \theta})}} \right\},
\end{equation}
where $\varepsilon_1$ is specified Proposition~\ref{prop_tubular_neighbourhood}. 
Let $W$ be the Weingarten map associated with $\Sigma$ introduced in Definition~\ref{def_Weingarten}. In our analysis, the operator $M_{\varepsilon} : L^2((-1,1); L^2(\Sigma;\C^N)) \to L^2((-1,1); L^2(\Sigma;\C^N))$ acting as 
	\begin{equation} \label{def_M_eps}
		\begin{split}
			M_\varepsilon f(t) = \det \left(I- t \varepsilon W \right)f(t)\qquad \text{ for a.e. } t \in(-1,1) 
		\end{split}
	\end{equation}
will be useful. In the following lemma, which is an immediate consequence of Proposition~\ref{prop_tubular_neighbourhood}~(ii), some relevant properties of $M_\varepsilon$ are stated.

\begin{lemma}\label{lem_M_operators}	
	For any $\varepsilon \in (0, \varepsilon_2)$ the operator $M_\varepsilon$ is boundedly invertible, and $\norm{M_{\varepsilon}}_{0\to0} \leq (1+\varepsilon C)$ and $\norm{M_{\varepsilon} - I}_{0 \to 0} \leq \varepsilon C.$
\end{lemma}
% \begin{proof}
% 	The claims follow immediately from  the properties of $\det(I - t W(x_\Sigma)), (x_\Sigma,t) \in \Sigma \times (-\varepsilon_1,\varepsilon_1)$ given in  Proposition \ref{prop_tubular_neighbourhood} (ii).
% \end{proof}

To formulate the result about the convergence of $C_\varepsilon(z)$, recall that the embedding $\mathfrak{J}$ is defined in~\eqref{eq_def_embedding_op} and introduce the operator
	\begin{equation} \label{def_C_0}
		\begin{split}
			C_0( z):= \mathfrak{J} \Phi_{\overline{z}}^* :L^2(\R^\theta;\C^N) \to L^2((-1,1);L^2(\Sigma;\C^N)).
		\end{split}
	\end{equation}
In fact,  the properties of $\mathfrak{J}$ and $\Phi_{\overline{z}}^*$, see \eqref{eq_def_embedding_op} and Proposition \ref{prop_Phi_z}, imply that $C_0(z)$  gives rise also to a bounded operator from $L^2(\R^{\theta};\C^N)$ to $L^2((-1,1);H^{1/2}(\Sigma;\C^N))$.
\begin{proposition}\label{prop_C_eps}
	Let $z \in \rho(H)$, $R_ z = (H-  z)^{-1}$, $\tau_{(\cdot)}$ be the shift operator in~\eqref{def_tau_delta}, and $\varepsilon \in (0, \varepsilon_2)$ with $\varepsilon_2$ given by~\eqref{eq_eps_2}. Then, for any $u \in L^2(\R^{\theta};\C^N)$ the relation 
	\begin{equation}\label{eq_C_eps}
		C_\varepsilon( z)u(t) = \tr \tau_{\varepsilon t}R_ z u  \qquad \text{ for a.e. } t \in(-1,1)
	\end{equation} 
	holds in $L^2(\Sigma;\C^N)$ and $\ran C_\varepsilon( z) \subset L^2((-1,1);H^{1/2}(\Sigma;\C^N))$.  Moreover, the operators 
	$C_\varepsilon(z)$ are uniformly bounded from $L^2(\R^\theta;\C^N)$ to $L^2((-1,1);H^{1/2}(\Sigma;\C^N))$ and for any $r \in (0 ,\tfrac{1}{2})$ one has
	\begin{equation} \label{equation_convergence_C_eps}
		\norm{C_\varepsilon( z) -C_0( z) }_{L^2(\R^\theta;\C^N) \to 0} \leq C\varepsilon^{1/2 -r}.
	\end{equation}
\end{proposition}
\begin{proof}
  First, we show~\eqref{eq_C_eps} for $u \in \mathcal{D}(\mathbb{R}^\theta; \mathbb{C}^N)$. By density and continuity, this implies~\eqref{eq_C_eps} for all $u \in L^2(\mathbb{R}^\theta; \mathbb{C}^N)$. Recall that $R_z: H^s(\mathbb{R}^\theta; \mathbb{C}^N) \to H^{s+1}(\mathbb{R}^\theta; \mathbb{C}^N)$ is bounded for $s \in \R$; cf. Section~\ref{sec_free_Dirac}. Hence, by the Sobolev embedding theorem $R_ z u$ is continuous for $u \in \mathcal{D}(\mathbb{R}^\theta; \mathbb{C}^N)$ and the same is true for $\tau_{\varepsilon t}R_ z u$.
  Furthermore,  as $\tau_{\varepsilon t}R_ z u\in H^1(\mathbb{R}^\theta; \mathbb{C}^N)$ we conclude with Proposition~\ref{prop_ABC_operators} for $t \in (-1,1)$ and $x_\Sigma \in \Sigma$ that
  \begin{equation*}
    \tr \tau_{\varepsilon t}R_ z u(x_\Sigma) = \tau_{\varepsilon t}R_ z u(x_\Sigma) = \int_{\mathbb{R}^\theta} G_z(x_\Sigma + \varepsilon t\nu(x_\Sigma) - y) u(y) dy = C_\varepsilon(z) u(t)(x_\Sigma).
  \end{equation*}
  Hence, \eqref{eq_C_eps} is true.

  Next, we show the inclusion $\ran C_\varepsilon( z) \subset L^2((-1,1);H^{1/2}(\Sigma;\C^N))$. Assume that $u \in L^2(\R^\theta;\C^N)$. Then, 
  by Proposition~\ref{prop_continous_shift} and the boundedness of the trace $\tr:H^1(\R^\theta;\C^N) \to H^{1/2}(\Sigma;\C^N)$ it follows that the function $\tr \tau_{(\cdot) \varepsilon}R_ z u$ is continuous as a mapping from $ (-1,1)$ to  $H^{1/2}(\Sigma;\C^N)$. In particular, $\tr \tau_{(\cdot) \varepsilon}R_ z u$ is  measurable as  a mapping from $ (-1,1)$ to  $H^{1/2}(\Sigma;\C^N)$. Using again the boundedness of $\tr:H^1(\R^\theta;\C^N) \to H^{1/2}(\Sigma;\C^N)$ and the uniform boundedness of the shift operator in $H^1(\R^\theta;\C^N)$, see Propositions~\ref{proposition_trace_theorem}  and~\ref{prop_shift_operator}, respectively, we conclude
	\begin{equation*}
		\int_{-1}^1 \norm{\tr \tau_{\varepsilon t}R_ z u}^2_{H^{1/2}(\Sigma;\C^N)} \, dt \leq \int_{-1}^1  C\norm{R_ z u}^2_{H^1(\R^\theta;\C^N)} \, dt \leq C \norm{u}_{L^2(\R^\theta;\C^N)}^2
	\end{equation*}
	and therefore $\tr \tau_{(\cdot) \varepsilon}R_ z u \in L^2((-1,1);H^{1/2}(\Sigma;\C^N))$. Moreover, this also shows that $C_\varepsilon(z)$ is uniformly bounded from $L^2(\R^\theta;\C^N)$ to $L^2((-1,1);H^{1/2}(\Sigma;\C^N))$. 
	
	Eventually, with Proposition~\ref{prop_shift_operator} and the identity $C_0( z) = \mathfrak{J} \Phi_{\overline{z}}^* = \mathfrak{J}\tr R_ z$, see Proposition~\ref{prop_Phi_z}~(iii), we have for $r \in (0,\tfrac{1}{2})$ and $u \in L^2(\R^\theta;\C^N)$ 
	\begin{equation*}
		\begin{split}
		\norm{C_\varepsilon( z)u -C_0( z)u}_{0}^2  &=\norm{\tr \tau_{\varepsilon (\cdot)}R_ z u -\mathfrak{J}\tr R_ zu}_{0}^2\\
		&=\int_{-1}^1 \norm{\tr(\tau_{\varepsilon t}- I)R_ z u }_{L^2(\Sigma;\C^N)}^2 \,dt \\
% 		&\leq \int_{-1}^1 \norm{\tr(\tau_{\varepsilon t}- I)R_ z u }_{H^r(\Sigma;\C^N)}^2 \,dt\\
		&\leq C\int_{-1}^1 \norm{(\tau_{\varepsilon t}-I)R_ z u }_{H^{r+1/2}(\R^\theta;\C^N) }^2\,dt\\
		&\leq C\int_{-1}^1 \abs{\varepsilon t}^{1-2r}\norm{R_ z u }_{H^1(\R^\theta;\C^N)}^2 \, dt \\
		&\leq C\int_{-1}^1 \varepsilon^{1-2r}\norm{u }_{L^2(\R^\theta;\C^N)}^2 \,dt \\
		&\leq C \varepsilon^{1-2r}\norm{u }_{L^2(\R^\theta;\C^N)}^2,
		\end{split}
	\end{equation*}
	which leads to~\eqref{equation_convergence_C_eps}. Therefore, all claims are shown.
\end{proof}

Using the convergence of $C_\varepsilon(z)$, it is not difficult to show the convergence of $A_\varepsilon(z)$. We define the natural candidate for the limit operator by
\begin{equation*} %\label{def_A_0}
  A_0( z):=  \Phi_z  \mathfrak{J}^* : L^2((-1,1);L^2(\Sigma;\C^N)) \to L^2(\R^\theta;\C^N).
\end{equation*} 

\begin{proposition}\label{prop_A_eps}
	Let $z \in \rho(H)$ and $\varepsilon \in (0, \varepsilon_2)$ with $\varepsilon_2$ given by~\eqref{eq_eps_2}. Then, for any 
	$r \in (0,\tfrac{1}{2})$ one has
	\begin{equation*}
	  \norm{A_\varepsilon( z) - A_0( z)}_{0\to L^2(\R^\theta;\C^N)} \leq C \varepsilon^{1/2-r}
	\end{equation*}
	and, in particular, the operators $A_\varepsilon(z): L^2((-1,1);L^2(\Sigma;\C^N)) \to L^2(\R^\theta;\C^N)$ are uniformly bounded.
\end{proposition}
\begin{proof}
  Let $\mathcal{I}_\varepsilon$, $S_\varepsilon$, and $M_\varepsilon$ be the operators given by~\eqref{eq_I_eps},~\eqref{eq_S_eps}, and~\eqref{def_M_eps}, respectively. One verifies by a direct calculation using Proposition~\ref{prop_tubular_neighbourhood} (iii), \eqref{eq_I_eps}, and  \eqref{eq_S_eps}  that $(\mathcal{I}_\varepsilon \mathcal{S}_\varepsilon)^* = M_\varepsilon \mathcal{S}_\varepsilon^{-1} \mathcal{I}_\varepsilon^{-1}$. Using this relation we conclude from~\eqref{def_ABC_operators} that 
	\begin{equation*}	
		\begin{split}
		(A_\varepsilon(z) M_\varepsilon^{-1})^* &= M_\varepsilon^{-1} (A_\varepsilon(z))^* = M_\varepsilon^{-1} (R_z U_\varepsilon^* \mathcal{I}_\varepsilon \mathcal{S}_\varepsilon  )^* \\
		&= M_\varepsilon^{-1} M_\varepsilon \mathcal{S}_\varepsilon^{-1} \mathcal{I}_\varepsilon^{-1} U_\varepsilon R_{\overline{z}}  = \mathcal{S}_\varepsilon^{-1} \mathcal{I}_\varepsilon^{-1} U_\varepsilon R_{\overline{z}} = C_\varepsilon(\overline{z}).
		\end{split}
	\end{equation*}
	Moreover,  $(A_0(z))^* = (\Phi_z \mathfrak{J}^*)^* = \mathfrak{J}\Phi_z^* = C_0(\overline{z})$. Hence, Lemma~\ref{lem_M_operators} and Proposition~\ref{prop_C_eps} yield 
	\begin{equation*}
		\begin{split}
		&\norm{A_\varepsilon(z) - A_0(z)}_{0\to L^2(\R^\theta;\C^N)} \\
		&\qquad=  \norm{A_\varepsilon(z)M_\varepsilon^{-1}(M_\varepsilon - I)+A_\varepsilon(z)M_\varepsilon^{-1}- A_0(z) }_{0 \to L^2(\R^\theta;\C^N)} \\
		&\qquad\leq C \varepsilon \norm{C_\varepsilon(\overline{z}) }_{L^2(\R^\theta;\C^N) \to 0}  + \norm{C_\varepsilon(\overline{z})- C_0(\overline{z}) }_{L^2(\R^\theta;\C^N) \to 0} \\
		&\qquad\leq  C \varepsilon^{1/2 -r},
		\end{split}
	\end{equation*}
	which is the claimed estimate.
\end{proof}

\begin{remark}
For completeness we note that one can follow the proof of Proposition~\ref{prop_C_eps} to show 
	\begin{equation*} 
		\norm{C_\varepsilon( z) -C_0( z) }_{L^2(\R^\theta;\C^N) \to r} \leq C\varepsilon^{1/2 -r}
	\end{equation*}
	for $r \in (0 ,\tfrac{1}{2})$.
	Moreover, if $\Sigma$ is $C^3$ smooth, one can extend the result of Lemma~\ref{lem_M_operators} to spaces $L^2((-1,1);H^t(\Sigma;\C^N))$, $t \in [-1,1]$, and use this to verify
	\begin{equation*}
	  \norm{A_\varepsilon( z) - A_0( z)}_{-r\to L^2(\R^\theta;\C^N)} \leq C \varepsilon^{1/2-r}
	\end{equation*}
	for $r \in (0,\tfrac{1}{2})$ in a similar way as in the proof of Proposition~\ref{prop_A_eps}. 
\end{remark}

Next, we study the convergence of the operators $B_\varepsilon(z)$. Define the limit operator by \begin{equation*}
	B_0( z): L^2((-1,1);L^2(\Sigma;\C^N)) \to L^2((-1,1);L^2(\Sigma;\C^N))
\end{equation*} 
which acts on $f \in L^2((-1,1);L^2(\Sigma;\C^N))$ evaluated for a.e. $t \in (-1,1)$ as 
	\begin{equation}\label{eq_B_0}
		\begin{split}
		B_0( z) f(t)&:= \frac{i}{2}(\alpha \cdot \nu)\int_{-1}^1 \sign(t-s)f(s) \,ds + \mathcal{C}_z \int_{-1}^1  f(s) \, ds,
		\end{split}
	\end{equation} 
where $\mathcal{C}_z: L^2(\Sigma;\C^N) \to L^2(\Sigma;\C^N)$ is the extension of the operator defined in \eqref{def_C_z} from Proposition \ref{proposition_C_z}. Using the mapping properties of $\mathcal{C}_z$ in Proposition \ref{proposition_C_z}~(i) it follows that 
$B_0( z)$ can also be regarded as an operator in $L^2((-1,1);H^r(\Sigma;\C^N))$ for any $r\in [-\frac{1}{2},\frac{1}{2}]$.
In the following proposition we show that $B_\varepsilon(z)$ converges to $B_0(z)$. The proof of this result is more complicated as the proofs of Proposition~\ref{prop_C_eps} and Proposition~\ref{prop_A_eps}, and therefore some of the more technical calculations are shifted to Appendix~\ref{sec_appendix_diff_B}.

\begin{proposition}\label{prop_B_conv} 
Let $z \in \rho(H)$ and $\varepsilon \in (0, \varepsilon_2)$ with $\varepsilon_2$ given by~\eqref{eq_eps_2}. Then, 
the operators $B_\varepsilon(z)$ are uniformly bounded in $L^2((-1,1);L^2(\Sigma;\C^N))$ and
for any $r \in (0,\tfrac{1}{2})$ one has
\begin{equation*}
	\norm{B_\varepsilon( z) - B_0( z)}_{1/2 \to  0} \leq C \varepsilon^{1/2 -r}.
\end{equation*}
\end{proposition}

\begin{proof}
  The proof is split into several steps. Let $\Phi_z$ be as in~\eqref{def_Phi_z} and let $\tau_{(\cdot)}^{\Omega_\pm}$ be defined by~\eqref{def_tau_delta_pm}. We introduce the auxiliary operators 
  \begin{equation}\label{wtwtb}
  \widetilde{B}_\varepsilon(z):= B_\varepsilon(z) M_\varepsilon^{-1}: L^2((-1,1);L^2(\Sigma;\C^N)) \to L^2((-1,1);L^2(\Sigma;\C^N)),
  \end{equation}
  which are, due to the properties of $B_\varepsilon(z)$ and $M_\varepsilon$ in~\eqref{def_ABC_operators} and~\eqref{def_M_eps}, bounded and
  act on $ f \in L^2((-1,1);L^2(\Sigma;\C^N))$ evaluated at $t \in (-1,1)$ and $x_\Sigma \in \Sigma$ as
  \begin{equation} \label{def_B_tilde}
    \widetilde B_\varepsilon( z)f(t)(x_\Sigma) =  \int_{-1}^1 \int_{\Sigma} G_ z(x_\Sigma +\varepsilon t\nu(x_\Sigma) -y_\Sigma - \varepsilon s\nu(y_\Sigma))   f(s)(y_\Sigma) \, d\sigma(y_\Sigma) \, ds.
  \end{equation}  
  Moreover, we define 
  $$\overline{B}_\varepsilon( z): L^2((-1,1);H^{1/2}(\Sigma;\C^N)) \to  L^2((-1,1);H^{1/2}(\Sigma;\C^N))$$ 
  acting on $ f \in L^2((-1,1);H^{1/2}(\Sigma;\C^N))$ for a.e. $t \in (-1,1)$ as
  \begin{equation}\label{def_B_bar}
		\overline{B}_\varepsilon( z)f(t) =  \int_{-1}^t \tr^- \tau_{\varepsilon (t-s)  }^{\Omega_-}(\Phi_zf(s))_- ds + \int_t^1 \tr^+ \tau_{\varepsilon(t-s) }^{\Omega_+}(\Phi_zf(s))_+ \,ds.
	\end{equation} 
  First, in \textit{Step~1} we show that $\overline{B}_\varepsilon(z)$ is bounded and converges to $B_0(z)$, then in \textit{Step~2} we verify an alternative representation of $\overline{B}_\varepsilon(z)$. In \textit{Step~3} we use Appendix~\ref{sec_appendix_diff_B} to compare $\overline{B}_\varepsilon(z)$ and $\widetilde{B}_\varepsilon(z)$, and show that $\widetilde{B}_\varepsilon(z)$ is uniformly bounded in~$\varepsilon$. In \textit{Step~4} we combine the results from \textit{Step~1} to \textit{Step~3} to conclude the claim of this proposition.

  \textit{Step~1.} First, we note that, due to Definition~\ref{def_Bochner_measurable} and Proposition~\ref{prop_continous_shift}, the function 
  \begin{equation*}
  	(-1,1)^2 \ni (t,s) \mapsto \Theta(\mp(t - s)) \tau_{\varepsilon(t-s)}^{\Omega_\pm} \in \mathcal{L}(H^1(\Omega_\pm; \mathbb{C}^N),H^1(\Omega_\pm; \mathbb{C}^N))
  \end{equation*} is measurable, where $\Theta$ is the Heaviside function. Hence, it follows with Definition~\ref{def_Bochner_measurable} that the integrands in \eqref{def_B_bar} are  measurable with respect to $(t,s) \in (-1,1)^2$. Moreover, by the mapping properties of $\tr^\pm$, $\tau_{(\cdot)}^{\Omega_\pm}$, and  $\Phi_z$ in Proposition~\ref{proposition_trace_theorem}, Corollary~\ref{cor_shift_operator},  and Proposition~\ref{prop_Phi_z}, respectively, the  integrands are bounded by $C \norm{f(s)}_{H^{1/2}(\Sigma;\C^N)}$ for $(t,s) \in (-1,1)^2$. In particular, we conclude  for $f \in L^2((-1,1);H^{1/2}(\Sigma;\C^N))$ that
  \begin{equation*}
    \int_{-1}^1 \int_{-1}^1 \| \Theta(t-s) \tr^- \tau_{\varepsilon (t-s)  }^{\Omega_-}(\Phi_zf(s))_- + \Theta(s-t) \tr^+ \tau_{\varepsilon(t-s) }^{\Omega_+}(\Phi_zf(s))_+ \|^2_{H^{1/2}(\Sigma; \mathbb{C}^N)} dt ds 
  \end{equation*}
  is finite.
  Thus, Fubini's theorem for Bochner integrals, cf. Section~\ref{sec_Bochner}, yields the integrability of the integrands in \eqref{def_B_bar} with respect to $s \in (-1,1)$ and  the measurability of $t \mapsto \overline{B}_\varepsilon(z)f(t)$. Furthermore, the bound for the integrands implies also the inequality $\norm{\overline{B}_\varepsilon(z) f}_{1/2} \leq C \norm{f}_{1/2}$ for $f \in L^2((-1,1);H^{1/2}(\Sigma;\C^N))$. Hence, $\overline{B}_\varepsilon(z)$ is well-defined and uniformly bounded in $L^2((-1,1);H^{1/2}(\Sigma;\C^N))$. We claim that
  \begin{equation} \label{convergence_B_bar}
		\norm{\overline{B}_\varepsilon( z) - {B}_{0}( z)}_{1/2 \to 0} \leq C {\varepsilon}^{1/2-r}.
	\end{equation}
	To see this we remark that with Proposition~\ref{proposition_C_z}~(ii) we have the pointwise representation
	\begin{equation} \label{representation_B_0}
		 B_0( z)f(t) = \int_{-1}^t  \tr^- (\Phi_ z  f(s))_- \,ds +  \int_{t}^1  \tr^+ (\Phi_ z  f(s))_+ \,ds 
	\end{equation}
	for a.e. $t \in (-1,1)$ and $f \in L^2((-1,1);H^{1/2}(\Sigma;\C^N))$. 
	Thus, $r \in (0,\tfrac{1}{2})$ and direct estimates show
	\begin{equation}\label{eq_convergence_B_bar}
		\begin{split}
		&\norm{\overline{B}_\varepsilon( z)f - {B}_0( z)f}_{0}^2 = \int_{-1}^1 \biggl\| \int_{-1}^t \tr^- (\tau_{\varepsilon (t-s)  }^{\Omega_-}- I)(\Phi_z f(s))_- ds \\
		&\qquad \qquad \qquad \qquad \qquad \qquad +  \int_t^1 \tr^+ (\tau_{\varepsilon (t-s) }^{\Omega_+}-I)( \Phi_z f(s))_+ ds\, \biggr\|_{L^2(\Sigma;\C^N)}^2 \, dt \\
		&\qquad \qquad \qquad\leq \int_{-1}^1 \biggl(\int_{-1}^t \norm{\tr^- (\tau_{\varepsilon (t-s) }^{\Omega_-}- I)( \Phi_z f(s))_-}_{H^{r}(\Sigma;\C^N)} ds  \\
		&\qquad \qquad \qquad\qquad \qquad \qquad+ \int_t^1 \norm{\tr^+ (\tau_{\varepsilon (t-s)  }^{\Omega_+}-I)(\Phi_z f(s))_+}_{H^{r}(\Sigma;\C^N)} ds\biggr)^2 \,dt. 
		\end{split}
	\end{equation}
	Employing  Proposition \ref{proposition_trace_theorem}, Corollary~\ref{cor_shift_operator}, and Proposition~\ref{prop_Phi_z} yields for all $s,t \in (-1,1)$
	\begin{equation*}
		\norm{\tr^\pm (\tau_{\varepsilon (t-s) }^{\Omega_\pm}- I)( \Phi_z f(s))_\pm}_{H^{r}(\Sigma;\C^N)} \leq  C \varepsilon^{1/2-r}\norm{f(s)}_{H^{1/2}(\Sigma;\C^N)} .
	\end{equation*}
	Plugging this into \eqref{eq_convergence_B_bar} we obtain
	\begin{equation*}
	\begin{split}
		\|\overline{B}_\varepsilon&( z)f - {B}_0( z)f\|_{0}^2 \\
		& 
		\leq C \varepsilon^{2(1/2-r)} \int_{-1}^1 \biggl(\int_{-1}^t \norm{f(s)}_{H^{1/2}(\Sigma;\C^N)}\, ds + \int_t^1 \norm{f(s)}_{H^{1/2}(\Sigma;\C^N)} \,ds\biggr)^2 \,dt\\
		&\leq C\varepsilon^{2(1/2-r)} \int_{-1}^{1} \norm{f(s)}_{H^{1/2}(\Sigma;\C^N)}^2\,ds = C \varepsilon^{2(1/2-r)}\norm{f}_{1/2}^2,
		\end{split}
	\end{equation*}
	which implies~\eqref{convergence_B_bar}.

  \textit{Step~2.} We show that the operator $\overline{B}_\varepsilon(z)$ in~\eqref{def_B_bar} has the alternative representation  
    \begin{equation} \label{eq_B_bar}
    \overline B_\varepsilon( z)f(t)(x_\Sigma) =  \int_{-1}^1 \int_{\Sigma} G_ z(x_\Sigma +\varepsilon (t-s)  \nu(x_\Sigma) -y_\Sigma)   f(s)(y_\Sigma) \, d\sigma(y_\Sigma) \, ds
  \end{equation}
  for $f \in L^2((-1,1);H^{1/2}(\Sigma;\C^N))$, a.e. $t \in (-1,1)$, and $\sigma$-a.e. $x_\Sigma \in \Sigma$.
  Let $f \in L^2((-1,1);H^{1/2}(\Sigma;\C^N))$ and $t,s \in (-1,1)$  be fixed such that $t > s $. Note that the choice of $\varepsilon_1$ in Proposition~\ref{prop_tubular_neighbourhood} implies $\varepsilon_2 \leq \frac{\varepsilon_1}{2} < \frac{\varepsilon_A}{2}$, cf. also Proposition~\ref{prop_iota}. Hence, by Corollary~\ref{cor_eps_A} we have   
  $x_\Sigma + \varepsilon (t-s) \nu(x_\Sigma) \in \Omega_-$ for all $x_\Sigma \in \Sigma$. Moreover, we conclude from the representation of $\Phi_z$ given in \eqref{def_Phi_z} and the form of the integral kernel $G_z$, see~\eqref{eq_G_z_2D}--\eqref{eq_G_z_3D}, that $\Phi_z  f(s)$ is continuous away from $\Sigma$. Thus, we have for $\sigma$-a.e. $x_\Sigma \in \Sigma$
	\begin{equation*}
	\begin{split}
		\tr^- \tau_{\varepsilon (t-s) }^{\Omega_-}(\Phi_zf(s))_-(x_\Sigma) &= (\Phi_ z  f(s))(x_\Sigma + \varepsilon (t-s)  \nu(x_\Sigma)) \\
		&= \int_\Sigma G_ z(x_\Sigma + \varepsilon (t-s) \nu(x_\Sigma) -y_\Sigma) f(s)(y_\Sigma)\,d\sigma(y_\Sigma).
		\end{split}
	\end{equation*}
	Analogously, for $t <s$ and $\sigma$-a.e. $x_\Sigma \in \Sigma$
	\begin{equation*}
	\begin{split}
			\tr^+ \tau_{\varepsilon (t-s) }^{\Omega_+}(\Phi_zf(s))_+(x_\Sigma) &= (\Phi_ z  f(s))(x_\Sigma + \varepsilon (t-s)  \nu(x_\Sigma))\\
			&= \int_\Sigma G_ z(x_\Sigma + \varepsilon (t-s) \nu(x_\Sigma) -y_\Sigma) f(s)(y_\Sigma) \,d\sigma(y_\Sigma).
			\end{split}
	\end{equation*}
	 Combining the previous two equations yields 
	\begin{equation}\label{eq_B_bar_pointwise}
	\begin{split}
		 \int_{-1}^t &\tr^- \tau_{\varepsilon (t-s) }^{\Omega_-}(\Phi_zf(s))_-(x_\Sigma) \, ds + \int_t^1 \tr^+ \tau_{\varepsilon (t-s) }^{\Omega_+}(\Phi_zf(s))_+(x_\Sigma) \, ds \\
		 &= \int_{-1}^1 \int_{\Sigma} G_ z(x_\Sigma +\varepsilon (t-s) \nu(x_\Sigma) -y_\Sigma)   f(s)(y_\Sigma) \, d\sigma(y_\Sigma) \, ds.
		 \end{split}
	\end{equation}
	Moreover, as  the integrands on the right hand side in~\eqref{def_B_bar} are Bochner integrable (cf. \textit{Step~1}), Proposition~\ref{prop_Bochner}~(iii) shows  that the pointwise evaluation of the Bochner integrals in the definition of $\overline{B}_\varepsilon(z)$ in \eqref{def_B_bar} coincides with \eqref{eq_B_bar_pointwise}, i.e.
	\begin{equation*}
	\begin{split}
	&\bigg(\int_{-1}^t  \tr^- \tau_{\varepsilon (t-s) }^{\Omega_-}(\Phi_ z  f(s))_- \,ds +  \int_{t}^1  \tr^+ \tau_{\varepsilon (t-s) }^{\Omega_+}(\Phi_ z  f(s))_+ \,ds \bigg)(x_\Sigma)  \\
	&\qquad = \int_{-1}^t \tr^- \tau_{\varepsilon (t-s) }^{\Omega_-}(\Phi_zf(s))_-(x_\Sigma) \, ds + \int_t^1 \tr^+ \tau_{\varepsilon (t-s) }^{\Omega_+}(\Phi_zf(s))_+(x_\Sigma) \, ds.
	\end{split}
	\end{equation*} 
	This is exactly the claimed formula in \eqref{eq_B_bar}.

% 	Next, we claim that $\overline{B}_\varepsilon(z)$ has a bounded extension in $L^2((-1,1);L^2(\Sigma;\C^N))$ which is uniformly bounded in $\varepsilon$. Indeed, by duality and formal symmetry of $G_z$ in   

  \textit{Step~3.} By the results in Appendix~\ref{sec_appendix_diff_B} the map 
  $\widetilde{B}_\varepsilon( z) - \overline{B}_\varepsilon(z)$ admits an extension to a bounded operator from $L^2((-1,1);L^2(\Sigma;\C^N))$ to $L^2((-1,1);H^{1/2}(\Sigma;\C^N))$ and
  \begin{equation} \label{ESTIMATE_B_BAR_B_TILDE}
			\|\widetilde{B}_\varepsilon( z) - \overline{B}_\varepsilon(z)\|_{0 \to 0} \leq   \|\widetilde{B}_\varepsilon( z) - \overline{B}_\varepsilon(z)\|_{0 \to 1/2} \leq C (\varepsilon + \varepsilon \abs{\log(\varepsilon)})^{1/2}.
	\end{equation}
	Moreover, we claim that $\widetilde{B}_\varepsilon(z)$ is uniformly bounded in $L^2((-1,1);L^2(\Sigma;\C^N))$. To see this, observe first that 
	\begin{equation}\label{klkl}
	\begin{split}
	  \|\widetilde{B}_\varepsilon( z)\|_{1/2 \to 1/2} &\leq 
	\|\widetilde{B}_\varepsilon( z) - \overline{B}_\varepsilon(z)\|_{1/2 \to 1/2}  + 
	\|\overline{B}_\varepsilon(z)\|_{1/2 \to 1/2} \\
	&\leq 
	\|\widetilde{B}_\varepsilon( z) - \overline{B}_\varepsilon(z)\|_{0 \to 1/2}  + 
	\|\overline{B}_\varepsilon(z)\|_{1/2 \to 1/2}.
	\end{split}
	\end{equation}
   Therefore, the estimate \eqref{ESTIMATE_B_BAR_B_TILDE} and the uniform boundedness of the operators $\overline{B}_\varepsilon(z)$ in $L^2((-1,1);H^{1/2}(\Sigma;\C^N))$
	shown in \textit{Step~1} imply that $\widetilde{B}_\varepsilon(z)$ is also uniformly bounded in $L^2((-1,1); H^{1/2}(\Sigma;\C^N))$. 
	The same is true for $\widetilde{B}_\varepsilon(\overline z)$ and hence also the anti-dual 
	\begin{equation*}%\label{anitb}
	(\widetilde{B}_\varepsilon(\overline z))': L^2((-1,1);H^{-1/2}(\Sigma;\C^N)) \to L^2((-1,1);H^{-1/2}(\Sigma;\C^N))
	\end{equation*}
	is uniformly bounded. We claim that $(\widetilde{B}_\varepsilon(\overline{z}))'$  
	is an extension of $\widetilde{B}_\varepsilon(z)$, that is, 
	\begin{equation}\label{extb}
	\widetilde{B}_\varepsilon(z)f=(\widetilde{B}_\varepsilon(\overline z))'f,\qquad f\in L^2((-1, 1); L^2(\Sigma; \mathbb{C}^N)).
	\end{equation} 
	The identity  $(\mathcal{I}_\varepsilon \mathcal{S}_\varepsilon)^* = M_\varepsilon \mathcal{S}_\varepsilon^{-1} \mathcal{I}_\varepsilon^{-1}$ and \eqref{def_ABC_operators} yield for the adjoint of  $B_\varepsilon(\overline z)$  in $L^2((-1,1);L^2(\Sigma;\C^N))$
	\begin{equation}\label{eq_B_adjoint}
		(B_\varepsilon( \overline z))^* =    (\mathcal{S}_\varepsilon^{-1}\mathcal{I}_\varepsilon^{-1} U_\varepsilon R_{\overline z} U_\varepsilon^* \mathcal{I}_\varepsilon \mathcal{S}_\varepsilon )^* = M_\varepsilon \mathcal{S}_\varepsilon^{-1}\mathcal{I}_\varepsilon^{-1} U_\varepsilon R_{ z} U_\varepsilon^* \mathcal{I}_\varepsilon \mathcal{S}_\varepsilon M_\varepsilon^{-1} = M_\varepsilon {B}_\varepsilon( z)M_\varepsilon^{-1}.
	\end{equation}
	In turn, $M_\varepsilon = (M_\varepsilon)^*$ and $\widetilde{B}_\varepsilon(z) = B_\varepsilon(z) M_\varepsilon^{-1}$ give us 
	\begin{equation*}
		(\widetilde{B}_\varepsilon(\overline z))^* =   	(B_\varepsilon(\overline z)M_\varepsilon^{-1})^* = M_{\varepsilon}^{-1}(B_\varepsilon( \overline z))^* =  M_{\varepsilon}^{-1} M_\varepsilon {B}_\varepsilon( z)M_\varepsilon^{-1} = \widetilde{B}_{\varepsilon}(z)
	\end{equation*}
	%With this, it follows that the anti-dual $(\widetilde{B}_\varepsilon(z))': L^2((-1,1);H^{-1/2}(\Sigma;\C^N)) \to L^2((-1,1);H^{-1/2}(\Sigma;\C^N))$ of the operator $\widetilde{B}_\varepsilon(z): L^2((-1,1);H^{1/2}(\Sigma;\C^N)) \to L^2((-1,1);H^{1/2}(\Sigma;\C^N))$ is an extension of $\widetilde{B}_\varepsilon(\overline{z})$. 
	and hence 
	Proposition~\ref{prop_Bochner}~(ii) implies for $f \in L^2((-1, 1); L^2(\Sigma; \mathbb{C}^N))$ and $g \in L^2((-1, 1); H^{1/2}(\Sigma; \mathbb{C}^N))$
	\begin{equation*}
	  \begin{split}
	    &\langle (\widetilde{B}_\varepsilon(\overline z))' f, g \rangle_{L^2((-1, 1); H^{-1/2}(\Sigma; \mathbb{C}^N)) \times L^2((-1, 1); H^{1/2}(\Sigma; \mathbb{C}^N))} \\
	    &\quad= \langle  f, \widetilde{B}_\varepsilon(\overline z) g \rangle_{L^2((-1, 1); H^{-1/2}(\Sigma; \mathbb{C}^N)) \times L^2((-1, 1); H^{1/2}(\Sigma; \mathbb{C}^N))} \\
	    &\quad= (  f, \widetilde{B}_\varepsilon(\overline z) g )_{L^2((-1, 1); L^2(\Sigma; \mathbb{C}^N)) } = ( \widetilde{B}_\varepsilon(z) f,  g )_{L^2((-1, 1); L^2(\Sigma; \mathbb{C}^N)) } \\
	    &\quad = \langle \widetilde{B}_\varepsilon(z) f, g \rangle_{L^2((-1, 1); H^{-1/2}(\Sigma; \mathbb{C}^N)) \times L^2((-1, 1); H^{1/2}(\Sigma; \mathbb{C}^N))},
	  \end{split}
	\end{equation*}	
	where $\langle \cdot, \cdot\rangle_{L^2((-1, 1); H^{-1/2}(\Sigma; \mathbb{C}^N)) \times L^2((-1, 1); H^{1/2}(\Sigma; \mathbb{C}^N))}$ denotes the sesquilinear duality product, which is anti-linear in the second argument.
	This implies \eqref{extb} and since $(\widetilde{B}_\varepsilon(\overline z))'$ and $\widetilde{B}_\varepsilon(z)$ are both uniformly bounded
	in $L^2((-1,1); H^{-1/2}(\Sigma;\C^N))$ and $L^2((-1,1); H^{1/2}(\Sigma;\C^N))$, respectively, an interpolation argument leads to the uniform boundedness of $\widetilde{B}_\varepsilon(z)$ in $L^2((-1,1);L^2(\Sigma;\C^N))$.

	\textit{Step~4.} Using the results from \textit{Step~1} to \textit{Step~3} we will now complete the proof of Proposition~\ref{prop_B_conv}. Since $B_\varepsilon(z) = \widetilde{B}_{\varepsilon}(z)M_\varepsilon$, Lemma \ref{lem_M_operators} and  the uniform boundedness of $\widetilde{B}_\varepsilon(z)$ shown in \textit{Step~3} imply the uniform boundedness of $B_\varepsilon(z)$ in $L^2((-1,1);L^2(\Sigma;\C^N))$, proving the first claim of this proposition. Moreover, the uniform boundedness of $\widetilde{B}_\varepsilon(z)$, \eqref{ESTIMATE_B_BAR_B_TILDE} and Lemma~\ref{lem_M_operators} show that $\overline{B}_\varepsilon(z)$ also acts  as a uniformly bounded operator in $L^2((-1,1); L^2(\Sigma;\C^N))$ and
	\begin{equation}\label{eq_diff_B_B_bar}
	\begin{split}
		\norm{B_\varepsilon( z) - \overline{B}_\varepsilon(z)}_{0 \to 0} &\leq \norm{\widetilde{B}_\varepsilon( z) (M_\varepsilon - I)}_{0 \to 0} + \norm{\widetilde{B}_\varepsilon( z) - \overline{B}_\varepsilon(z)}_{0 \to 0} \\
		&\leq C\bigl(\varepsilon  + (\varepsilon+\varepsilon \abs{\log (\varepsilon)})^{1/2} \bigr) \leq C (\varepsilon + \varepsilon\abs{\log (\varepsilon)})^{1/2}.
		\end{split}
	\end{equation}
	Combining \eqref{eq_diff_B_B_bar} with \eqref{convergence_B_bar} yields 
	\begin{equation*}
		\begin{split}
		\norm{B_\varepsilon( z) - B_0( z)}_{1/2 \to  0}  & \leq  \norm{B_\varepsilon( z) - \overline{B}_\varepsilon(z)}_{1/2 \to 0} +\norm{ \overline{B}_\varepsilon(z) - B_0(z) }_{1/2 \to 0}  \\
		&\leq  \norm{ B_\varepsilon( z) - \overline{B}_\varepsilon(z)}_{0 \to 0} +\norm{\overline{B}_\varepsilon(z) - B_0(z) }_{1/2 \to 0}  \\
		&\leq C\bigl(  ( \varepsilon+\varepsilon\abs{\log(\varepsilon)})^{1/2} + \varepsilon^{1/2 -r}\bigr) \leq C \varepsilon^{1/2-r}.
		\end{split}
	\end{equation*}
	This is the claimed norm estimate and finishes the proof of this proposition.
\end{proof}

\begin{remark}
	Note that by combining \eqref{eq_B_adjoint} and Lemma~\ref{lem_M_operators} we obtain for $\varepsilon \in (0,\varepsilon_2)$ the  estimate
	\begin{equation*}
	\begin{split}
		\norm{B_\varepsilon(z) - (B_{\varepsilon}(\overline z))^*}_{0\to0} &= \norm{B_\varepsilon(z) - M_{\varepsilon}B_{\varepsilon}( z)M_{\varepsilon}^{-1}}_{0\to0} \\
		&\leq \norm{(I-M_\varepsilon)B_\varepsilon(z)}_{0\to0} + \norm{M_{\varepsilon}B_{\varepsilon}( z)M_{\varepsilon}^{-1}(M_{\varepsilon}-I)}_{0\to0} \\
		&\leq C \varepsilon.
		\end{split}
	\end{equation*}
\end{remark}

\subsection{Convergence of $(I + {B_\varepsilon( z)}Vq)^{-1}$} \label{sec_inverse}

Recall that $B_\varepsilon(z)$ is defined in~\eqref{def_ABC_operators} and let $q$ and $V$ be as in~\eqref{eq_q} and~\eqref{eq_V}, respectively. In this section we treat the convergence of $(I + {B_\varepsilon( z)}Vq)^{-1}$ under suitable assumptions on $q$ and $V$. 
The crucial assumption here, which is also needed in our main result, Theorem~\ref{THEO_MAIN}, is that $Vq$ is bounded by 
\begin{equation} \label{def_X_z}
X_z := \frac{1}{\max \{ \sup_{ \varepsilon \in (0,\varepsilon_2)} \norm{B_\varepsilon(z) }_{0 \to 0} , \norm{B_0(z) }_{1/2 \to 1/2}\}},
\end{equation}
where $\varepsilon_2$ is given by~\eqref{eq_eps_2} and $z \in \rho(H)$ is fixed. After the convergence result in the next proposition we 
provide more information on $X_z$ in Lemma~\ref{lemma_smallness_condition}.

\begin{proposition}\label{prop_I+B_conv}
Let $z \in \rho(H)$, $\varepsilon_3 \in (0, \varepsilon_2]$, and $\varepsilon_2>0$ be given by ~\eqref{eq_eps_2}, and let $q$ and $V$ be as in \eqref{eq_q} and \eqref{eq_V}, respectively, such that the following conditions are fulfilled:
	\begin{itemize}
		\item[(i)] The operators $I + B_\varepsilon(z)Vq$ are bijective in $L^2((-1,1);L^2(\Sigma;\C^N))$ for $\varepsilon \in (0,\varepsilon_3)$ and 
		their inverses are uniformly bounded. 
		\item[(ii)] The operator $I+ B_0(z)Vq$ is bijective in $L^2((-1,1);H^{1/2}(\Sigma;\C^N))$.
		% 		\item[(iii)] Since $L^2((-1,1);H^{1/2}(\Sigma;\C^N))$ is dense in $L^2((-1,1);L^2(\Sigma;\C^N))$, Proposition~\ref{prop_B_conv} implies that also $\norm{B_0(z) }_{0 \to 0} < \frac{1}{X_z}$ and hence, $I+ B_0(z)Vq$ is also bijective in $L^2((-1,1);L^2(\Sigma;\C^N))$.
	\end{itemize}
	Then, for any $r \in (0,\tfrac{1}{2})$ one has
	\begin{equation*}
		\norm{(I + B_{\varepsilon}( z)Vq)^{-1} -(I +B_{0}( z)Vq)^{-1}}_{1/2 \to 0} \leq C \varepsilon^{1/2-r}, \qquad  \varepsilon \in (0, \varepsilon_3).
	\end{equation*}
	In particular, this is true if  $\varepsilon_3 = \varepsilon_2$ and 
	\begin{equation}\label{eq_smallnes_assumption}
		\norm{V}_{W^1_\infty(\Sigma;\C^{N\times N})}  \norm{q}_{L^\infty(\R;\R)} < X_z.
	\end{equation}
\end{proposition}

%\begin{remark}\label{rem_smallnes_assumption_2}
%	Note that assumption \eqref{eq_smallnes_assumption} implies the following:  
%	\begin{itemize}
%		\item[(i)] The operators $I + B_\varepsilon(z)Vq$ are bijective in $L^2((-1,1);L^2(\Sigma;\C^N))$ for $\varepsilon \in (0,\varepsilon_2)$ and 
%		their inverses are uniformly bounded. 
%		\item[(ii)] The operator $I+ B_0(z)Vq$ is bijective in $L^2((-1,1);H^{1/2}(\Sigma;\C^N))$.
%% 		\item[(iii)] Since $L^2((-1,1);H^{1/2}(\Sigma;\C^N))$ is dense in $L^2((-1,1);L^2(\Sigma;\C^N))$, Proposition~\ref{prop_B_conv} implies that also $\norm{B_0(z) }_{0 \to 0} < \frac{1}{X_z}$ and hence, $I+ B_0(z)Vq$ is also bijective in $L^2((-1,1);L^2(\Sigma;\C^N))$.
%	\end{itemize}
%	The statement of Proposition~\ref{prop_I+B_conv} remains true, if~\eqref{eq_smallnes_assumption} is replaced by the weaker assumptions (i) and (ii).
%\end{remark}

\begin{proof}
	Since the conditions (i) and (ii) are fulfilled, we have 
	\begin{equation*}
	\begin{split}
		((I+B_0( z)Vq)^{-1} &- (I+B_\varepsilon( z)Vq)^{-1} ) f \\
		&=(I+B_\varepsilon( z)Vq)^{-1}(B_\varepsilon( z) -B_0( z))Vq(I+B_0( z)Vq)^{-1}f
		\end{split}
	\end{equation*}
	for $f \in L^2((-1,1);H^{1/2}(\Sigma;\C^N))$ and $\varepsilon \in (0,\varepsilon_3)$. 
	Thus, with  Proposition \ref{prop_B_conv} and the fact that $Vq$ 
	induces a bounded operator in  $L^2((-1,1);H^{1/2}(\Sigma;\C^N))$, see \eqref{W_1_infty_mult_op} and Section \ref{sec_Bochner}, we find
	\begin{equation*}
		\begin{split}
		&\norm{(I+B_0( z)Vq)^{-1} - (I+B_\varepsilon( z)Vq)^{-1}}_{1/2 \to 0} \\
		&\quad\leq\norm{(I+B_\varepsilon( z)Vq)^{-1}}_{0 \to 0} \norm{(B_\varepsilon( z) -B_0( z))Vq}_{1/2 \to 0}\norm{(I+B_0( z)Vq)^{-1}}_{1/2 \to 1/2}\\
		&\quad\leq C \norm{B_0( z) -B_\varepsilon( z)}_{1/2 \to 0} \leq C \varepsilon^{1/2-r}, \qquad  \varepsilon \in (0, \varepsilon_3),
		\end{split}
	\end{equation*}
	which is the claimed result. Finally, note that if \eqref{eq_smallnes_assumption} is true, then 
	\begin{equation*}
		\begin{split}
		\max \Bigl\{& \sup_{ \varepsilon \in (0,\varepsilon_2)} \norm{B_\varepsilon(z)Vq}_{0 \to 0} , \norm{B_0(z)Vq}_{1/2 \to 1/2} \Bigr\} \\
		&\leq  \max \Bigl\{ \sup_{ \varepsilon \in (0,\varepsilon_2)} \norm{B_\varepsilon(z)}_{0 \to 0} , \norm{B_0(z)}_{1/2 \to 1/2} \Bigr\} \norm{V}_{W^1_\infty(\Sigma;\C^{N\times N})}  \norm{q}_{L^\infty(\R;\R)} \\
		&= \frac{1}{X_z} \norm{V}_{W^1_\infty(\Sigma;\C^{N\times N})}  \norm{q}_{L^\infty(\R;\R)} <1
		\end{split}
	\end{equation*}
	shows that (i) and (ii) are fulfilled.
\end{proof}

In the following lemma we give more explicit estimates for the constant $X_z$ in~\eqref{def_X_z}. Recall that $E$ denotes Stein's extension operator. Moreover, we introduce $C_{1/2, -1/2} := \frac{C_2}{C_1} \geq  1$, 
where $C_2\geq C_1 >0$ are chosen such that 
\begin{equation*}
	C_1 \norm{f}_{0} \leq  \norm{f}_{[L^2((-1,1);H^{-1/2}(\Sigma;\C^N)),L^2((-1,1);H^{1/2}(\Sigma;\C^N))]_{1/2}} \leq C_2 \norm{f}_{0}
\end{equation*}
holds for $f \in L^2((-1,1);L^2(\Sigma;\C^N))$; note that according to Proposition~\ref{prop_Bochner}~(i) such a choice is possible. 
If $B$ is a bounded operator in $L^2((-1,1);L^2(\Sigma;\C^N))$ that admits a bounded extension to
$L^2((-1,1);H^{-1/2}(\Sigma;\C^N))$ and a bounded restriction to $L^2((-1,1);H^{1/2}(\Sigma;\C^N))$, then
Proposition~\ref{prop_Bochner}~(i) and \eqref{eq_interpolation_inequality} imply   
\begin{equation} \label{interpolation_constant}
  \| B \|_{0 \rightarrow 0} \leq C_{1/2, -1/2} \| B \|_{1/2 \rightarrow 1/2}^{1/2} \| B \|_{-1/2 \rightarrow -1/2}^{1/2}.
\end{equation}
\begin{lemma} \label{lemma_smallness_condition}
Let $z \in \rho(H)$ and let $X_z$ be as in~\eqref{def_X_z}. Then,
the following assertions hold:
  \begin{itemize}
    \item[$\textup{(i)}$] $X_z \geq \frac{1}{C_{\tr,E,z}} + \mathcal{O}((\varepsilon_2+\varepsilon_2\lvert \log(\varepsilon_2)\rvert)^{1/2})$, where the constant
    \begin{equation*}
    \begin{split}
    C_{\tr,E,z}= 2 C_{1/2,-1/2} &\max_{(\sign,\mu) \in \{+,-\} \times\{z,\overline{z}\}}\|\tr^\sign\|_{H^{1}(\Omega_\sign;\C^N) \to H^{1/2}(\Sigma;\C^N)} \\
    &\quad\cdot\norm{E}_{H^{1}(\Omega_\sign;\C^N) \to H^{1}(\mathbb{R}^\theta;\C^N)}\norm{\Phi_\mu}_{H^{1/2}(\Sigma;\C^N) \to H^{1}(\Omega_\sign;\C^N)}
    \end{split}
    \end{equation*} 
    depends only on the geometry of $\Sigma$ and $z \in \rho(H)$.
    \item[$\textup{(ii)}$] If $\Sigma$ is $C^\infty$-smooth and compact and $m > 0$, then $X_z \leq \frac{\pi}{4}$.
  \end{itemize}
  In particular, if $\norm{V}_{W^1_\infty(\Sigma;\C^{N\times N})}  \norm{q}_{L^\infty(\R;\R)} < \frac{1}{C_{\tr,E,z}}$ and $\varepsilon_2$ is chosen sufficiently small, then the claim in Proposition~\ref{prop_I+B_conv} is true.
\end{lemma}

\begin{proof}
	(i) First, by the definition of $\overline{B}_\varepsilon(z)$ and $\tau_\delta^{\Omega_\pm}$ in \eqref{def_B_bar} and \eqref{def_tau_delta_pm}, respectively, and the Cauchy-Schwarz inequality we have for $f \in L^2((-1,1);H^{1/2}(\Sigma;\C^N))$ and $\varepsilon \in (0,\varepsilon_2)$ 
	\begin{equation*}
	\begin{split}
	&\norm{\overline{B}_\varepsilon( z)f}_{1/2}^2 \leq \int_{-1}^1 \biggl(\int_{-1}^t \norm{\tr^- \tau_{\varepsilon (t-s) }^{\Omega_-}( \Phi_z f(s))_-}_{H^{1/2}(\Sigma;\C^N)} ds \\
	&\qquad + \int_t^1 \norm{\tr^+ \tau_{\varepsilon (t-s)  }^{\Omega_+}(\Phi_z f(s))_+}_{H^{1/2}(\Sigma;\C^N)} ds\biggr)^2 \,dt\\
	&\leq \! \int_{-1}^{1}\! \biggl(\int_{-1}^{1} \frac{C_{\tr,E,z}}{2 C_{1/2, -1/2}}\! \sup_{\delta \in (-2\varepsilon_2,2\varepsilon_2)} \!\!\norm{\tau_\delta}_{H^{1}(\R^\theta;\C^N) \to H^{1}(\R^\theta;\C^N)} \norm{f(s)}_{H^{1/2}(\Sigma;\C^N)} ds \biggr)^2 \!dt \\
	& \leq \biggl( \frac{C_{\tr,E,z}}{C_{1/2, -1/2}} \sup_{\delta \in (-2\varepsilon_2,2\varepsilon_2)}\norm{\tau_\delta}_{H^{1}(\R^\theta;\C^N) \to H^{1}(\R^\theta;\C^N)} \norm{f}_{1/2} \biggr)^2.
	\end{split}
	\end{equation*}
	Hence,
	\begin{equation*}
	\sup_{\varepsilon \in (0,\varepsilon_2)} \norm{ \overline{B}_\varepsilon(z)}_{1/2 \to 1/2}  \leq \frac{C_{\tr,E,z}}{C_{1/2, -1/2}} \sup_{\delta \in (-2\varepsilon_2,2\varepsilon_2)}\norm{\tau_\delta}_{H^{1}(\R^\theta;\C^N) \to H^{1}(\R^\theta;\C^N)}
	\end{equation*}
	 and a similar calculation using~\eqref{representation_B_0}, the estimate $\norm{E}_{H^{1}(\Omega_\sign;\C^N) \to H^{1}(\mathbb{R}^\theta;\C^N)} \geq 1$, and $C_{1/2, -1/2} \geq 1$ leads to
	\begin{equation} \label{estimate_B_0}
	   \norm{ \overline{B}_0(z)}_{1/2 \to 1/2}  \leq C_{\tr,E,z}.
	\end{equation}
	Let us proceed with the estimate for $\overline{B}_\varepsilon(z)$. Since by \eqref{abc1} and~\eqref{abc2}
	\begin{equation*}
	\sup_{\delta \in (-2\varepsilon_2,2\varepsilon_2)}\norm{\tau_\delta}_{H^{1}(\R^\theta;\C^N) \to H^{1}(\R^\theta;\C^N)} \leq \frac{1+2\varepsilon_2\norm{D \nu}_{L^\infty(\R^\theta;\R^{\theta\times \theta})}}{{(1- 2\varepsilon_2\norm{D \nu}_{L^\infty(\R^\theta;\R^{\theta\times \theta})})^{\theta/2}}} = 1 +\mathcal{O}(\varepsilon_2),
	\end{equation*} 
	the results from~\eqref{ESTIMATE_B_BAR_B_TILDE}--\eqref{klkl} and Appendix~\ref{sec_appendix_diff_B} imply 
	\begin{equation}\label{supbbb}
	  \sup_{\varepsilon \in (0,\varepsilon_2)} \|\widetilde{B}_\varepsilon(z)\|_{1/2 \to 1/2} \leq \frac{C_{\tr,E,z}}{C_{1/2,-1/2}}+ \mathcal{O}((\varepsilon_2+\varepsilon_2\lvert \log(\varepsilon_2)\rvert)^{1/2}).
    \end{equation}
    It is clear that the estimates \eqref{estimate_B_0} and \eqref{supbbb} hold also with $z$ replaced by $\overline z$.
    As in \textit{Step~3} of the proof of Proposition \ref{prop_B_conv} we use again that the anti-dual $(\widetilde{B}_\varepsilon (\overline z))'$ of $\widetilde{B}_\varepsilon (\overline z)$ 
    in $L^2((-1,1)H^{-1/2}(\Sigma;\C^N))$ is an extension of 
    $\widetilde{B}_\varepsilon (z)$ that is bounded by $\|\widetilde{B}_\varepsilon(\overline{z})\|_{1/2 \to 1/2}$. Hence, with~\eqref{interpolation_constant} one gets
	\begin{equation*}
	\begin{split}
	  \sup_{\varepsilon \in (0,\varepsilon_2)} \| \widetilde{B}_\varepsilon(z)\|_{0 \to 0} 
	  & \leq \sup_{\varepsilon \in (0,\varepsilon_2)}
	   C_{1/2,-1/2}\|\widetilde{B}_\varepsilon(z)\|_{1/2 \to 1/2}^{1/2} \|(\widetilde{B}_\varepsilon(\overline z))'\|_{-1/2 \to -1/2}^{1/2}\\
	   &= \sup_{\varepsilon \in (0,\varepsilon_2)}
	   C_{1/2,-1/2}\|\widetilde{B}_\varepsilon(z)\|_{1/2 \to 1/2}^{1/2} \|\widetilde{B}_\varepsilon(\overline z)\|_{1/2 \to 1/2}^{1/2} \\
	 &\leq C_{\tr,E,z}
	+ \mathcal{O}((\varepsilon_2+\varepsilon_2\lvert \log(\varepsilon_2)\rvert)^{1/2}).
	\end{split}
    \end{equation*}
	Thus, it follows from the definition of $\widetilde B_\varepsilon(z)$ in \eqref{wtwtb} and Lemma \ref{lem_M_operators} that  
	$$\sup_{ \varepsilon \in (0,\varepsilon_2) } \norm{ B_\varepsilon(z)}_{0 \to 0} \leq C_{\tr,E,z}
	+ \mathcal{O}((\varepsilon_2+\varepsilon_2\lvert \log(\varepsilon_2)\rvert)^{1/2}).$$  
	Together with~\eqref{estimate_B_0} this yields the claim in~(i).
%	 It remains to find an estimate for the operator norm of  $B_0(z)$. \mh{Note that $C_{1/2,-1/2} \geq 1$ and $\| E \|_{H^{1}(\Omega_\sign;\C^N) \to H^{1}(\mathbb{R}^\theta;\C^N)} \geq 1$}. Therefore, Proposition \ref{proposition_C_z} and \eqref{eq_B_0} yield \mh{
%	\begin{equation*}
%	\begin{split}
%	\|B_0&(z)\|_{1/2\to1/2} \\
%	&\leq 2 \max_{\sign \in \{+,-\}}\big\|\tr^\sign\big\|_{H^{1}(\Omega_\sign;\C^N) \to H^{1/2}(\Sigma;\C^N)}\norm{\Phi_z}_{H^{1/2}(\Sigma;\C^N) \to H^{1}(\Omega_\sign;\C^N)}\\
%	&\leq 2  \max_{\sign \in \{+,-\}}\| E \|_{H^{1}(\Omega_\sign;\C^N) \to H^{1}(\mathbb{R}^\theta;\C^N)} \\
%	&\qquad \quad \cdot\big\|\tr^\sign\big\|_{H^{1}(\Omega_\sign;\C^N) \to H^{1/2}(\Sigma;\C^N)}\norm{\Phi_z}_{H^{1/2}(\Sigma;\C^N) \to H^{1}(\Omega_\sign;\C^N)}\\
%	&\leq \frac{C_{\tr,E,z}}{C_{1/2,-1/2}}.
%	\end{split}
%	\end{equation*} }
%	Using  again duality and symmetry one obtains  $\norm{B_0(z)}_{0\to0} \leq C_{\tr,E,z}.$
	
	(ii) First, note that there exist $\varphi_n \in H^{1/2}(\Sigma; \mathbb{C}^N)$ such that 
	\begin{equation} \label{equation_phi_n}
	  \norm{\varphi_n}_{H^{1/2}(\Sigma;\C^N)} = 1 \quad \text{and} \quad \left\| \left( -\frac{1}{2} I_N + \mathcal{C}_z \right) \varphi_n \right\|_{H^{1/2}(\Sigma; \mathbb{C}^N)}  \rightarrow 0, \quad \text{as } n \rightarrow \infty;
	\end{equation}
	for $z \in (-m, m)$ this follows from \cite[Corollaries~3.8 and~4.12]{H22}, 
	while for $z \in \mathbb{C} \setminus \mathbb{R}$ one may additionally use that \cite[Proposition~4.4~(iv)]{BHSS22} implies that $\mathcal{C}_z - \mathcal{C}_0$ is compact in $H^{1/2}(\Sigma;\C^N)$.
	Define the functions $f_n \in L^2((-1,1); H^{1/2}(\Sigma; \mathbb{C}^N))$ by $f_n(t) := \exp( i (\alpha \cdot \nu) \pi t /4 ) \varphi_n$ and compute with Proposition~\ref{prop_Bochner}~(iii)
	\begin{equation*} 
	  \begin{split}
	    \frac{i}{2} (\alpha \cdot \nu) & \int_{-1}^1 \sign(t-s) f_n(s) ds\\
	    & = \frac{2}{\pi} \biggl( \int_{-1}^t \frac{d}{ds} \exp \bigl( i (\alpha \cdot \nu) \tfrac{\pi}{4} s \bigr) \varphi_n ds - \int_t^1 \frac{d}{ds} \exp \bigl( i (\alpha \cdot \nu) \tfrac{\pi}{4} s \bigr) \varphi_n ds \biggr) \\
	    &= \frac{4}{\pi} \bigl( \exp\bigl( i (\alpha \cdot \nu) \tfrac{\pi}{4} t \bigr) - \cos\bigl( (\alpha \cdot \nu) \tfrac{\pi}{4}  \bigr) \varphi_n = \frac{4}{\pi} \biggl( f_n(t) - \frac{\sqrt{2}}{2} I_N \varphi_n\Biggr),
	  \end{split}
	\end{equation*}
	where in the last step the identity $\cos( (\alpha \cdot \nu) \pi/ 4) = \cos(\pi/ 4) I_N = \frac{\sqrt{2}}{2} I_N$, which can be shown via the power series representation of  cosine  and $(\alpha \cdot \nu)^2=I_N$, was used. With a similar argument, one verifies that $(\alpha \cdot \nu) \sin( (\alpha \cdot \nu) \pi/ 4) = \sin(\pi/ 4) I_N = \frac{\sqrt{2}}{2} I_N$, which yields
	\begin{equation*} 
	  \begin{split}
	    \int_{-1}^1 f_n(s) ds 
	    &= \int_{-1}^1 \exp \bigl( i (\alpha \cdot \nu) \tfrac{\pi}{4} s \bigr) \varphi_n ds 
	    = \frac{8}{\pi} (\alpha \cdot \nu) \sin\bigl( (\alpha \cdot \nu) \tfrac{\pi}{4} \bigr) \varphi_n 
	    = \frac{4 \sqrt{2}}{\pi} \varphi_n.
	  \end{split}
	\end{equation*}
	Combining the last two displayed formulas with the representation of $B_0(z)$ in~\eqref{eq_B_0}, we conclude
	\begin{equation*}
	\begin{split}
  	  B_0 (z) f_n(t) &= \frac{i}{2}(\alpha \cdot \nu)\int_{-1}^1 \sign(t-s)f(s) \,ds + \mathcal{C}_z \int_{-1}^1  f(s) \, ds \\
  	  &= \frac{4}{\pi}\Bigl(f_n(t) + \sqrt{2} \Bigl(-\frac{1}{2}I_N +  \mathcal{C}_z \Bigr)\varphi_n \Bigr).
  	  \end{split}
	\end{equation*}
	Hence, using~\eqref{equation_phi_n}, we find that
	\begin{equation*}
	  \|B_0(z) \|_{1/2 \to 1/2} \geq \lim_{n \rightarrow \infty} \frac{\| B_0 (z) f_n \|_{1/2}}{\| f_n \|_{1/2}} = \frac{4}{\pi}
	\end{equation*}
	and thus $X_z \leq \tfrac{\pi}{4}$.
\end{proof}

\section{Proof of the main results}\label{sec_main}

In this section we prove the main results of this paper. For this we show in Proposition~\ref{prop_main_conv} 
that the resolvent $(H_{\varepsilon}-z)^{-1}$ of the self-adjoint operator $H_\varepsilon$ in \eqref{eq_H_eps} converges to the limit 
operator 
\begin{equation}\label{lz}
 \mathcal{L}(z) := R_ z  - A_0( z)Vq(I +B_{0}( z)Vq)^{-1}C_0( z)
\end{equation}
whenever $q$ and $V$ in \eqref{eq_q} and \eqref{eq_V}, respectively,  satisfy the condition \eqref{eq_smallnes_assumption} for some $z\in\rho(H) = \mathbb{C} \setminus ((-\infty, -|m|] \cup [|m|, \infty))$.
Lemma~\ref{lem_I+T} and Lemma~\ref{lem_BtoC_z} collect some auxiliary considerations that are needed in Proposition~\ref{prop_main_limit},
where it is shown  that
$\mathcal{L}(z) = (H_{\widetilde{V}} - z)^{-1}$ with $H_{\widetilde{V}}$ defined by~\eqref{def_H_Vtilde}. After these preparations 
we complete the proofs of Theorem~\ref{THEO_MAIN}, Corollary~\ref{cor_main}, and Corollary~\ref{cor_main_inv}.

\begin{proposition}\label{prop_main_conv}
Let $z \in \rho(H)$ and $\varepsilon \in (0, \varepsilon_2)$ with $\varepsilon_2>0$ given by~\eqref{eq_eps_2}, and let $q$ and $V$ be 
as in \eqref{eq_q} and \eqref{eq_V}, respectively, such that \eqref{eq_smallnes_assumption} holds.
Then, $z \in \rho(H_\varepsilon)$ and for any $r \in (0,\tfrac{1}{2})$ one has
	  \begin{equation*}
		  \begin{split}
		  & \|(H_{\varepsilon}-  z)^{-1} - \mathcal{L}(z)  \|_{L^2(\R^\theta;\C^N) \to L^2(\R^\theta;\C^N)} 
		  \leq C \varepsilon^{1/2-r}, \quad \varepsilon \in (0, \varepsilon_2),
		  \end{split}
	  \end{equation*}
	  where the operator $\mathcal{L}(z)$ is given by \eqref{lz}.
	  In particular, $(H_{\varepsilon}-  z)^{-1}$ converges to $\mathcal{L}(z)$ in the operator norm as $\varepsilon \to 0+$.
\end{proposition}
\begin{proof}
	 According to Proposition~\ref{prop_resolvent_formula} and \eqref{eq_smallnes_assumption}, see also Proposition~\ref{prop_I+B_conv},  we have $z \in \rho(H_\varepsilon)$ and
	\begin{equation} \label{equation_resolvent_difference}
	\begin{split}
	(H_{\varepsilon}&- z)^{-1} - \mathcal{L}(z) \\
	%=(H_{\varepsilon} -  z)^{-1} -	R_ z  + A_0( z)Vq(I + B_{0}(z)Vq)^{-1}C_0( z)  \\
		&= - A_\varepsilon( z)Vq(I + B_{\varepsilon}( z)Vq)^{-1}C_\varepsilon( z) +A_0( z)Vq(I + B_{0}( z)Vq)^{-1}C_0( z) \\
		&=-A_\varepsilon( z)Vq(I + B_{\varepsilon}( z)Vq)^{-1}(C_\varepsilon( z) - C_0(z)) - A_\varepsilon( z)Vq((I + B_{\varepsilon}( z)Vq)^{-1} \\
		&- (I + B_{0}( z)Vq)^{-1})C_0( z)- (A_\varepsilon (z) - A_0( z))Vq(I + B_{0}( z)Vq)^{-1}C_0( z).
	\end{split}
	\end{equation}
	 We note that by Propositions~\ref{prop_A_eps} and~\ref{prop_I+B_conv} the  maps  $A_\varepsilon(z): L^2((-1,1);L^2(\Sigma;\C^N)) \to L^2(\R^\theta;\C^N)$ and ${(I+B_{\varepsilon}( z)Vq)^{-1}: L^2((-1,1);L^2(\Sigma;\C^N)) \to L^2((-1,1);L^2(\Sigma;\C^N))}$ are uniformly bounded.
	Employing this and Proposition~\ref{prop_C_eps} we see that 
	\begin{equation} \label{equation_estimate1}
	\begin{split}
		\|A_\varepsilon( z)Vq(I + B_{\varepsilon}( z)Vq)^{-1}(C_\varepsilon( z)- C_0(z) )\|_{{L^2(\R^\theta;\C^N) \to L^2(\R^\theta;\C^N)}}& \\
		\leq C \norm{C_\varepsilon( z)-C_0( z)}_{L^2(\R^\theta;\C^N) \to 0}  \leq C \varepsilon^{1/2 - r}&. 
		\end{split}
	\end{equation} 
	Since $C_0(z): L^2(\R^\theta;\C^N) \to L^2((-1,1);H^{1/2}(\Sigma;\C^N))$ in \eqref{def_C_0} is bounded,  Proposition~\ref{prop_I+B_conv} yields
	\begin{equation} \label{equation_estimate2}
		\begin{split}
		&\| A_\varepsilon( z)Vq((I + B_{\varepsilon}( z)Vq)^{-1} -(I + B_{0}( z)Vq)^{-1})C_0( z) \|_{L^2(\R^\theta;\C^N)\to L^2(\R^\theta;\C^N)}\\
		&\quad\leq C \norm{ (I + B_{\varepsilon}( z)Vq)^{-1} -(I + B_{0}( z)Vq)^{-1}}_{1/2 \to 0}\leq C \varepsilon^{1/2-r}.
	\end{split}
	\end{equation}
	Eventually, in a similar way as in~\eqref{equation_estimate1} we find with Proposition~\ref{prop_A_eps} that
	\begin{equation} \label{equation_estimate3}
		\begin{split}
		&\norm{(A_\varepsilon(z)-A_0(z))Vq(I + B_{0}( z)Vq)^{-1} C_0( z) }_{L^2(\R^\theta;\C^N)\to L^2(\R^\theta;\C^N)} \\
		&\qquad\leq C \norm{A_\varepsilon( z)-A_0( z)}_{0 \to L^2(\R^\theta;\C^N)}\leq C \varepsilon^{1/2 - r}.
		\end{split}
	\end{equation} 
	Combining~\eqref{equation_estimate1}--\eqref{equation_estimate3} with~\eqref{equation_resolvent_difference} shows the claim of this proposition.
\end{proof}

The next goal is to show that the limit operator $\mathcal{L}(z)$ in \eqref{lz} 
is the resolvent of $H_{\widetilde{V}}$ defined in~\eqref{def_H_Vtilde}. This requires some technical preparations and we 
first introduce the  operator
\begin{equation} \label{def_TF} 
\begin{split}
 		T &:L^2((-1,1);L^2(\Sigma;\C^N)) \to L^2((-1,1);L^2(\Sigma;\C^N)), \\ 
 		Tf(t) &:= \frac{i}{2}  \int_{-1}^1 \textup{sign}(t-s)f(s) \,ds,
 		\end{split}
 	\end{equation}
and the function
\begin{equation} \label{def_Q}
  Q(t) := -\frac{1}{2} + \int_{-1}^t q (s) d s, \quad t \in [-1,1].
\end{equation}
Note that $Q$ satisfies $Q' = q$, $Q(-1) = -\frac{1}{2}$, and by~\eqref{eq_q} also $Q(1) = \frac{1}{2}$. Moreover, for $r \in [0, \frac{1}{2}]$ the map $T$ gives rise to a bounded operator in $L^2((-1, 1); H^r(\Sigma; \mathbb{C}^N))$.

\begin{lemma}\label{lem_I+T}
	Let  $q$ and $V$ be as in~\eqref{eq_q} and~\eqref{eq_V}, respectively, let $r \in [0,\frac{1}{2}]$, and assume that $\cos\bigl(\tfrac{1}{2}(\alpha \cdot \nu)V\bigr)^{-1} \in  W^1_\infty (\Sigma;\C^{N \times N})$. Then, the following is true:
	\begin{itemize}
		\item[$\textup{(i)}$] $I +  T(\alpha \cdot \nu) Vq$ is boundedly invertible in   $L^2((-1,1);H^r(\Sigma;\C^N))$ and its inverse is given by the operator $O$ in~\eqref{eq_O}.
		\item[$\textup{(ii)}$] If $f \in \ran \mathfrak{J}$ (that is, $f$ is independent of $t \in (-1,1)$), then 
		\begin{equation}\label{eq_Tqinv_const}
		(I +  T(\alpha \cdot \nu) Vq)^{-1}f(t) = \cos\bigl(\tfrac{1}{2}(\alpha \cdot \nu)V\bigr)^{-1}\exp(-i(\alpha \cdot \nu) V Q(t))f(t)
		\end{equation}
		holds for a.e. $t \in (-1,1)$.
    \end{itemize}
\end{lemma}

\begin{proof}%[Proof of Lemma~\ref{lem_I+T}]
	(i) It will be shown that the operator defined in \eqref{eq_O} below is the inverse of $I +  T(\alpha \cdot \nu) Vq$.
	Fix $r \in [0, \frac{1}{2} ]$ and define the operators
	\begin{equation}\label{eq_Xi}
	\begin{split}
			\Xi&: L^2((-1,1);H^r(\Sigma;\C^N)) \rightarrow H^r(\Sigma;\C^N), \\
			\Xi f &=  \frac{1}{2} \cos\bigl(\tfrac{1}{2}(\alpha \cdot \nu)V\bigr)^{-1} i (\alpha \cdot \nu) V  \int_{-1}^1  \exp\bigl( i (\alpha \cdot \nu) V \bigl(Q(s)-\tfrac{1}{2}\bigr) \bigr) q(s) f(s)\,ds,
			\end{split}
	\end{equation}
	and 
	\begin{equation}\label{eq_O}
		\begin{split}
		O &: L^2((-1,1); H^r(\Sigma;\C^N))  \to L^2((-1,1);H^r(\Sigma;\C^N))\\
		Of(t)&:=f(t)+   \exp(-i (\alpha \cdot \nu) V Q(t))\Xi f\\
		&\qquad - i (\alpha \cdot \nu) V \int_{-1}^t  \exp( i(\alpha \cdot \nu) V (Q(s)-Q(t))) q(s) f(s) \,ds.
		\end{split}
	\end{equation}
	We will show that $\Xi$ and $O$ are bounded and that $O = (I +  T(\alpha \cdot \nu) Vq)^{-1}$. First, we verify that $\Xi$ is well-defined and bounded. Let
	$f \in L^2((-1,1);H^r(\Sigma;\C^N))$. Then, the integrand in \eqref{eq_Xi} is measurable as a function from $(-1,1)$ to $H^r(\Sigma;\C^N)$
	since $\bigl\langle\exp\bigl( i (\alpha \cdot \nu) V \bigl(Q(\cdot)-\tfrac{1}{2}\bigr) \bigr) q(\cdot) f(\cdot),\psi\bigl\rangle_{H^r(\Sigma;\C^N)}$ is measurable for all $\psi \in H^r(\Sigma;\C^N)$, see Definition \ref{def_Bochner_measurable}. In fact, the latter function is the pointwise limit of the sequence of measurable functions
	\begin{equation*}
		\begin{split}
		&t \mapsto \sum_{k=0}^n\frac{\bigl(\bigl( i (\alpha \cdot \nu) V \bigl(Q(t)-\tfrac{1}{2}\bigr)\bigr)^k q(t) f(t), \psi\bigr)_{H^r(\Sigma;\C^N)}}{k!} \\
		&\qquad = \sum_{k=0}^n \frac{\bigl(Q(t)-\tfrac{1}{2}\bigr)^kq(t)\spv{(i (\alpha \cdot \nu) V )^kf(t)}{ \psi }_{H^r(\Sigma;\C^N)}}{k!}.
	 	\end{split}
	\end{equation*}
	Moreover, as $\cos\bigl(\tfrac{1}{2}(\alpha \cdot \nu)V\bigr)^{-1}, \alpha \cdot \nu, V \in W^1_\infty(\Sigma;\C^{N \times N})$ 
	it follows that
	\begin{equation*}
		\begin{split}
		 \|&\Xi f\|_{H^r(\Sigma;\C^N)}\\
		 &= \frac{1}{2} \biggl\|\cos\bigl(\tfrac{1}{2}(\alpha \cdot \nu)V\bigr)^{-1} i (\alpha \cdot \nu) V\!\! \! \int_{-1}^1  \exp\bigl( i (\alpha \cdot \nu) V \bigl(Q(s)-\tfrac{1}{2}\bigr) \bigr) q(s)f(s) ds\biggr\|_{H^r(\Sigma;\C^N)} \\
		 &\leq C \int_{-1}^1 \bigl\| \exp\bigl( i (\alpha \cdot \nu) V \bigl(Q(s)-\tfrac{1}{2}\bigr) \bigr)q(s) f(s)\bigr\|_{H^r(\Sigma;\C^N)} \,ds \\
		 \end{split}
	\end{equation*}	
	and $\alpha \cdot \nu, V \in W^1_\infty(\Sigma;\C^{N \times N})$ also imply 
	$\exp\bigl( i (\alpha \cdot \nu) V \bigl(Q(s)-\tfrac{1}{2}\bigr)\bigr)\in W^1_\infty(\Sigma;\C^{N \times N})$ via the 
	 power series of the exponential function. Using $q \in L^\infty((-1,1);\R)$ we conclude
	\begin{equation*}
	\begin{split}
		 \|&\Xi f\|_{H^r(\Sigma;\C^N)} \\
		 &\leq C \int_{-1}^1 \bigl\| \exp\bigl( (\alpha \cdot \nu) V \bigl(Q(s)-\tfrac{1}{2}\bigr) \bigr) \bigr\|_{W^1_\infty(\Sigma;\C^{N \times N})} \norm{q}_{L^\infty((-1,1);\R)}\norm{f(s)}_{H^r(\Sigma;\C^N)} \,ds\\
		 &\leq C\int_{-1}^1  \norm{f(s)}_{H^r(\Sigma;\C^N)} \\
		 &\leq 
		  C \norm{f}_r.
		  \end{split}
	\end{equation*}	
	This shows that $\Xi$ is well-defined and  bounded. Analogously one can check that $O$ is well-defined and bounded. Hence, in order to show (i) it suffices to prove
	\begin{equation}\label{eq_inv_id}
	(I+ T(\alpha \cdot \nu) Vq)O f= 	O(I+ T(\alpha \cdot \nu) Vq) f = f 
	\end{equation}
	for all $f \in L^2((-1,1);L^2(\Sigma;\C^N))$.
	By
	Proposition~\ref{prop_Bochner}~(iii) and \cite[Proposition~1.2.24]{HNVW16} this is true, if for $\sigma$-a.e. $x_\Sigma \in \Sigma$ the relation
	\begin{equation*}%\label{eq_inv_ptw}
		(I+ T(\alpha \cdot \nu) Vq)O f(\cdot)(x_\Sigma) = O(I+ T(\alpha \cdot \nu) Vq) f(\cdot)(x_\Sigma) =  f(\cdot)(x_\Sigma) 
	\end{equation*}
	holds a.e. on $(-1,1)$. Let $f \in L^2((-1,1);L^2(\Sigma;\C^N))$ and $x_\Sigma \in \Sigma$ be fixed such that
	$\varphi = \varphi(x_\Sigma) := f(\cdot)(x_\Sigma) \in L^2((-1,1);\C^N)$ and set $A = A(x_\Sigma) := (\alpha \cdot \nu(x_\Sigma)) V(x_\Sigma)$.  Then, we have for a.e. $ t \in (-1,1)$
	\begin{equation}\label{eq_right_inv_ptw}
		\begin{split}
		(I+ T(\alpha \cdot \nu) Vq)&O f(t)(x_\Sigma)= \varphi(t) + \exp(-iAQ(t))\Xi f(x_\Sigma) \\
		&- i A \int_{-1}^t \exp(iA(Q(s)-Q(t)))q(s) \varphi(s) \,ds\\
		&\phantom{=}+\frac{i}{2} \int_{-1}^1 \sign(t-s)  A q(s)\bigg(\varphi(s) + \exp(-iAQ(s))\Xi f(x_\Sigma) \\
		&\qquad - i A \int_{-1}^s \exp(iA(Q(r)-Q(s)))q(r) \varphi(r) \,dr \bigg)\,ds.
		\end{split}
	\end{equation}
	With a direct calculation we find that
	\begin{equation*}
	\begin{split}
			&  \frac{i}{2} \int_{-1}^{1}  \sign(t-s)A \exp(-i AQ(s)) q(s) \Xi f(x_\Sigma) \,ds \\
			&\quad=\frac{1}{2}\bigg(\int_{-1}^{t} i A\exp(-i AQ(s)) q(s)  \,ds  - \int_{t}^{1} i A\exp(-i AQ(s)) q(s) \,ds \bigg) \Xi f(x_\Sigma)\\
		&\quad = - \exp(-i AQ(t) ) \Xi f(x_\Sigma) + \frac{1}{2}\bigl(\exp\bigl(\tfrac{-i}{2} A\bigr) + \exp\bigl( \tfrac{i}{2}A\bigr)\bigr) \Xi f(x_\Sigma) \\
		&\quad=- \exp(-i AQ(t) ) \Xi f(x_\Sigma) +  \cos\bigl(\tfrac{1}{2} A\bigr)\Xi f(x_\Sigma).
	\end{split}
	\end{equation*}
	Furthermore, integration by parts gives us
	\begin{equation*}
	\begin{split}
		&-\frac{i}{2}\int_{-1}^1  \sign(t-s)Aq(s) i A \int_{-1}^s \exp(iA(Q(r)-Q(s)))q(r) \varphi(r) \,dr\,ds \\
		&\quad= 	\frac{i}{2} A\int_{-1}^t  \frac{d}{ds}(\exp(-iAQ(s))) \int_{-1}^s  \exp(iAQ(r)) q(r)\varphi(r) \,dr \, ds \\
		&\qquad -\frac{i}{2}A \int_{t}^1  \frac{d}{ds}(\exp(-iAQ(s))) \int_{-1}^s  \exp(iAQ(r)) q(r)\varphi(r) \,dr \, ds \\
		 &\quad= iA \exp(-iAQ(t))\int_{-1}^t  \exp(iAQ(r)) q(r)\varphi(r) \,dr\\
		 &\qquad  -\frac{i}{2} A\int_{-1}^1 \exp\bigl(iA\bigl(Q(r)- \tfrac{1}{2}\bigr)\bigr)q(r)\varphi(r) \, d r - \frac{i}{2}A \int_{-1}^1 \sign(t-s) q(s)\varphi(s) \,ds \\
		&\quad=iA \exp(-iAQ(t))\int_{-1}^t  \exp(iAQ(r)) q(r)\varphi(r) \,dr \\
		&\qquad - \frac{i}{2}A \int_{-1}^1 \sign(t-s) q(s)\varphi(s) \,ds  - \cos\bigl(\tfrac{1}{2}A\bigr)\Xi f(x_\Sigma).
	\end{split}
	\end{equation*}
	A combination of the last two calculations with~\eqref{eq_right_inv_ptw} yields
	\begin{equation*}
	  (I+ T(\alpha \cdot \nu) Vq)O f(t)(x_\Sigma) = \varphi(t).
	\end{equation*}
	One verifies in a very similar way that $O$ is also the left inverse of $I+ T(\alpha \cdot \nu) Vq$. Consequently, \eqref{eq_inv_id} is true. 

	(ii) Let $f \in \ran \mathfrak{J}$, that is, $f$ is independent of $t \in (-1,1)$. Instead of inserting $f$ in \eqref{eq_O} 
	we find it more convenient and easier to verify this claim directly by showing 
	\begin{equation*}
		 (I+ T(\alpha \cdot \nu) Vq)\cos\bigl(\tfrac{1}{2}(\alpha\cdot\nu)V \bigr)^{-1}\exp(-i(\alpha \cdot \nu) V Q)f = f.
	\end{equation*}
	Similar as above it suffices to prove 
	\begin{equation*}
	(I+ T(\alpha \cdot \nu) Vq)\cos\bigl(\tfrac{1}{2}(\alpha\cdot\nu)V \bigr)^{-1}\exp(-i(\alpha \cdot \nu) V Q)f(\cdot)(x_\Sigma) = f(\cdot)(x_\Sigma)
	\end{equation*}
	for $\sigma$-a.e. $x_\Sigma \in \Sigma$
	a.e. on $(-1,1)$. Thus, we again fix $x_\Sigma \in \Sigma$ and use the same abbreviations as in the proof of (i). Since here 
	$f$ is constant  with respect to $t$ also $\varphi = f(t)(x_\Sigma)$ is independent of $t$. We then compute
	\begin{equation*}
	\begin{split}
	(&I +  T(\alpha \cdot \nu) Vq) \cos\bigl(\tfrac{1}{2}(\alpha\cdot\nu)V \bigr)^{-1}\exp(-i(\alpha \cdot \nu) V Q)f(t)(x_\Sigma)\\
	&= \cos\bigl(\tfrac{1}{2} A \bigr)^{-1}\!\exp(-iA Q(t))\varphi+ \frac{i}{2}\!\int_{-1}^1\sign(t-s)A q(s)\cos\bigl(\tfrac{1}{2} A \bigr)^{-1}\exp(-iA Q(s)) \varphi ds \\
	&= \cos\bigl(\tfrac{1}{2} A\bigr)^{-1}\exp(-iA Q(t)) \varphi - \frac{1}{2}\cos\bigl(\tfrac{1}{2}A \bigr)^{-1} \int_{-1}^1\sign(t-s)\frac{d}{ds}\exp(-iA Q(s)) \, ds \varphi \\
	&=\cos\bigl(\tfrac{1}{2} A \bigr)^{-1}\exp(-iAQ(t))\varphi \\
	&\qquad -\frac{1}{2} \cos\bigl(\tfrac{1}{2} A \bigr)^{-1}  \Bigl( 2 \exp(-iA Q(t)) - \exp\bigl(\tfrac{i}{2}A\bigr) - \exp\bigl(\tfrac{-i}{2}A\bigr)  \Bigr)\varphi \\
	&=\varphi=f(t)(x_\Sigma)
	\end{split}
	\end{equation*}
	for a.e. $t \in (-1,1)$,  which shows $\eqref{eq_Tqinv_const}$.
\end{proof}

In the next lemma we study relations connecting the coefficient matrix $V$ defined in \eqref{eq_V} and the matrix $\widetilde{V} = V S$ with 
$$S = \textup{sinc}\bigl(\tfrac{1}{2}(\alpha \cdot \nu)V\bigr)\cos\bigl(\tfrac{1}{2}(\alpha\cdot\nu)V\bigr)^{-1}.$$  

\begin{lemma}\label{lem_BtoC_z}
	 	Let $z \in \rho(H)$, $q$ and $V$ be as in~\eqref{eq_q} and~\eqref{eq_V}, respectively, assume that $\cos\bigl(\tfrac{1}{2}(\alpha\cdot\nu)V\bigr)^{-1} \in  W^1_\infty(\Sigma;\C^{N \times N})$ and set $\widetilde{V} =VS$, where $S$ is as above; cf. \eqref{eq_scaling_matrix}.
	 	Then, the following is true:
	 	\begin{itemize}
            \item[$\textup{(i)}$] $S, \widetilde{V} \in W^1_\infty(\Sigma;\C^{N \times N})$ and, in particular, the multiplication by $\widetilde{V}$ gives rise to a bounded operator in $H^{1/2}(\Sigma;\C^N)$.
	 		\item[$\textup{(ii)}$]	$\mathfrak{J}^* q\cos\bigl(\tfrac{1}{2}(\alpha\cdot\nu)V\bigr)^{-1} \exp(-i(\alpha \cdot \nu) VQ) \mathfrak{J} = S $.
	 		\item[$\textup{(iii)}$] $(I+B_0(z)Vq) (I+ T(\alpha \cdot \nu) Vq)^{-1}  \mathfrak{J} = \mathfrak{J} (I + \mathcal{C}_z \widetilde{V}  )$.
	 		\item[$\textup{(iv)}$] If \eqref{eq_smallnes_assumption} holds, then $I+B_0(z)Vq$ and $I +\mathcal{C}_z \widetilde{V} $ are boundedly  invertible in the spaces $L^2((-1,1);H^{1/2}(\Sigma;\C^N))$ and $H^{1/2}(\Sigma;\C^N)$, respectively, and 
	 		\begin{equation*}
	 		(I+B_0(z)Vq)^{-1}\mathfrak{J}  = (I+ T(\alpha \cdot \nu) Vq)^{-1}  \mathfrak{J} (I + \mathcal{C}_z \widetilde{V}  )^{-1} .
	 		\end{equation*}
	 	\end{itemize}
\end{lemma}
\begin{proof}
    (i) From $\alpha \cdot \nu, V \in W^1_\infty(\Sigma;\C^{N \times N})$ we conclude $\textup{sinc}\bigl(\tfrac{1}{2}(\alpha\cdot\nu)V\bigr) \in  W^1_\infty(\Sigma;\C^{N \times N})$ using the power series of $\textup{sinc}$. The assumption $\cos\bigl(\tfrac{1}{2}(\alpha\cdot\nu)V\bigr)^{-1} \in  W^1_\infty(\Sigma;\C^{N \times N})$ implies $S  \in  W^1_\infty(\Sigma;\C^{N \times N})$ and hence also $\widetilde{V} = VS \in W^1_\infty(\Sigma;\C^{N \times N})$. Eventually, since the multiplication by any $B \in W^1_\infty(\Sigma;\C^{N \times N})$ gives rise to a bounded operator in $H^{1/2}(\Sigma;\C^N)$, the same is true for $\widetilde{V}$; cf.~\eqref{W_1_infty_mult_op}.
    
	 (ii) Recall that  $\mathfrak{J}$ is defined by~\eqref{eq_def_embedding_op} and that its adjoint acts as 
	 \begin{equation*}
	   \mathfrak{J}^*f = \int_{-1}^1 f(t) \,dt, \qquad f \in L^2((-1,1);H^{1/2}(\Sigma;\C^N)).
	 \end{equation*}
	 As in the proof of the previous lemma we use the abbreviation $A = (\alpha \cdot \nu)V$. Then, with Proposition~\ref{prop_Bochner}~(iii) we get for $\psi \in L^2(\Sigma;\C^N)$ 
	\begin{equation*}
	\begin{split}
	\mathfrak{J}^* q \cos\bigl(\tfrac{1}{2}A\bigr)^{-1} &\exp(-iAQ) \mathfrak{J}\psi   = \int_{-1}^1 \cos\bigl(\tfrac{1}{2}A\bigr)^{-1} \exp(-iAQ(s))q(s) (\mathfrak{J}\psi)(s) \,ds \\
% 	&= \int_{-1}^1 \cos\big(\tfrac{A}{2}\big)^{-1} \exp(-iAQ(s))q(s)\,ds \psi \\
	&=  \cos\bigl(\tfrac{1}{2}A\bigr)^{-1}\int_{-1/2}^{1/2}  \exp(-iAr) \,dr  \psi\\
	&= \cos\bigl(\tfrac{1}{2}A\bigr)^{-1}\int_{0}^{1/2} 2\cos(Ar) \,dr  \psi \\
	&=  \textup{sinc}\bigl(\tfrac{1}{2}A\bigr) \cos\bigl(\tfrac{1}{2}A\bigr)^{-1} \psi \\
	&=S\psi.   
	\end{split}
	\end{equation*}
%	Using the power series of $\textup{sinc}$ one verifies $\int_{0}^{1/2} 2\cos(Ar) \,dr  =\textup{sinc}(\tfrac{A}{2})$
%	and hence 
%	\begin{equation*}
%	\mathfrak{J}^* q \cos\big(\tfrac{A}{2}\big)^{-1} \exp(-iAQ) \mathfrak{J}\psi =  \textup{sinc}\big(\tfrac{A}{2}\big) \cos\big(\tfrac{A}{2}\big)^{-1} \psi 
%	=S\psi
%	\end{equation*}
%	for $\psi \in L^2(\Sigma;\C^N)$. 
This shows (ii).
	
	(iii) Using the definition of  $\mathfrak{J}$ and $\mathfrak{J}^*$, see \eqref{eq_def_embedding_op}, and the representation of $B_0(z)$ in \eqref{eq_B_0} one sees  
	\begin{equation}\label{eq_B_0_J}
	B_0(z) = T(\alpha \cdot \nu) + \mathfrak{J} \mathcal{C}_z\mathfrak{J}^*.
	\end{equation}
	 Hence, for $\psi \in L^2(\Sigma;\C^N)$ item~(ii) above and Lemma \ref{lem_I+T}~(ii) imply 
	\begin{equation*}
		\begin{split}
		(I+B_0(z)Vq)&(I+T (\alpha \cdot \nu) V q)^{-1} \mathfrak{J}\psi  = \mathfrak{J} \psi +   \mathfrak{J} \mathcal{C}_z\mathfrak{J}^*Vq (I+ T (\alpha \cdot \nu) Vq)^{-1} \mathfrak{J}\psi \\
		&= \mathfrak{J}\psi +  \mathfrak{J}\mathcal{C}_z \mathfrak{J}^* V  q \cos\bigl(\tfrac{1}{2}(\alpha\cdot\nu)V \bigr)^{-1} \exp(-i(\alpha \cdot \nu) VQ)\mathfrak{J} \psi \\
		& =  \mathfrak{J}\psi +  \mathfrak{J}\mathcal{C}_z V S\psi = \mathfrak{J}(I + \mathcal{C}_z \widetilde{V} )\psi.
		\end{split}
	\end{equation*}
	
	(iv) First, \eqref{eq_smallnes_assumption} and	Proposition~\ref{prop_I+B_conv} imply that  $I+B_0(z)Vq$ is boundedly  invertible in $L^2((-1,1);H^{1/2}(\Sigma;\C^N))$. Moreover, as $\cos\bigl(\tfrac{1}{2}(\alpha\cdot\nu)V\bigr)^{-1} \in  W^1_\infty(\Sigma;\C^{N \times N})$
	by assumption, the operator $I+ T (\alpha \cdot \nu)Vq$ is bijective in $L^2((-1,1);H^{1/2}(\Sigma;\C^N))$ according to Lemma~\ref{lem_I+T}~(i). Thus, it follows from (iii) and $\mathfrak{J}^* \mathfrak{J} = 2 I$ that $\frac{1}{2}\mathfrak{J}^*(I+  T(\alpha \cdot \nu)Vq)(I+B_0(z)Vq)^{-1} \mathfrak{J}$
	is the left inverse of $I + \mathcal{C}_z \widetilde{V}$ and, in particular, $I + \mathcal{C}_z \widetilde{V}$ is injective.  To show that $I + \mathcal{C}_z \widetilde{V}$ is surjective consider $\varphi \in H^{1/2}(\Sigma;\C^N)$. Then, there exists a unique $f \in  L^2((-1,1);H^{1/2}(\Sigma;\C^N))$ such that 
	\begin{equation} \label{equation_unique_f}
		(I+B_0(z)Vq)(I+T (\alpha \cdot \nu) V q)^{-1} f  = \mathfrak{J}\varphi.
	\end{equation}
 	Define
 	\begin{equation*}
 		\psi:=\varphi - \mathcal{C}_z \mathfrak{J}^*Vq(I+T (\alpha \cdot \nu) V q)^{-1} f.
 	\end{equation*}
 	Since $V \in W^1_\infty(\Sigma;\C^{N \times N})$ we conclude together with Proposition~\ref{proposition_C_z}~(i) that $\psi \in H^{1/2}(\Sigma;\C^N)$.
 	Using~\eqref{eq_B_0_J} and \eqref{equation_unique_f} we see that
 	\begin{equation*} 
 	\begin{split}
		\mathfrak{J} \psi &= \mathfrak{J} \varphi - \mathfrak{J} \mathcal{C}_z \mathfrak{J}^* Vq(I+T (\alpha \cdot \nu) V q)^{-1} f \\
		&= ( I+B_0(z)Vq - \mathfrak{J} \mathcal{C}_z \mathfrak{J}^* Vq ) (I+T (\alpha \cdot \nu) V q)^{-1} f \\
		&= f.
		\end{split}
	\end{equation*}
 	Thus, we conclude with~(iii)  that
 	\begin{equation*}
 	\mathfrak{J}(I + \mathcal{C}_z \widetilde{V} )\psi = (I+B_0(z)Vq)(I+T (\alpha \cdot \nu) V q)^{-1} \mathfrak{J}\psi = \mathfrak{J}\varphi
 	\end{equation*}
  	and therefore also $(I + \mathcal{C}_z \widetilde{V} )\psi = \varphi$. Hence, the operator $I + \mathcal{C}_z \widetilde{V}$ is surjective and 
  	thus bijective. Finally, the formula for the inverse follows directly by applying $(I+B_0(z)Vq)^{-1}$ from the left and $(I+ \mathcal{C}_z\widetilde{V})^{-1}$ from the right to the identity in (iii).
\end{proof}

In the next proposition we use the results from Lemma~\ref{lem_I+T} and Lemma~\ref{lem_BtoC_z} to show that under the assumption in~\eqref{eq_smallnes_assumption} the operator $H_{\widetilde{V}}$ in~\eqref{def_H_Vtilde} is self-adjoint and satisfies $(H_{\widetilde{V}}-z)^{-1} = \mathcal{L}(z)$, where $\mathcal L(z)$ is given by \eqref{lz}.

\begin{proposition}\label{prop_main_limit}
	Let $z \in \rho(H)$, $q$ and $V$ be as in~\eqref{eq_q} and~\eqref{eq_V}, respectively, assume that $\cos\bigl(\tfrac{1}{2}(\alpha\cdot\nu)V\bigr)^{-1} \in  W^1_\infty(\Sigma;\C^{N \times N})$ and set $\widetilde{V} =VS$, where $S$ is given by \eqref{eq_scaling_matrix}.
	If \eqref{eq_smallnes_assumption} holds, then the operator $H_{\widetilde{V}}$ in~\eqref{def_H_Vtilde} is self-adjoint in $L^2(\mathbb{R}^\theta; \mathbb{C}^N)$, $z \in \rho(H_{\widetilde{V}})$, and one has
	\begin{equation*}
  	  (H_{\widetilde{V}} - z)^{-1} = R_ z - A_0( z)Vq(I+ B_0( z)Vq)^{-1}C_0( z) = \mathcal{L}(z).
    \end{equation*}
\end{proposition}

\begin{proof}
	The proof of this proposition consists of two steps. In \textit{Step~1} we show that for the given $z \in \rho(H)$ satisfying \eqref{eq_smallnes_assumption} the operator 
	\begin{equation*}
	  \mathcal{L}(z) = R_ z - A_0( z)Vq(I+ B_0( z)Vq)^{-1}C_0( z)
    \end{equation*}
    fulfils $\mathcal{L}(z) = (H_0 - z)^{-1}$ for a self-adjoint operator $H_0$, and in \textit{Step~2} we prove $H_0 = H_{\widetilde V}$.

	 \textit{Step~1.} First, we check that $\ran \mathcal{L}(z)$ is dense in $L^2(\mathbb{R}^\theta; \mathbb{C}^N)$. Recall that $H$ is the free Dirac operator defined by~\eqref{def_free_op}, let $u \in \mathcal{D}(\R^\theta \setminus \Sigma; \mathbb{C}^N) \subset H^1(\mathbb{R}^\theta; \mathbb{C}^N)=\dom H$ and set $v := (H-z) u$, which is equivalent to $u = R_ z v$. Then, it follows with~\eqref{def_C_0} and Proposition~\ref{prop_Phi_z} that 
	 $$C_0(z) v = \mathfrak{J} \tr R_z v = \mathfrak{J} \tr u = 0$$ 
	 and hence,
	\begin{equation*}
		u = R_ z v - A_0( z)Vq(I+ B_0( z)Vq)^{-1}C_0( z)v = \mathcal{L}(z)v.
	\end{equation*}
	We conclude that $\mathcal{D}(\R^\theta \setminus \Sigma; \mathbb{C}^N) \subset \ran \mathcal{L}(z)$ and therefore $\ran \mathcal{L}(z)$ is dense.

	To show that $\mathcal{L}(z)$ is the resolvent of a self-adjoint operator, we first consider the case $z \in \mathbb{C} \setminus \mathbb{R}\subset\rho(H)$ such that \eqref{eq_smallnes_assumption} holds. Note that, by Proposition~\ref{prop_main_conv}, $\mathcal{L}(z)$ is the limit in the operator norm of $(H_\varepsilon -z)^{-1}$ as $\varepsilon \to 0+$, which are resolvents of self-adjoint operators. Moreover, $(H_\varepsilon - \overline{z})^{-1}$ converges to $\mathcal{L}(z)^*$ as $\varepsilon \to 0+$. Thus, as $\ran \mathcal{L}(z)$ is dense in $L^2(\mathbb{R}^\theta; \mathbb{C}^N)$, we conclude from \cite[Theorem~VIII.22]{RS72} that there exists a self-adjoint operator $H_0$ in $L^2(\R^{\theta};\C^N)$ such that 
	$z \in \rho(H_0)$ and $\mathcal{L}(z)=(H_0-z)^{-1}$. 
	
	Eventually, consider $z \in \R\cap\rho(H)$ such that \eqref{eq_smallnes_assumption} holds. Then, $\mathcal{L}(z)$ is self-adjoint, as it is the limit of the bounded and self-adjoint operators $(H_\varepsilon -z)^{-1}$. Moreover, since $\ran \mathcal{L}(z)$ is dense in $L^2(\R^{\theta};\C^N)$, we have
	$\ker \mathcal{L}(z) = (\ran \mathcal{L}(z))^\perp =\{0\}$, which shows that $\mathcal{L}(z)$ is injective. Hence, 
	\begin{equation*}
		H_0 := z + (\mathcal{L}(z))^{-1}
	\end{equation*}
	defines a  self-adjoint operator with $z \in \rho(H_0)$ and $(H_0-z)^{-1} = \mathcal{L}(z)$.
	
	\textit{Step~2.} We show that $H_{\widetilde V} = H_0$. Since $\mathcal{D}(\mathbb{R}^\theta \setminus \Sigma; \mathbb{C}^N) \subset \dom H_{\widetilde{V}}$, the operator $H_{\widetilde V}$ is densely defined. Moreover, as $\widetilde{V} = \widetilde{V}^*$
	we compute for $u, v \in \dom H_{\widetilde{V}}$ using
	integration by parts (see~\eqref{eq_Greens_formula}) 
	\begin{equation*}
	  \begin{split}
	    (H_{\widetilde{V}} u, v)_{L^2(\mathbb{R}^\theta; \mathbb{C}^N)}& - (u, H_{\widetilde{V}} v)_{L^2(\mathbb{R}^\theta; \mathbb{C}^N)} 
	    \\
	    &= \frac{1}{2} (\tr^+ u_+ + \tr^- u_-, i (\alpha \cdot \nu) (\tr^+ v_+ - \tr^- v_-) )_{L^2(\Sigma; \mathbb{C}^N)} \\
	    &\qquad- \frac{1}{2} (i (\alpha \cdot \nu) (\tr^+ u_+ - \tr^- u_-), \tr^+ v_+ + \tr^- v_- )_{L^2(\Sigma; \mathbb{C}^N)} \\
	    &= -\frac{1}{4} (\tr^+ u_+ + \tr^- u_-, \widetilde{V} (\tr^+ v_+ + \tr^- v_-) )_{L^2(\Sigma; \mathbb{C}^N)} \\
	    &\qquad + \frac{1}{4} (\widetilde{V} (\tr^+ u_+ + \tr^- u_-), \tr^+ v_+ + \tr^- v_- )_{L^2(\Sigma; \mathbb{C}^N)} \\
	    &= 0,
	  \end{split}
	\end{equation*}
	where the jump condition for $u, v \in \dom H_{\widetilde{V}}$ in \eqref{def_H_Vtilde} was used in the last step. Therefore, $H_{\widetilde{V}}$ is symmetric and to see $H_{\widetilde V} = H_0$ it suffices to prove $H_0 \subset H_{\widetilde V}$. 
	Let $z \in \rho(H)$ such that 
	 \eqref{eq_smallnes_assumption} holds and 
	let $u \in \dom H_0 = \ran \mathcal{L}(z)$.
	Then, there exists  $v \in L^2(\R^\theta;\C^N)$ such that 
	\begin{equation}\label{ujuj}
		u = (H_0-z)^{-1}v=\mathcal{L}(z) v  = R_ z v - A_0( z)Vq(I+ B_0( z)Vq)^{-1}C_0( z) v.
		\end{equation} 
	We show that $u \in H^1(\mathbb{R}^\theta \setminus \Sigma; \mathbb{C}^N)=H^1 (\Omega_+;\C^N) \oplus H^1(\Omega_-; \mathbb{C}^N)$. By~\eqref{def_C_0}, Lemma~\ref{lem_BtoC_z} (iv), and \eqref{eq_Tqinv_const} we get 
	\begin{equation*}
	\begin{split}
		(I+ B_0( z)Vq)^{-1}C_0( z) v &= (I + T (\alpha \cdot \nu) Vq)^{-1}\mathfrak{J}(I +\mathcal{C}_z\widetilde{V})^{-1}\Phi_{\overline{ z}}^* v\\
		&=\cos\bigl(\tfrac{1}{2}(\alpha \cdot \nu) V \bigr)^{-1} \exp(-i(\alpha \cdot \nu) VQ) \mathfrak{J}(I+ \mathcal{C}_z \widetilde{V})^{-1}\Phi_{\overline{ z}}^* v.
		\end{split}
	\end{equation*}
	With  $A_0( z) = \Phi_z \mathfrak{J}^* $, Lemma~\ref{lem_BtoC_z}~(ii),  and $V S= \widetilde{V}$ we conclude from this
	\begin{equation} \label{krein_type}
	\begin{split}
		A_0( z)Vq&(I+ B_0( z)Vq)^{-1}C_0( z) v \\
		&=  \Phi_z V \mathfrak{J}^* q \cos\bigl(\tfrac{1}{2}(\alpha\cdot\nu)V \bigr)^{-1} \exp(-i(\alpha \cdot \nu) VQ) \mathfrak{J}(I+ \mathcal{C}_z \widetilde{V})^{-1}\Phi_{\overline{ z}}^* v\\
		&= \Phi_z V S(I+\mathcal{C}_z \widetilde{V})^{-1}   \Phi_{\overline{ z}}^*v=\Phi_ z\widetilde{V}(I+\mathcal{C}_z \widetilde{V})^{-1}   \Phi_{\overline{ z}}^*v.
	\end{split}
	\end{equation}
	Proposition~\ref{prop_Phi_z} and Lemma~\ref{lem_BtoC_z}~(i), (iv) imply  
	$\Phi_ z\widetilde{V}(I+\mathcal{C}_z \widetilde{V})^{-1}   \Phi_{\overline{ z}}^*v\in H^1(\R^\theta \setminus \Sigma;\C^N)$, and since
	$R_ z v\in \dom H = H^1(\mathbb{R}^\theta; \mathbb{C}^N)$ we conclude from \eqref{ujuj} and \eqref{krein_type} that $u\in H^1(\mathbb{R}^\theta \setminus \Sigma; \mathbb{C}^N)$. Next, we show that $u$ satisfies the transmission condition in $\dom H_{\widetilde{V}}$. 
    Note that $R_z v \in \dom H = H^1(\mathbb{R}^\theta; \mathbb{C}^N)$. This,~\eqref{krein_type},~\eqref{equation_Phi_star}, \eqref{def_C_z}, and Proposition~\ref{proposition_C_z}~(ii) (applied to the function $\varphi = \widetilde{V}(I+\mathcal{C}_z \widetilde{V})^{-1}   \Phi_{\overline{ z}}^*v)$ yield
	\begin{equation*}
	\begin{split}
		\frac{\widetilde V}{2} (\tr^+ u_+ &+ \tr^- u_-) + i (\alpha \cdot \nu) (\tr^+ u_+ - \tr^- u_-) \\
		&=  \widetilde{V} \Phi_{\overline{ z}}^*v - \widetilde{V}  \mathcal{C}_z  \widetilde{V} (I+  \mathcal{C}_z \widetilde{V})^{-1} \Phi_{\overline{ z}}^* v - i (\alpha \cdot \nu) (-i( \alpha \cdot \nu)) \widetilde{V} (I+  \mathcal{C}_z \widetilde{V})^{-1} \Phi_{\overline{ z}}^* v \\
%		\\
%		&= \widetilde{V} \Phi_{\overline{ z}}^*v - \widetilde{V} ( I + \mathcal{C}_z  \widetilde{V} ) (I+  \mathcal{C}_z \widetilde{V})^{-1} \Phi_{\overline{ z}}^* v 
		&=0.
		\end{split}
	\end{equation*}
	Hence, $u \in \dom H_{\widetilde V}$.
	Finally, we get with $(-i (\alpha \cdot \nabla) + m \beta - z I_N)R_z v = v$ and Proposition~\ref{prop_Phi_z}~(ii) that
	\begin{equation*}
	\begin{split}
		[(H_{\widetilde{V}}-z)u]_\pm &= (-i (\alpha \cdot \nabla) + m \beta - z I_N) u_\pm  \\
		&= (-i (\alpha \cdot \nabla) + m \beta -  z I_N) (R_ z v)_\pm  \\
		&\qquad - (-i (\alpha \cdot \nabla) + m \beta -  z I_N) (\Phi_ z \widetilde{V} (I+  \mathcal{C}_z \widetilde{V})^{-1} \Phi_{\overline{ z}}^* v )_\pm \\
		&= v_\pm = [(H_0-z) u]_\pm.
	\end{split}
	\end{equation*}
	Therefore, $H_0 \subset H_{\widetilde V}$ and the proof is complete.
\end{proof}

\begin{proof}[Proof of Theorem~\ref{THEO_MAIN}]
  Combining the results from Proposition~\ref{prop_main_conv} and Proposition~\ref{prop_main_limit} we find that $H_{\widetilde V}$ is self-adjoint and for $r \in (0, \frac{1}{2})$ we have
  \begin{equation*}
    \begin{split}  
      \|(H_{\varepsilon}&-  z)^{-1} - (H_{\widetilde{V}} - z)^{-1} \|_{L^2(\R^\theta;\C^N) \to L^2(\R^\theta;\C^N)} \\
      &= \|(H_{\varepsilon}-  z)^{-1} - R_ z  + A_0( z)Vq(I +B_{0}( z)Vq)^{-1}C_0( z)\|_{L^2(\R^\theta;\C^N) \to L^2(\R^\theta;\C^N)}\\
      		  &\leq C \varepsilon^{1/2-r}.
    \end{split}
  \end{equation*}
  Hence, Theorem~\ref{THEO_MAIN} is proved
\end{proof}

\begin{remark}\label{remark_smallnes_assumption}
		We point out that Proposition~\ref{prop_main_conv}, Lemma~\ref{lem_BtoC_z}~(iv), and Proposition~\ref{prop_main_limit} remain
		valid if the condition \eqref{eq_smallnes_assumption} is replaced by assumptions (i) and (ii) in Proposition~\ref{prop_I+B_conv}. In particular, Theorem~\ref{THEO_MAIN} remains valid if one assumes that
		(i) and (ii) in Proposition~\ref{prop_I+B_conv} hold.
\end{remark}

Now we prove Corollary~\ref{cor_main}, where the special case  $V= \eta I_N  + \tau \beta + \lambda i(\alpha \cdot \nu) \beta$ with real-valued functions $\eta,\tau,\lambda \in W^1_\infty(\Sigma;\R)$ is considered.

\begin{proof}[Proof of Corollary~\ref{cor_main}]
    Let $d = \eta^2-\tau^2-\lambda^2$. Then, $((\alpha \cdot \nu)V)^2 = d I_N$ and the appearance of only even powers in the power series representations of $\cos$ and $\textup{sinc}$ yield
    \begin{equation*}
    \cos\bigl(\tfrac{1}{2}(\alpha \cdot \nu)V\bigr) 
   % = \sum_{j=0}^\infty (-1)^{j}\frac{((\alpha \cdot \nu)V/2)^{2j}}{(2j)!}  
    %= \sum_{j=0}^\infty (-1)^{j}\frac{(\sqrt{d}/2)^{2j}}{(2j)!} I_N 
    = \cos\Bigl(\tfrac{\sqrt{d}}{2}\Bigr) I_N \quad \textup{ and } \quad 
    \textup{sinc}\bigl(\tfrac{1}{2}(\alpha \cdot \nu)V\bigr) 
    %= \sum_{j=0}^\infty (-1)^{j}\frac{((\alpha \cdot \nu)V/2)^{2j}}{(2j+1)!} = \sum_{j=0}^\infty (-1)^{j}\frac{(\sqrt{d}/2)^{2j}}{(2j+1)!}I_N 
    = \textup{sinc}\Bigl(\tfrac{\sqrt{d}}{2}\Bigr)I_N.
    \end{equation*}
    Hence, the condition $\cos\bigl(\tfrac{1}{2}(\alpha \cdot \nu)V\bigr)^{-1} \in W^1_\infty(\Sigma;\C^{N\times N})$ in \eqref{cond234} reduces to
    \begin{equation*}
    \inf_{x_\Sigma \in \Sigma, k \in \N_0} \abs{(2k+1)^2\pi^2- d(x_\Sigma)}>0
    \end{equation*}
     and $S = \textup{sinc}\bigl(\tfrac{1}{2}(\alpha \cdot \nu)V\bigr)\cos\bigl(\tfrac{1}{2}(\alpha \cdot \nu)V\bigr) ^{-1}=\tfrac{2}{\sqrt{d}} \tan(\sqrt{d}/2)I_N$.
Therefore, Corollary~\ref{cor_main} follows immediately from Theorem~\ref{THEO_MAIN}.
     \end{proof}
 
% Finally, we prove Corollary~\ref{cor_main_inv}.

\begin{proof}[Proof of Corollary~\ref{cor_main_inv}]
  First, we mention some identities for functions of matrices that will be useful. For $A \in \C^{N \times N}$ with $\abs{A} < 1$ the   series 
	\begin{equation*}
	  \arctan(A) = \sum_{n=0}^\infty (-1)^{n} \frac{A^{2n+1}}{2n+1} 
    \end{equation*}
	converges absolutely with respect to the matrix norm (in our case the Frobenius norm), and with the help of~\eqref{Riesz_Dunford} (see also \cite[Chapter~VII.~\S4]{C90} and \cite[Chapter~VIII.3.1]{DS58}) one finds that $\cos(\arctan(A))$ is invertible, 
	\begin{equation}\label{eq_tan_id}
	\begin{split}
			\sin(\arctan(A)) \cos(\arctan(A))^{-1} &=  A, 
			\\
			\cos(\arctan(A))^{-1} &= \cos(\arctan(A))(I_N + A^2).
			\end{split}
	\end{equation}
 
    Now, set 
    $$V = 2(\alpha \cdot \nu) \arctan\bigl(\tfrac{1}{2}(\alpha \cdot \nu)\widetilde{V}\bigr).$$ It suffices to show that $V$ fulfils the assumptions of Theorem~\ref{THEO_MAIN}, i.e. \eqref{cond123} and \eqref{cond234}, and $\widetilde{V} = S V$, as then the claim follows from Theorem~\ref{THEO_MAIN}. The assumption \eqref{cond567} implies $\|(\alpha \cdot \nu) \widetilde{V}\|_{W_\infty^1(\Sigma;\C^N)} <2$. Therefore, the power series defining $\arctan\bigl(\tfrac{1}{2}(\alpha \cdot \nu)\widetilde{V}\bigr)$ converges in $W_\infty^1(\Sigma;\C^{N\times N})$ and hence also
    $V\in W^1_\infty(\Sigma;\C^{N \times N})$. Pointwise application of the results mentioned above and $\tfrac{1}{2}(\alpha\cdot\nu)V =  \arctan\bigl(\tfrac{1}{2}(\alpha \cdot \nu)\widetilde{V}\bigr)$ show 
    that $\cos\bigl(\tfrac{1}{2}(\alpha \cdot \nu)V\bigr) = \cos\bigl(\arctan\bigl(\tfrac{1}{2}(\alpha \cdot \nu)\widetilde{V}\bigr)\bigr)$ is invertible 
    $\sigma$-a.e. and  \eqref{eq_tan_id} yields
	\begin{equation*}
			\cos\bigl(\tfrac{1}{2}(\alpha \cdot \nu)V\bigr)^{-1} =\cos\bigl(\tfrac{1}{2}(\alpha \cdot \nu)V\bigr)\bigl(I_N + \tfrac{1}{4}((\alpha \cdot \nu) \widetilde{V})^2\bigr) \quad \sigma\textup{-a.e. on } \Sigma.
	\end{equation*}
	As $\alpha \cdot \nu , \widetilde{V}, V, \cos\bigl(\tfrac{1}{2}(\alpha \cdot \nu)V\bigr) \in W^1_\infty(\Sigma;\C^{N\times N})$, this implies that also $$\cos\bigl(\tfrac{1}{2}(\alpha \cdot \nu)V\bigr)^{-1} \in W^1_\infty(\Sigma;\C^{N\times N}).$$
	Similarly, we get with \eqref{eq_tan_id}
	\begin{equation*}
		\begin{split}
			V S  &= V \textup{sinc}\bigl(\tfrac{1}{2}(\alpha \cdot \nu)V\bigr) \cos\bigl(\tfrac{1}{2}(\alpha \cdot \nu)V\bigr)^{-1}  \\
			&= 2 (\alpha \cdot \nu) \frac{(\alpha \cdot \nu) V}{2} \textup{sinc}\bigl(\tfrac{1}{2}(\alpha \cdot \nu)V\bigr) \cos\bigl(\tfrac{1}{2}(\alpha \cdot \nu)V\bigr)^{-1} \\
			&= 2 (\alpha \cdot \nu) \sin\bigl(\arctan\bigl(\tfrac{1}{2}(\alpha \cdot \nu)\widetilde{V}\bigr)\bigr) \cos\bigl(\arctan\bigl(\tfrac{1}{2}(\alpha \cdot \nu)\widetilde{V}\bigr)\bigr)^{-1} \\
			&= \widetilde{V}.
		\end{split}
	\end{equation*}
	Finally, we verify that \eqref{cond123} is satisfied. By \eqref{cond567} and the definition of $V$ we have
		\begin{equation*}
        \begin{split}
		\frac{1}{2}\norm{V}_{W^1_\infty(\Sigma;\C^{N\times N})} &\leq
		\norm{\alpha \cdot \nu}_{W^1_\infty(\Sigma;\C^{N\times N})} \sum_{n = 0}^{\infty} \bigl\| \frac{((\alpha \cdot \nu) \widetilde{V})^{2n +1}}{2^{2n+1} (2n+1)}\bigr\|_{W^1_\infty(\Sigma;\C^{N\times N})}\\
		&=\norm{\alpha \cdot \nu}_{W^1_\infty(\Sigma;\C^{N\times N})} \textup{artanh}\Big( \tfrac{\|(\alpha \cdot \nu) \widetilde{V}\|_{W^1_\infty(\Sigma;\C^{N\times N})}}{2}\Big)\\
		&< X_z.
		\end{split}
	\end{equation*} 
 Hence, \eqref{cond123} is fulfilled with $ q = \tfrac{1}{2} \chi_{(-1,1)}$. This finishes the proof of Corollary~\ref{cor_main_inv}.
\end{proof}

In the following remark we comment on the condition $\norm{V}_{W^1_\infty(\Sigma;\C^N)} \norm{q}_{L^\infty(\R;\R)}  < X_z$ in \eqref{cond123}, which is the main restriction in Theorem \ref{THEO_MAIN}, and explain that it is sharp (in a certain sense). 

\begin{remark}\label{remark_kleinheit} 
    Consider the operator $H_{\widetilde{V}}$ in~\eqref{def_H_Vtilde} for a so-called critical interaction strength in 
    the purely electrostatic setting with $m>0$ and $\Sigma$ compact and smooth, that is, we set
    $\widetilde V=\pm  2 I_N$. In this case it is known that the operator $H_{\widetilde{V}}$
    defined on functions from $H^1 (\Omega_+;\C^N) \oplus H^1(\Omega_-; \mathbb{C}^N)$ that satisfy the transmission conditions in $\dom H_{\widetilde{V}}$
    is essentially self-adjoint, but not self-adjoint, and one has $\sigma_\textup{ess}(\overline{H_{\widetilde V}}) \cap (-m, m) \neq \emptyset$; cf. \cite{BHSS22, BP22} and the references therein for more details. 
    In particular, $H_{\widetilde{V}}$ or $\overline{H_{\widetilde V}}$ can not be the norm resolvent limit of self-adjoint operators $H_\varepsilon$, as $\sigma_\textup{ess}(H_\varepsilon) = (-\infty, -m] \cup [m, \infty)$ for all $\varepsilon \in (0, \varepsilon_2]$, see, e.g., \cite[Theorem~4.7 and~(4.53)]{T92} for the case $\theta = 3$.
    However, to construct 
    approximating operators with potentials $V_\varepsilon$ 
    Corollary~\ref{cor_main_inv} would suggest the specific choices
    $q = \tfrac{1}{2} \chi_{(-1,1)}$ and 
    $$
    V = 2 (\alpha \cdot \nu )\arctan\bigl(\tfrac{\pm1}{2}(\alpha \cdot \nu)2 I_N\bigr)=\pm 2 \arctan(1) I_N=\pm \frac{\pi}{2}I_N,
    $$
    and in this situation
      \begin{equation*}
      \norm{V}_{W^1_\infty(\Sigma;\C^{N\times N})}   \norm{q}_{L^\infty(\R;\R)} = \frac{\pi}{4},
    \end{equation*}
    which is exactly the upper bound for $X_z$ from Lemma~\ref{lemma_smallness_condition}~(ii). 
    Hence, this critical case can not be treated by Theorem~\ref{THEO_MAIN} and the upper bound for $X_z$ in Lemma~\ref{lemma_smallness_condition}~(ii) is sharp in the sense that it does not allow electrostatic interactions being larger (in absolute value) or equal to the critical values $\pm 2$.
\end{remark}

%\setcounter{section}{0}
%\renewcommand{\thesection}{\Alph{section}}
%\numberwithin{equation}{section}
%\setcounter{equation}{0}
\appendix{
\section{Properties of the map $\iota$} \label{sec_appendix_iota}

In this section we investigate the map 
\begin{equation}\label{iotaapp}
	\iota: \Sigma \times \R   \to \R^{\theta}, \qquad \iota (x_\Sigma,t) := x_\Sigma + t \nu(x_\Sigma),
\end{equation}
defined in \eqref{eq_def_iota}. The goal is to show that for some $\varepsilon_A>0$ the mapping
$\iota$ and their local counterparts $\iota_l$ defined below in \eqref{eq_def_iota_l} are bi-Lipschitz on $\Sigma \times (-\varepsilon_A,\varepsilon_A)$ and $\R^{\theta-1} \times (-\varepsilon_A,\varepsilon_A)$, respectively, provided that $\varepsilon_A$ is sufficiently small.
Recall that  $x_{\Sigma_l}(x') = \kappa_l (x',\zeta_l(x'))$ for $x' \in \R^{\theta-1}$ and $l \in \{1,\dots,p\}$, where $\kappa_l$ is a rotation matrix in $\R^{\theta \times \theta}$ and $\zeta_l \in C^2_b(\R^{\theta-1};\R)$; cf. Hypothesis~\ref{hypothesis_Sigma} and \eqref{def_x_Sigma_l}.  
Besides the map $\iota$ in \eqref{iotaapp} we shall also make use  of the maps
\begin{equation}\label{eq_def_iota_l}
	\iota_l: \R^{\theta-1} \times \R   \to \R^{\theta},\qquad \iota_l(x',t) := x_{\Sigma_l}(x') + t \nu_l(x'), 
\end{equation} 
where $l \in \{1,\dots,p\}$ and the unit normal vector field $\nu_l$ on $\Sigma_l$ is given by 
$$\nu_l(x')=\frac{\kappa_l(-\nabla \zeta_l(x'),1)}{\sqrt{1+ \abs{\nabla \zeta_l(x')}^2}},\qquad x' \in \R^{\theta-1}.$$ 
Note that if $x' \in \R^{\theta-1}$ is such that $x_{\Sigma_l}(x') = x_\Sigma \in \Sigma$, then $\nu(x_\Sigma) = \nu_l(x')$ and  $\iota_l(x',t) = \iota(x_\Sigma,t)$ for all $t \in \R$. 

First, we provide a variant of the mean value theorem for vector and matrix-valued functions, which will be used frequently in the following.
\begin{lemma}\label{lem_mean_value}
	Let $k,l,n \in \N$, $U \subset \R^{n}$ be an open set, and $A \in C^1_b(U;\C^{k\times l})$. If $x,y \in U$ and the line segment connecting $x$ and $y$ is contained in $U$, then
	\begin{equation*}
	\begin{split}
		\abs{A(x)- A(y)} &\leq \sup_{\mu \in [0,1]} \Big( \sum_{j = 1}^n \abs{ (\partial_j A)(x + \mu (y -x))}^2 \Big)^{1/2} \abs{x-y} \\
		&\leq \sqrt{n} \sup_{\mu \in [0,1],j\in \{1,\dots,n\}} \abs{(\partial_j A)(x+\mu (y-x))}\abs{x-y}.
		\end{split}
	\end{equation*}
	In particular, if $U = \R^{n}$ and $l=1$, then
	\begin{equation*}
	\abs{A(x)- A(y)} \leq    \norm{D A}_{L^\infty(\R^{n};\C^{k\times n}) } \abs{x-y}, \quad  x,y \in \R^{n}.
\end{equation*}
\end{lemma}
\begin{proof}
	Recall that $| \cdot |$ denotes, depending on the argument, the absolute value, the Euclidean vector norm, or the  Frobenius matrix norm. 
	The fundamental theorem of calculus and the Cauchy-Schwarz inequality lead to 
	\begin{equation*}
	\begin{split}
		\abs{A(x)- A(y)} &= \bigg|\int_0^1 \sum_{j = 1}^n (\partial_j A)(x + \mu (y -x)) \, (x-y)_j \,d\mu\bigg| \\
		&\leq  \int_0^1  \sum_{j = 1}^n \abs{ (\partial_j A)(x + \mu (y -x))} \abs{(x-y)_j}  \, d\mu \\
		%\leq \sup_{\mu \in [0,1]} \sum_{j = 1}^n \abs{ (\partial_j A)(x + \mu (y -x))} \abs{(x-y)_j} 
		&\leq \sup_{\mu \in [0,1]}\Big( \sum_{j = 1}^n \abs{ (\partial_j A)(x + \mu (y -x))}^2 \Big)^{1/2} \abs{x-y} \\
		&\leq \sqrt{n} \sup_{\mu \in [0,1],j\in \{1,\dots,n\}} \abs{(\partial_j A)(x+\mu (y-x))} \abs{x-y}.
	\end{split}
	\end{equation*}
	The estimate for the special case $U = \R^{n}$ and $l=1$ is an immediate consequence of the above estimate.
\end{proof}

Now we turn to the properties of the mappings $\iota$ and $\iota_l$ in \eqref{iotaapp} and
\eqref{eq_def_iota_l}, respectively.

\begin{proposition} \label{prop_iota} 
  Let $\Omega_\pm, \Sigma \subset \mathbb{R}^\theta$, $\theta  \in \{ 2,3  \}$, satisfy Hypothesis~\ref{hypothesis_Sigma}, and let $\iota$ and $\iota_l$, 
  $l \in \{1,\dots,p\}$, be as in~\eqref{iotaapp} and~\eqref{eq_def_iota_l}. Then, there exists $\varepsilon_A > 0$ and  constants $C_{A,1},C_{A,2} >0$ such that the following holds:
  \begin{itemize}
	 \item[$\textup{(i)}$] For all $x',y'\in \R^{\theta-1}$ and $t,s\in (-\varepsilon_A, \varepsilon_A)$ we have for all  $l \in \{1,\dots ,p\}$
	 \begin{equation*}
	 	 C_{A,1}^{-1}(\abs{x'-y'} + \abs{t-s}) \leq \abs{\iota_l(x',t)- \iota_l(y',s)} \leq  C_{A,1}(\abs{x'-y'} + \abs{t-s}).
	 \end{equation*}
 	 \item[$\textup{(ii)}$] For all $x_\Sigma, y_\Sigma \in \Sigma$ and $t,s\in (-\varepsilon_A, \varepsilon_A)$ we have
 	  \begin{equation*}
 	 C_{A,2}^{-1}(\abs{x_\Sigma-y_\Sigma} + \abs{t-s}) \leq \abs{\iota(x_\Sigma,t)- \iota(y_\Sigma,s)} \leq  C_{A,2}(\abs{x_\Sigma - y_\Sigma} + \abs{t-s}). 
 	 \end{equation*}
  \end{itemize}
\end{proposition}
\begin{proof}
	(i) Let  $x', y' \in \mathbb{R}^{\theta-1}$ and $t,s \in (-\varepsilon_A, \varepsilon_A)$ be fixed, where $\varepsilon_A > 0$ is, at the moment, a fixed number. Using \eqref{eq_def_iota_l}, $x_{\Sigma_l}(x') = \kappa_l (x',\zeta_l(x')), x_{\Sigma_l}(y') = \kappa_l (y',\zeta_l(y'))$, 
	and Lemma~\ref{lem_mean_value} we find
	\begin{equation*} 
	\begin{split}
	|\iota_l(x',t)&- \iota_l(y',s)| \leq \abs{x_{\Sigma_l}(x') - x_{\Sigma_l}(y')} + \abs{t \nu_l(x') - s \nu_l(y')} \\
	&\leq \abs{x_{\Sigma_l}(x') - x_{\Sigma_l}(y')} + \abs{t}\abs{\nu_l(x') - \nu_l(y')} +\abs{t-s}\\
	&\leq \abs{x'-y'} + \abs{\zeta_l(x')-\zeta_l(y')} +\varepsilon_A\abs{\nu_l(x') - \nu_l(y')} +\abs{t-s}\\
	&\leq \abs{x'-y'} + \norm{\nabla \zeta_l}_{L^\infty(\R^{\theta-1};\R^{ \theta-1})}\abs{x'-y'} \\
	&\qquad +\varepsilon_A
	\norm{D \nu_l}_{L^\infty(\R^{\theta-1};\R^{\theta\times(\theta-1)})}\abs{x' - y'} +\abs{t-s}\\
	&\leq \bigl(1+\norm{\nabla \zeta_l}_{L^\infty(\R^{\theta-1};\R^{ \theta-1})} + \varepsilon_A \norm{D \nu_l}_{L^\infty(\R^{\theta-1};\R^{\theta\times(\theta-1)})} \bigr)( \abs{x'-y'} +\abs{t-s}).
	\end{split}
	\end{equation*}
	Now the second estimate in (i) follows if we
	fix  $0 < \varepsilon_A \leq 1$ and choose 
	$$  C_{A,1} \geq 1 + \max_{l \in \{1,\dots,p\}}\bigl(\norm{\nabla \zeta_l}_{L^\infty(\R^{\theta-1};\R^{\theta-1})} + \norm{D \nu_l}_{L^\infty(\R^{\theta-1};\R^{\theta \times (\theta-1)})}\bigr),$$
	which is finite since we assumed in Hypothesis~\ref{hypothesis_Sigma} that $\zeta_l \in C^{2}_b(\R^{\theta-1};\R)$ for all $l \in \{1,\dots,p\}$. 
	
	Next, we prove the first inequality in (i). We start by rewriting
	\begin{equation} \label{estimate1}
	\begin{split}
	|\iota_l(x',t)- \iota_l(y',s)|^2 &= |x_{\Sigma_l}(x') - x_{\Sigma_l}(y')|^2 + 2\spf{x_{\Sigma_l}(x') - x_{\Sigma_l}(y')}{t\nu_l(x')-s\nu_l(y')}\\
	&\qquad+ \abs{t\nu_l(x') - s\nu_l(y')}^2.
	\end{split}
	\end{equation}
	We are going to estimate all three terms on the right hand side separately. For the first one, we find with $x_{\Sigma_l}(x')=\kappa_l (x', \zeta_l(x'))$, $x_{\Sigma_l}(y') = \kappa_l (y', \zeta_l(y'))$, and as $\kappa_l$ is unitary that
	\begin{equation} \label{estimate2}
		|x_{\Sigma_l}(x') - x_{\Sigma_l}(y')|^2 = |x'-y'|^2 + |\zeta_l(x')-\zeta_l(y')|^2 \geq |x'-y'|^2.
	\end{equation}
	Next, we consider the second term on the right hand side of \eqref{estimate1}. We start by observing
	\begin{equation*}
	    \mathcal{I}:=|\langle x_{\Sigma_l}(x') - x_{\Sigma_l}(y'),t\nu_l(x')\rangle| 
	    %= \Bigg|\frac{1}{\sqrt{1+ \abs{\nabla \zeta_l(x')}^2 }} \spf{\kappa_l\begin{pmatrix}  x'-y' \\ \zeta_l(x') - \zeta_l(y')\end{pmatrix}}{t\kappa_l \begin{pmatrix} -\nabla \zeta_l (x') \\1 \end{pmatrix}}\Bigg| \\
	    =\Bigg| t\frac{\spf{x'-y'}{-\nabla\zeta_l(x')} +\zeta_l(x') - \zeta_l(y') }{\sqrt{1+ \abs{\nabla \zeta_l(x')}^2}}\Bigg|.
	\end{equation*}
	The mean value theorem shows $\zeta_l(x') - \zeta_l(y') = \spf{x'-y'}{\nabla \zeta_l (x'+ \mu(y'-x'))}$ for some $\mu \in [0,1]$.
	Using Lemma \ref{lem_mean_value} the above expression can be further estimated by 
	\begin{equation*}
	\begin{split}
	  \mathcal{I}  &\leq  \sup_{\mu  \in [0,1]} \Bigg|t\frac{\spf{y'-x'}{\nabla\zeta_l(x') - \nabla\zeta_l(x' + \mu(y'-x'))}}{\sqrt{1+ \abs{\nabla \zeta_l(x')}^2}} \Bigg|\\
	  &\leq \sup_{\mu \in [0,1]}  \varepsilon_A \abs{y'-x'}\abs{\mu(y'-x')}\norm{D \nabla \zeta_l}_{L^{\infty}(\R^{\theta-1};\R^{(\theta-1) \times (\theta -1)})}\\
	    &\leq \varepsilon_A \abs{x'-y'}^2\norm{D \nabla \zeta_l}_{L^{\infty}(\R^{\theta-1};\R^{(\theta-1)\times(\theta -1)})}.
	    \end{split}
	\end{equation*}
	 Similarly, one has
	\begin{equation*}
	\begin{split}
	|\spf{x_{\Sigma_l}(x') - x_{\Sigma_l}(y')}{s\nu_l(y')}|
	&\leq \varepsilon_A \abs{x'-y'}^2 \norm{D \nabla \zeta_l}_{L^{\infty}(\R^{\theta-1};\R^{(\theta-1) \times (\theta -1)})},
	\end{split}
 \end{equation*}
 and thus
 \begin{equation}\label{estimate3}
	2\spf{x_{\Sigma_l}(x') - x_{\Sigma_l}(y')}{t\nu_l(x')-s\nu_l(y')} \geq  -4\varepsilon_A \abs{x'-y'}^2 \norm{D \nabla \zeta_l}_{L^{\infty}(\R^{\theta-1};\R^{(\theta-1)\times (\theta -1)})}.
 \end{equation}
 Eventually, to estimate the third term on the right hand side in~\eqref{estimate1}, we use Lemma~\ref{lem_mean_value} as well as $(a-b)^2 \geq \tfrac{1}{2} a^2 - b^2$ for $ a,b >0$ and calculate
 \begin{equation} \label{estimate4}
   \begin{split}
     \abs{t\nu_l(x') - s\nu_l(y')}^2 &=\abs{(t-s)\nu_l(x') - s(\nu_l(y') - \nu_l(x'))}^2\\
     &\geq ( \abs{t-s} - \abs{s(\nu_l(y') - \nu_l(x'))})^2 \\
     &\geq \frac{1}{2}\abs{t-s}^2 - s^2 \abs{\nu_l(y')- \nu_l(x')}^2 \\
     &\geq \frac{1}{2}\abs{t-s}^2 -  \varepsilon_A^2 |x'-y'|^2 \norm{D\nu_l}_{L^{\infty}(\R^{\theta-1};\R^{\theta \times (\theta-1) })}^2.
   \end{split}
 \end{equation}
 By combining~\eqref{estimate2}--\eqref{estimate4} in~\eqref{estimate1} we obtain
	\begin{equation*}
	\begin{split}
		\abs{\iota_l(x',t)- \iota_l(y',s)}^2\geq \frac{1}{2}\abs{t-s}^2+ \abs{x'-y'}^2\bigl(1 -   4  \varepsilon_A\norm{D \nabla \zeta_l}_{L^{\infty}(\R^{\theta-1};\R^{(\theta-1) \times (\theta-1) })} &\\
		-\varepsilon_A^2\norm{D\nu_l}_{L^{\infty}(\R^{\theta-1};\R^{\theta\times (\theta-1) })}^2\bigr)&.
		\end{split}
	\end{equation*}
	As before we conclude from $\zeta_l \in C^2_b(\R^{\theta-1};\R)$ that for $\varepsilon_A>0$ sufficiently small and $C_{A,1}>0$ sufficiently large also the first inequality in (i) is fulfilled. 
 	
 	(ii) We fix  $x_\Sigma, y_\Sigma \in \Sigma$ and $t,s \in (-\varepsilon_A, \varepsilon_A)$. Let us first assume that $x_\Sigma,y_\Sigma \in  \Sigma_l$ for some $l \in \{1,\dots,p\} $. Then, there exist $x',y'\in \R^{\theta-1}$ such that $x_\Sigma = x_{\Sigma_l}(x')$ and $y_\Sigma = x_{\Sigma_l}(y')$, and therefore $\iota(x_\Sigma,t) = \iota_l(x',t)$ and $\iota(y_\Sigma,s) = \iota_l(y',s)$. In this case we see 
 	$$\abs{x_\Sigma - y_\Sigma } = \sqrt{\abs{x'-y'}^2 + \abs{\zeta_l(x') - \zeta_l(y')}^2}$$ and therefore combining 
 	\begin{equation*}
 	\abs{x'-y'}  \leq \abs{x_\Sigma -y_\Sigma} \leq \abs{x'-y'}\sqrt{1 + \norm{ \nabla \zeta_l}_{L^{\infty}(\R^{\theta-1};\R^{ \theta-1 })}^2}
 	\end{equation*}
 	with (i) yields (ii). It remains to consider the case where $x_\Sigma, y_\Sigma \in \Sigma$ and there is no $l \in\{1,\dots,p\}$ such that $x_\Sigma,y_\Sigma \in \Sigma _l$. Then, (ii) and (iii) from Hypothesis \ref{hypothesis_Sigma} imply $\abs{x_\Sigma - y_\Sigma} \geq \varepsilon_0$, where $\varepsilon_0$ is the number specified in Hypothesis \ref{hypothesis_Sigma}.  Choose $\varepsilon_A \leq \varepsilon_0/6$. Then, $\abs{x_\Sigma -y_\Sigma} \geq  6 \varepsilon_A$, $\abs{t\nu(x_\Sigma) - s \nu(y_\Sigma)} \leq 2 \varepsilon_A$, and $\abs{t-s} \leq 2\varepsilon_A$ yield
 	\begin{equation*}
 		\abs{\iota(x_\Sigma,t) - \iota(y_\Sigma,s)} \leq \abs{x_\Sigma-y_\Sigma} + 2\varepsilon_A \leq \frac{4}{3} \abs{x_\Sigma -y_\Sigma}  \leq \frac{4}{3} (\abs{x_\Sigma -y_\Sigma}  +\abs{t-s})
 	\end{equation*}
 	and
 	\begin{equation*}
 	\begin{split}
 		\frac{1}{2}\bigl(\abs{x_\Sigma -y_\Sigma} + \abs{t-s}\bigr) &\leq \frac{\abs{x_\Sigma -y_\Sigma}}{2} + \varepsilon_A =   \frac{\abs{x_\Sigma -y_\Sigma}}{2} +3\varepsilon_A -2 \varepsilon_A \\
 		&\leq \abs{x_\Sigma -y_\Sigma} -2\varepsilon_A \leq  \abs{\iota(x_\Sigma,t) - \iota(y_\Sigma,s)},
 		\end{split}
 	\end{equation*}
 	which imply (ii) also in this case.
\end{proof}

Eventually, we state a useful consequence of Proposition~\ref{prop_iota}.

\begin{corollary} \label{cor_eps_A}
  Assume that $\Sigma, \Omega_\pm \subset \mathbb{R}^\theta$, $\theta \in \{ 2, 3 \}$, satisfy Hypothesis~\ref{hypothesis_Sigma} and let $\varepsilon_A$ be as in Proposition~\ref{prop_iota}. Then, the following holds:
  \begin{itemize}
    \item[(i)] For any $x_\Sigma \in \Sigma$ and $t \in (0, \varepsilon_A)$ one has $x_\Sigma + t \nu(x_\Sigma) \in \Omega_-$.
    \item[(ii)] For any $x_\Sigma \in \Sigma$ and $t \in (-\varepsilon_A, 0)$ one has $x_\Sigma + t \nu(x_\Sigma) \in \Omega_+$.
  \end{itemize}
\end{corollary}
\begin{proof}
  We show item~(i), the proof of assertion~(ii) follows the same lines. The claim will be verified by an indirect proof. Assume that there are $x_\Sigma \in \Sigma$ and $t \in (0, \varepsilon_A)$ such that $x_\Sigma + t \nu(x_\Sigma) \notin \Omega_-$. Since $\nu$ is pointing outwards of $\Omega_+$, we have for small $\mu > 0$ that $x_\Sigma + \mu t \nu(x_\Sigma) \in \Omega_-$. By continuity, this implies that there exists $\mu_0 \in (0, 1]$ such that $x_\Sigma + \mu_0 t \nu(x_\Sigma) \in \Sigma$. However, we obtain  from  Proposition  \ref{prop_iota} for all $y_\Sigma \in \Sigma$ with a constant $C_{A,2} >0$ the inequality  $ \abs{x_\Sigma + \mu_0 t \nu(x_\Sigma)-y_\Sigma} \geq  C_{A,2}^{-1}\mu_0 t > 0$;
  this is a contradiction.
\end{proof}

\section{Proof of the estimate \eqref{ESTIMATE_B_BAR_B_TILDE}} 
\label{sec_appendix_diff_B}

Let $z \in \rho(H)$ and $\widetilde{B}_\varepsilon(z): L^2((-1,1);L^2(\Sigma;\C^N)) \to L^2((-1,1);L^2(\Sigma;\C^N))$ and $\overline{B}_\varepsilon(z): L^2((-1,1);H^{1/2}(\Sigma;\C^N)) \to L^2((-1,1);H^{1/2}(\Sigma;\C^N))$ be the operators defined by~\eqref{wtwtb} and~\eqref{def_B_bar}, respectively. In this appendix we show that $\widetilde{B}_{\varepsilon}(z) - \overline{B}_\varepsilon(z)$  can be extended to a bounded operator from $L^2((-1,1);L^2(\Sigma;\C^N))$ to $L^2((-1,1);H^{1/2}(\Sigma;\C^N))$ and that
\begin{equation}\label{ESTIMATE_B_BAR_B_TILDE_app}
\begin{split}
\norm{\widetilde{B}_{\varepsilon}(z) - \overline{B}_\varepsilon(z)}_{0 \to 1/2} \leq C (\varepsilon + \varepsilon\abs{\log(\varepsilon)})^{1/2}
\end{split}
\end{equation}
for some $C>0$, which is used in~\eqref{ESTIMATE_B_BAR_B_TILDE} in \textit{Step~3} in the proof of Proposition~\ref{prop_B_conv}.
With  \eqref{def_B_tilde} and \eqref{eq_B_bar}  one obtains for $ f\in L^2((-1,1);H^{1/2}(\Sigma;\C^N))$
\begin{equation}\label{eq_diff_B_tilde_bar}
	\begin{split}
	(\widetilde{B}_\varepsilon( z) - \overline{B}_\varepsilon(z))f(t)(x_\Sigma) = &\int_{-1}^1\int_\Sigma (G_ z(x_\Sigma-y_\Sigma +\varepsilon t\nu(x_\Sigma) - \varepsilon s\nu(y_\Sigma)) \\
	&- G_ z(x_\Sigma-y_\Sigma +\varepsilon (t-s)\nu(x_\Sigma))) f(s)(y_\Sigma) \,d\sigma(y_\Sigma) \,ds
	\end{split}
\end{equation}
for a.e. $t \in (-1,1)$ and for $\sigma$-a.e. $x_\Sigma \in \Sigma$, where $G_ z$ is the integral kernel of $R_z = (H-z)^{-1}$; 
cf. \eqref{eq_G_z_2D}--\eqref{eq_G_z_3D}. Thus, in order to show \eqref{ESTIMATE_B_BAR_B_TILDE_app}, we proceed as follows: We prove in Proposition~\ref{prop_diff_b} that for fixed $t \neq s \in (-1,1)$ the operator formally acting on $\psi \in L^2(\Sigma;\C^N)$  as
\begin{equation}\label{eq_def_b_t,s,eps}
		\begin{split}
	b_{t,s,\varepsilon}(z) \psi(x_\Sigma) =& \int_\Sigma (G_ z(x_\Sigma-y_\Sigma +\varepsilon t\nu(x_\Sigma) - \varepsilon s\nu(y_\Sigma))\\
	&- G_ z(x_\Sigma-y_\Sigma +\varepsilon (t-s)\nu(x_\Sigma))) \psi(y_\Sigma) \,d\sigma(y_\Sigma), \quad x_\Sigma \in \Sigma,
	\end{split}
\end{equation}
 gives rise to a bounded operator from $L^2(\Sigma;\C^N)$ to $H^{1/2}(\Sigma;\C^N)$ and we prove an  estimate for its operator norm. Then we show in Lemma~\ref{lem_Bochner_measuarble} that the map $(s, t) \mapsto b_{t,s,\varepsilon}(z)$ is measurable and use \eqref{eq_diff_B_tilde_bar} to transfer the results from $b_{t,s,\varepsilon}(z)$ to $ \widetilde{B}_\varepsilon( z) - \overline{B}_\varepsilon(z)$.

In the following, we always assume  $\varepsilon \in (0,\varepsilon_2)$ with $\varepsilon_2 > 0$ given by~\eqref{eq_eps_2}. Recall that $\varepsilon_1$ and $\varepsilon_A$ are the numbers that are specified in Propositions~\ref{prop_tubular_neighbourhood} and~\ref{prop_iota}, respectively.  Since $\varepsilon_1 < \varepsilon_A$, see the proof of  Proposition~\ref{prop_tubular_neighbourhood}, we conclude from \eqref{eq_eps_2} that $\varepsilon_2 < \tfrac{\varepsilon_A}{2}$. Define for $t \neq s \in (-1,1)$ and $x_\Sigma, y_\Sigma \in \Sigma$
\begin{equation}\label{eq_diff_G_z}
  \begin{split}
	\Delta G_ z (x_\Sigma,y_\Sigma,t,s) &:= G_ z(x_\Sigma-y_\Sigma +\varepsilon t\nu(x_\Sigma)  - \varepsilon s\nu(y_\Sigma)) - G_ z(x_\Sigma-y_\Sigma +\varepsilon (t-s)\nu(x_\Sigma)).
  \end{split}
\end{equation} 
Moreover, we introduce for $t \neq s \in (-1,1)$ and $x_\Sigma,y_\Sigma \in \Sigma$ the quantities
\begin{equation*}
	\begin{split}
	z_0(x_\Sigma,y_\Sigma,t,s) &:= {x_\Sigma-y_\Sigma +\varepsilon t\nu(x_\Sigma) - \varepsilon s\nu(y_\Sigma)} = \iota(x_\Sigma,\varepsilon t) - \iota(y_\Sigma,\varepsilon s),\\
	z_1(x_\Sigma,y_\Sigma,t,s) &:= {x_\Sigma-y_\Sigma +\varepsilon (t-s)\nu(x_\Sigma)} = \iota(x_\Sigma,\varepsilon (t-s)) - \iota(y_\Sigma,0) ,\\
	z_\mu(x_\Sigma,y_\Sigma,t,s) &:= \mu z_0(x_\Sigma,y_\Sigma,t,s) + (1-\mu)z_1(x_\Sigma,y_\Sigma,t,s)\\
	&=\iota(x_\Sigma,( \mu \varepsilon  t + (1-\mu)\varepsilon (t-s)) - \iota(y_\Sigma, \mu \varepsilon s)\quad \text{for } \mu \in (0,1),\\
	L(x_\Sigma,y_\Sigma,t,s) &:= \abs{x_\Sigma-y_\Sigma} +\abs{\varepsilon (t-s)}.\\
	\end{split}
\end{equation*}
Then, $\Delta G_z(x_\Sigma,y_\Sigma,t,s) = G_z(z_0(x_\Sigma,y_\Sigma,t,s))-G_z(z_1(x_\Sigma,y_\Sigma,t,s))$. 
It follows from Proposition~\ref{prop_iota}~(ii) that for $\mu \in [0,1]$ 
\begin{equation}\label{eq_estimate_L}
	C_{A,2}^{-1} L(x_\Sigma,y_\Sigma,t,s)\leq \abs{z_\mu(x_\Sigma,y_\Sigma,t,s)} \leq C_{A,2} L(x_\Sigma,y_\Sigma,t,s)
\end{equation} 
holds. To shorten notation we also set 
\begin{equation} \label{def_omega}
  \omega  := \textup{Im} \sqrt{ z^2 - m^2}C_{A,2}^{-1} > 0, \quad z \in \rho(H).
\end{equation}
Furthermore, until Lemma \ref{lem_Bochner_measuarble} we fix $t \neq s \in (-1,1)$ and hence omit the arguments $t,s$ in the functions $L$, $\Delta G_z$, $z_0$, $z_\mu$, and $z_1$.

\begin{lemma}\label{lem_diff_G_z}
	Let $G_ z$ be the integral kernel of $R_ z$ in \eqref{eq_G_z_2D}--\eqref{eq_G_z_3D}, $\Delta G_z$ as in \eqref{eq_diff_G_z}, $l \in \{1,\dots,p\}$, and $x_{\Sigma_l}$ as in \eqref{def_x_Sigma_l}. Then, the following is true:
	\begin{itemize}
		\item[$\textup{(i)}$] There exists $C>0$ which does not depend on $\varepsilon$, $t$, and $s$ such that 
		\begin{equation*}
		\abs{\Delta G_ z (x_\Sigma,y_\Sigma)} \leq C \varepsilon (L(x_\Sigma,y_\Sigma)+L(x_\Sigma,y_\Sigma)^{{1-\theta}}) e^{-\omega L(x_\Sigma,y_\Sigma)} 
		\end{equation*}
		for all $ x_\Sigma,y_\Sigma \in \Sigma$.
		\item[$\textup{(ii)}$] There exists $C>0$ which does not depend on $\varepsilon$, $t$, and $s$ such that 
		\begin{equation*}
		\bigl|\frac{d}{dx_k'}\Delta G_ z (x_{\Sigma_l}(x'),y_\Sigma)\bigr| \leq C \varepsilon (L(x_{\Sigma_l}(x'),y_\Sigma) +L(x_{\Sigma_l}(x'),y_\Sigma)^{{-\theta}} ) e^{-\omega L(x_{\Sigma_l}(x'),y_\Sigma)} 
		\end{equation*}
		for all $k \in \{1,\dots,\theta-1\}$, $ y_\Sigma \in \Sigma $, and $ x' \in x_{\Sigma_l}^{-1}(\Sigma)$.
	\end{itemize}
\end{lemma}
\begin{proof}
	Before we prove (i) and (ii) we show useful estimates of the difference $z_0(x_\Sigma,y_\Sigma) - z_1(x_\Sigma,y_\Sigma)$ and $\partial_j G_z$, $j \in \{1,\dots,\theta\}$.
	Since $\varepsilon_2 < \tfrac{\varepsilon_A}{2}$, it follows from Proposition  \ref{prop_iota} (ii)   that
	\begin{equation}\label{eq_diff_G_z_est_1}
	\begin{split}
	|z_0(x_\Sigma,y_\Sigma) - z_1(x_\Sigma,y_\Sigma)| &= |\varepsilon t\nu(x_\Sigma) - \varepsilon s\nu(y_\Sigma) - \varepsilon (t-s)\nu(x_\Sigma)|\\
	&= \varepsilon \abs{s} \abs{\nu(x_\Sigma) - \nu(y_\Sigma)} \\
	% \leq \frac{\varepsilon}{\varepsilon_2} \abs{ \varepsilon_2\nu(x_\Sigma) - \varepsilon_2 \nu(y_\Sigma) } \leq  \frac{\varepsilon}{\varepsilon_2} \big( \abs{x_\Sigma-y_\Sigma} + \big|x_\Sigma + \varepsilon_2\nu(x_\Sigma) - y_\Sigma - \varepsilon_2\nu(y_\Sigma) \big|\big)\\
	&\leq \frac{\varepsilon}{\varepsilon_2} ( \abs{x_\Sigma-y_\Sigma} + \abs{\iota(x_\Sigma,\varepsilon_2) - \iota(y_\Sigma,\varepsilon_2) })\\
	&\leq  \frac{1+C_{A,2}}{\varepsilon_2}  \varepsilon \abs{x_\Sigma-y_\Sigma}\\
	&\leq  \frac{1+C_{A,2}}{\varepsilon_2}  \varepsilon L(x_\Sigma,y_\Sigma)
	\end{split}
	\end{equation}
	for all $x_\Sigma,y_\Sigma \in \Sigma$. Next, we estimate $\partial_j G_z $, $j \in \{1,\dots,\theta\}$. For $\partial_j G_z$ 
	we obtain for  $\theta =2$ and $x\in \R^{2}\setminus\{0\}$ from \eqref{eq_G_z_2D}
	\begin{equation*}
	\begin{split}
	\partial_j G_z(x) =& \frac{\sqrt{ z^2-m^2}}{2\pi} K_1\Bigl(-i \sqrt{ z^2-m^2}\abs{x}\Bigr)\biggl(\frac{\alpha_j }{\abs{x}}- \frac{x_j(\alpha \cdot x)}{\abs{x}^3} \biggr) 
	\\
	&+i\frac{z^2-m^2}{4\pi} \Bigl(K_0\Bigl(-i \sqrt{ z^2-m^2}\abs{x}\Bigr) + K_2\Bigl(-i \sqrt{ z^2-m^2}\abs{x}\Bigr) \Bigr)\frac{x_j(\alpha \cdot x)}{\abs{x}^2} \\
	&+\frac{i \sqrt{ z^2-m^2}}{2\pi} K_1\Bigl(-i \sqrt{ z^2-m^2}\abs{x}\Bigr)\frac{x_j}{\abs{x}}(m\beta +  z I_2),
	\end{split}
	\end{equation*}
	where we used   $K_0' = - K_1$ and $K_1' = -\tfrac{1}{2}(K_0 + K_2)$; cf. \cite[\S 10.29(i)]{DLMF}. 
	For $\theta =3$ and $x \in \R^{3}\setminus\{0\}$ we obtain from \eqref{eq_G_z_3D}
	\begin{equation*}
	\begin{split}
	\partial_j G_z(x) &=  \Biggl(i\Bigl( 1 - i \sqrt{ z^2 -m^2}\abs{x} \Bigr)\biggl(\alpha_j - \frac{2x_j(\alpha \cdot x)}{\abs{x}^2} \biggr) + \sqrt{ z^2 -m^2}\frac{x_j(\alpha \cdot x)}{\abs{x}}\\
	&\phantom{=}+\biggl(  z I_4 + m \beta + i\Big( 1 - i \sqrt{ z^2 -m^2}\abs{x} \Big) \frac{ \alpha \cdot x }{\abs{x}^2}\biggr)\Bigl( i\sqrt{ z^2 -m^2}x_j \abs{x}
		-x_j\Bigr) \Biggr) \\
		&\qquad \qquad \qquad \qquad \qquad \qquad \qquad \qquad \qquad \qquad \qquad \qquad \qquad \cdot\frac{e^{i\sqrt{ z^2 -m^2}  \abs{x}}}{4 \pi \abs{x}^3}.
	\end{split}
	\end{equation*}
	By the well-known asymptotic expansions of the modified Bessel functions from \cite[\S10.25~(ii) \& \S10.30~(i)]{DLMF}, there exists $R > 0$ such that for all $x \in \mathbb{R}^\theta \setminus \{ 0 \}$ with $|x| \leq R$ one has
	\begin{equation*}
	  \Bigl| K_{n}\Bigl(-i\sqrt{z^2-m^2}\abs{x}\Bigr) \Bigr| \leq C \begin{cases} \abs{\log \abs{x}}, & n=0, \\ \abs{x}^{-n}, & n \in \{ 1,2 \}, \end{cases} 
	\end{equation*}
	and for all $x \in \mathbb{R}^\theta \setminus \{ 0 \}$ with $|x| > R$ and $n \in \{ 0, 1, 2 \}$
	\begin{equation*}
	  \Bigl| K_{n}\Bigl(-i\sqrt{z^2-m^2}\abs{x}\Bigr) \Bigr| \leq C e^{-\textup{Im}\sqrt{z^2-m^2}\abs{x}}.
	\end{equation*}
	Using this in the above formulas for $\partial_j G_z$, one concludes that
	\begin{equation}\label{eq_diff_G_z_est_2}
	\begin{split}
	\sup_{ j \in \{1,\dots \, ,\theta\} } \abs{\partial_j G_ z (x)} \leq C (1+\abs{x}^{-\theta} ) e^{-\textup{Im}\sqrt{ z^2-m^2}\abs{x}}, \quad  x \in \R^\theta \setminus \{0\}.
	\end{split}
	\end{equation}
	
	(i) Applying Lemma \ref{lem_mean_value}, \eqref{eq_diff_G_z_est_1}, \eqref{eq_diff_G_z_est_2}, and \eqref{eq_estimate_L} yields
	\begin{equation*}
	\begin{split}
	\abs{\Delta G_ z (x_\Sigma,y_\Sigma)} &=  \abs{ G_ z(z_0(x_\Sigma,y_\Sigma)) -  G_ z(z_1(x_\Sigma,y_\Sigma))}\\
	&\leq C\sup_{\mu \in [0,1], j \in \{1,\dots \, ,\theta\}  } \abs{\partial_j G_z (z_\mu(x_\Sigma,y_\Sigma))}  \abs{z_0(x_\Sigma,y_\Sigma) - z_1(x_\Sigma,y_\Sigma)} \\
	&\leq C\sup_{\mu \in [0,1], j \in \{1,\dots \, ,\theta\} } \abs{\partial_j G_z (z_\mu(x_\Sigma,y_\Sigma))} \varepsilon L(x_\Sigma,y_\Sigma) \\
	&\leq C\sup_{\mu \in [0,1]}(1+\abs{z_\mu(x_\Sigma,y_\Sigma)}^{-\theta} ) e^{-\textup{Im}\sqrt{ z^2-m^2}\abs{z_\mu(x_\Sigma,y_\Sigma)}} \varepsilon L(x_\Sigma,y_\Sigma) \\
	&\leq C \varepsilon (L(x_\Sigma,y_\Sigma)+L(x_\Sigma,y_\Sigma)^{{1-\theta}}) e^{-\omega L(x_\Sigma,y_\Sigma)}
	\end{split}
	\end{equation*}
	for all $ x_\Sigma,y_\Sigma \in \Sigma$,
	where $\omega$ is defined in~\eqref{def_omega} and $C>0$ is a constant which does not depend on $\varepsilon$, $t$, and $s$. Hence, the claim in~(i) is shown.
	
	(ii) For $k \in \{1,\dots,\theta-1\}$, $y_\Sigma \in \Sigma$, and $x' \in x_{\Sigma_l}^{-1}(\Sigma)$ we compute
	\begin{equation*}
	\begin{split}
	&\frac{d}{dx_k'}\Delta G_ z (x_{\Sigma_l}(x'),y_\Sigma)
	= \frac{d}{dx_k'} (G_ z(z_0(x_{\Sigma_l}(x'),y_\Sigma)) - G_ z(z_1(x_{\Sigma_l}(x'),y_\Sigma)) )\\
	&\quad= \sum_{j =1}^{\theta} (\partial_j G_ z)(z_0(x_{\Sigma_l}(x'),y_\Sigma)) \frac{d}{d x_k'}(z_0(x_{\Sigma_l}(x'),y_\Sigma))_j\\
	&\qquad \qquad -\sum_{j =1}^{\theta} (\partial_j G_ z)(z_1(x_{\Sigma_l}(x'),y_\Sigma)) \frac{d}{d x_k'}(z_1(x_{\Sigma_l}(x'),y_\Sigma))_j\\
	%&= \sum_{j =1}^{\theta} \big( (\partial_j G_ z)(z_0(x_{\Sigma_l}(x'),y_\Sigma)) - (\partial_j G_ z)(z_1(x_{\Sigma_l}(x'),y_\Sigma)) \big) \frac{d}{d x_k'}(z_0(x_{\Sigma_l}(x'),y_\Sigma) )_j\\
	%&\phantom{=}+\sum_{j =1}^{\theta} (\partial_j G_ z)(z_1(x_{\Sigma_l}(x'),y_\Sigma)) \frac{d}{d x_k'}\left(z_0(x_{\Sigma_l}(x'),y_\Sigma) - z_1(x_{\Sigma_l}(x'),y_\Sigma)\right)_j\\
	&\quad= \sum_{j =1}^{\theta} ( (\partial_j G_ z)(z_0(x_{\Sigma_l}(x'),y_\Sigma))  - (\partial_j G_ z)(z_1(x_{\Sigma_l}(x'),y_\Sigma)) )\frac{d}{d x_k'}(z_0(x_{\Sigma_l}(x'),y_\Sigma))_j\\
	&\qquad \qquad +\sum_{j =1}^{\theta} \varepsilon s(\partial_j G_ z)(z_1(x_{\Sigma_l}(x'),y_\Sigma)) \frac{d}{d x_k'}(\nu_l(x'))_j,
	\end{split}
	\end{equation*}
	where $\nu_l(x') = \nu(x_{\Sigma_l}(x'))$ was used in the last step.
	To estimate the second sum we use \eqref{eq_estimate_L}, \eqref{eq_diff_G_z_est_2}, $\zeta_l \in C^2_b(\R^{\theta-1};\C^N)$,  
	and $1 + a^{-\theta} \leq 2 (a+ a^{-\theta})$ for $a >0$ and obtain
	\begin{equation*}
	\begin{split}
	\biggl|\sum_{j =1}^{\theta} &\varepsilon s(\partial_j G_ z)(z_1(x_{\Sigma_l}(x'),y_\Sigma)) \frac{d}{d x_k'}(\nu_l(x'))_j\biggr |\\
	&\leq C \varepsilon \sup_{ j \in \{1,\dots ,\theta\}  } \abs{(\partial_j G_ z)(z_1(x_{\Sigma_l}(x'),y_\Sigma))} \norm{D \nu_l}_{L^\infty(\R^{\theta-1};\R^{\theta \times (\theta-1)})}\\
	&\leq C \varepsilon (1+L(x_{\Sigma_l}(x'),y_\Sigma)^{-\theta}) e^{-\omega L(x_{\Sigma_l}(x'),y_\Sigma)} \\
	&\leq C \varepsilon (L(x_{\Sigma_l}(x'),y_\Sigma)+L(x_{\Sigma_l}(x'),y_\Sigma)^{-\theta}) e^{-\omega L(x_{\Sigma_l}(x'),y_\Sigma)}.
	\end{split}
	\end{equation*}
	 For the remaining part given by
	\begin{equation}\label{eq_diff_G_z_est_3}
	\sum_{j =1}^{\theta} ( (\partial_j G_ z)(z_0(x_{\Sigma_l}(x'),y_\Sigma))  - (\partial_j G_ z)(z_1(x_{\Sigma_l}(x'),y_\Sigma)) )\frac{d}{d x_k'}(z_0(x_{\Sigma_l}(x'),y_\Sigma))_j
	\end{equation} 
	we proceed in the same way as in the proof of (i). Using  Lemma \ref{lem_mean_value} and $\zeta_l \in C^2_b(\R^{\theta-1};\R)$ one can show that the absolute value of the expression in \eqref{eq_diff_G_z_est_3} is bounded by the term 
 	\begin{equation*}
 	C \sup_{\mu \in [0,1], \,n,j \in \{1,\dots ,\theta\}} \abs{\partial_n \partial_j G_ z (z_\mu(x_{\Sigma_{l}}(x'),y_\Sigma))} \varepsilon L(x_{\Sigma_{l}}(x'),y_\Sigma).
 	\end{equation*}
 	Moreover, in a similar way as in~\eqref{eq_diff_G_z_est_2} one can prove
 	\begin{equation*}
 	\begin{split}
 	\sup_{ n,j \in \{1,\dots ,\theta\} } \abs{\partial_n \partial_j G_ z (x)} \leq C(1+\abs{x}^{-\theta-1} ) e^{-\textup{Im}\sqrt{ z^2-m^2}\abs{x}},\quad 
 	x \in \R^\theta \setminus \{0\}.
 	\end{split}
 	\end{equation*}
 	Combining these observations  with \eqref{eq_estimate_L} yield the claim in~(ii). 
\end{proof}

To estimate the operator norm of $b_{t,s,\varepsilon}(z)$ in Proposition \ref{prop_diff_b} below we will make use of a partition of unity 
for $\Sigma$ subordinate to $(W_l )_{l\in \{1,\dots,p\}}$ with the additional property that the derivatives are uniformly bounded; the existence of 
such a partition of unity (in the case that $\Sigma$ is unbounded) is shown in the next lemma.

\begin{lemma}\label{lem_part_of_unity} 
	Let $\Sigma \subset \mathbb{R}^\theta$, $\theta  \in \{ 2,3  \}$, satisfy Hypothesis~\ref{hypothesis_Sigma}. Then, there exists a partition of unity $(\varphi_l)_{l \in \{1,\dots,p \}}$ for $\Sigma$ subordinate to the open cover $(W_l)_{l \in \{1,\dots,p\}}$ of $\Sigma$ such that $ \varphi_l \in C^\infty_b(\R^{\theta};\R)$ for all $l \in \{1,\dots,p\}$. 
\end{lemma}
\begin{proof}
	According to \cite[Appendix A, Lemmas~1.2 and~1.3]{S92} there exists a sequence $(x_n)_{n \in \N} \subset \mathbb{R}^\theta$, $M \in \N$, $0<\delta < \frac{\varepsilon_0}{2}$, and a sequence of real-valued $C^\infty$-functions $(\phi_n)_{n \in \N}$  such that  $(B(x_n,\delta))_{n \in \N}$ is an open cover of $\R^\theta$, $(\phi_n)_{n \in \N}$  is a partition of unity for $\mathbb{R}^\theta$ subordinate
	to this open cover, $\supp \phi_n \subset B(x_n,\delta)$ for all $n \in \N$, every point $x \in \R^\theta$ is contained in at most $M$ of the sets $B(x_n,\delta)$, and the derivatives of the functions $\phi_n$ are uniformly bounded.  Next, we define the set $Y := \{ x_n: B(x_n,\delta) \cap \Sigma \neq \emptyset \}$. Note that for all $x_n \in Y$ there exists $l \in \{1,\dots, p\}$ such that $B(x_n,\delta) \subset W_l $. In fact, since $B(x_n,\delta) \cap \Sigma \neq \emptyset$, there exists $y_\Sigma \in B(x_n,\delta) \cap \Sigma$ and thus, item~(ii) in Hypothesis~\ref{hypothesis_Sigma} implies $B(y_\Sigma,\varepsilon_0) \subset W_l$ for an $l \in \{1,\dots,p\}$. Hence, for any $ y \in B(x_n,\delta)$ one has
	\begin{equation*}
	\abs{y-y_\Sigma} \leq \abs{y -x_n} + \abs{x_n - y_\Sigma} < 2 \delta < \varepsilon_0, 
	\end{equation*}
	which shows $B(x_n,\delta) \subset W_l$. Next, we define  $I_1 := \{ n: x_n \in Y, B(x_n, \delta) \subset W_1 \}$ and for $l \in \{2,\dots,p\}$ we introduce $I_l := \{ n: x_n \in Y, B(x_n, \delta) \subset W_l, B(x_n, \delta) \not\subset  W_k, k \in \{1,\dots, l-1\} \}$. Then, it is not difficult to see that $\varphi_l = 
	\sum_{n \in I_l}\phi_n $
	is a partition of unity having the claimed properties. 
\end{proof}

\begin{proposition}\label{prop_diff_b}
	Let $t \neq s \in (-1,1)$ and $\varepsilon \in (0,\varepsilon_2)$. Then, the operator formally defined by \eqref{eq_def_b_t,s,eps}
	gives rise to a bounded operator $b_{t,s,\varepsilon}(z): L^2(\Sigma;\C^N) \to H^{1/2}(\Sigma;\C^N)$ and there 
	exists $C>0$ which does not depend on $\varepsilon$, $t$, and $s$ such that 
	\begin{equation}\label{aberja}
	\norm{b_{t,s,\varepsilon}( z)}_{L^2(\Sigma;\C^N) \to H^{1/2}(\Sigma;\C^N)} \leq C (\varepsilon + \varepsilon\abs{ \log( \varepsilon\abs{t-s})})^{1/2} \frac{1}{\abs{t-s}^{1/2}}.
	\end{equation}
\end{proposition}

\begin{proof}
	We split this proof into four steps. In \textit{Step~1} we verify the preliminary estimate
	\begin{equation}\label{eq_diff_b_1}
	\begin{split}
	\sup_{x_\Sigma \in \Sigma}\int_\Sigma (L(x_\Sigma,y_\Sigma) + L(x_\Sigma,y_\Sigma)^{j} )e^{-\omega L(x_\Sigma,y_\Sigma)} \, d\sigma(y_\Sigma) &
	\\
	\leq C\begin{cases}
	1+\abs{\log(\varepsilon \abs{t-s})}, & j= 1-\theta, \\
	\frac{1}{\varepsilon \abs{t-s}}, & j = -\theta,
	\end{cases}&
	\end{split}
	\end{equation}
	which will be used in \textit{Step~2} and \textit{Step~3} to obtain bounds for $b_{t,s,\varepsilon}(z)$ 
	viewed as an operator from $L^2(\Sigma;\C^N)$ to $L^2(\Sigma;\C^N)$ and from $L^2(\Sigma;\C^N)$ to $H^1(\Sigma;\C^N)$, respectively. Finally, we conclude with an interpolation argument
	\eqref{aberja} in \textit{Step~4}.
	
	\textit{Step~1.} 
	Let $x_\Sigma \in \Sigma$ and $j \in \{1-\theta,-\theta\}$. Recall that $\Sigma $ satisfies Hypothesis~\ref{hypothesis_Sigma} and let $(\varphi_l)_{l \in \{1,\dots,p\}}$ be the partition of unity from Lemma \ref{lem_part_of_unity}. Using the definition of the boundary integral, we can write
	\begin{equation*}
	\begin{split}
	\int_\Sigma (L(x_\Sigma&,y_\Sigma) + L(x_\Sigma,y_\Sigma)^{j} )e^{-\omega L(x_\Sigma,y_\Sigma)} \, d\sigma(y_\Sigma)\\
	&=\sum_{n=1}^p \int_{x_{\Sigma_n}^{-1}(\Sigma) }( L(x_\Sigma,x_{\Sigma_n}(y'))+ L(x_\Sigma,x_{\Sigma_n}(y'))^{j}) e^{-\omega L(x_\Sigma,x_{\Sigma_n}(y'))}   \\
	&\qquad\qquad\qquad\qquad\qquad\qquad\qquad\qquad\cdot\varphi_n(x_{\Sigma_n}(y'))\sqrt{1+\abs{\nabla \zeta_n(y')}^2}\,dy'.
	\end{split}
	\end{equation*}
	Hence, $0 \leq \varphi_n \leq 1$, $\zeta_n \in C^2_b(\R^{\theta-1};\R)$, and $x_{\Sigma_n}^{-1}(\Sigma)  \subset \R^{\theta-1}$ yield
	\begin{equation*}
	\begin{split}
	\int_\Sigma& (L(x_\Sigma,y_\Sigma)+L(x_\Sigma,y_\Sigma)^{j}) e^{-\omega L(x_\Sigma,y_\Sigma)} \, d\sigma(y_\Sigma) \\
	&\leq C \max_{n \in \{1,\dots,p\}} \int_{\R^{\theta-1}}( L(x_\Sigma,x_{\Sigma_n}(y'))+ L(x_\Sigma,x_{\Sigma_n}(y'))^{j}) e^{-\omega L(x_\Sigma,x_{\Sigma_n}(y'))}\,dy' \\
	&\leq  C \max_{n \in \{1,\dots,p\}} \int_{\R^{\theta-1}}( \omega L(x_\Sigma,x_{\Sigma_n}(y'))+ \omega^j L(x_\Sigma,x_{\Sigma_n}(y'))^{j}) e^{-\omega L(x_\Sigma,x_{\Sigma_n}(y'))}\,dy',\\
	&
	\end{split}
	\end{equation*}
	where $C>0$ does not depend on $x_\Sigma$, $t$, $s$, and $\varepsilon$.
 	Next, let $n \in \{1,\dots,p\}$ and fix $x_n'\in \R^{\theta-1}$ such that  $\abs{ x_\Sigma - x_{\Sigma_n}(x_n') } = \min_{y'\in \R^{\theta-1}}\abs{x_\Sigma - x_{\Sigma_n}(y')}$. With this choice we obtain for all $y' \in \R^{\theta-1}$
	\begin{equation*}
	\begin{split}
	\frac{1}{2} \abs{x_n'- y'}  &\leq  \frac{1}{2} \abs{x_{\Sigma_n}(x_n')-x_{\Sigma_n}(y')} \\
	&\leq  \frac{1}{2}(\abs{x_\Sigma - x_{\Sigma_n}(y')} + \abs{x_\Sigma -x_{\Sigma_n}(x_n')} ) \\
	&\leq
	 \abs{x_\Sigma - x_{\Sigma_n}(y')}.
	 \end{split}
	\end{equation*}
	This implies for any $y' \in \R^{\theta-1}$
	\begin{equation*}
	\omega L(x_\Sigma,x_{\Sigma_n}(y')) = \omega (\abs{x_\Sigma-x_{\Sigma_n}(y')}  + \varepsilon\abs{t-s}) \geq \frac{\omega}{2}\abs{x_n'- y'} + \omega \varepsilon 
	\abs{t-s}.
	\end{equation*}
	Moreover,  $ a \mapsto (a + a^{j})e^{-a}$, $a >0$,  is a monotonically decreasing function and therefore we get with  $\rho(x_n',y') := \frac{\omega}{2}\abs{x_n'- y'}$ 
	\begin{equation*}
	\begin{split}
	\int_\Sigma (L(x_\Sigma,y_\Sigma)+&L(x_\Sigma,y_\Sigma)^{j}) e^{-\omega L(x_\Sigma,y_\Sigma)} \, d\sigma(y_\Sigma) \\
	&\leq C \max_{n \in \{1,\dots,p\}} \int_{\R^{\theta-1}}( (\rho(x_n',y') + \omega \varepsilon\abs{t-s}) \\
	&\qquad \qquad +  (\rho(x_n',y') + \omega \varepsilon \abs{t-s})^{j})  e^{- \rho(x_n',y') - \omega\varepsilon \abs{t-s}}\,dy' \\
	&\leq  C  \int_0^\infty( (\rho+ \omega \varepsilon\abs{t-s})+  (\rho + \omega \varepsilon \abs{t-s})^{j}) e^{- \rho- \omega\varepsilon \abs{t-s}} \rho^{\theta-2}d \rho\\
	%&\leq  C \int_0^\infty \big( (\rho+ \omega \varepsilon\abs{t-s})^{\theta-1} + (\rho+ \omega \varepsilon\abs{t-s})^{j+\theta-2} \big) e^{- \rho- \omega\varepsilon \abs{t-s}}  \, d \rho \\
	&=  C \biggl(  \int_{\omega \varepsilon \abs{t-s}}^\infty \rho^{\theta-1}e^{- \rho} \, d \rho +  \int_{\omega \varepsilon \abs{t-s}}^\infty \rho^{j+\theta-2} e^{- \rho} \, d \rho \biggr) \\
	&\leq C\begin{cases}
	1+\abs{\log(\varepsilon \abs{t-s})}, &  j = 1-\theta, \\
	\frac{1}{\varepsilon \abs{t-s}}, & j = -\theta,
	\end{cases}
	\end{split}
	\end{equation*}
	where $C>0$ does not depend on $x_\Sigma$, $t$, $s$, and $\varepsilon$. This proves \eqref{eq_diff_b_1}.
	
	\textit{Step~2.} 
	In this step we verify the estimate
	\begin{equation}\label{l2bound}
	 \Vert b_{t,s,\varepsilon}( z)\psi\Vert_{L^2(\Sigma;\C^N)}\leq C \varepsilon (1+\abs{\log(\varepsilon\abs{t-s})})\Vert \psi \Vert_{L^2(\Sigma;\C^N)},
	 \quad \psi \in L^2(\Sigma;\C^N).
	\end{equation}
 In fact, with the help of the Cauchy-Schwarz  inequality, Lemma \ref{lem_diff_G_z}~(i), and \eqref{eq_diff_b_1} we obtain
	for $\psi \in L^2(\Sigma;\C^N)$ and $x_\Sigma \in \Sigma$
	\begin{equation}\label{eq_diff_b_2}
		\begin{split}
		\abs{b_{t,s,\varepsilon}( z)\psi (x_\Sigma)}^2 &= \biggl|\int_{\Sigma} \Delta G_ z(x_\Sigma,y_\Sigma) \psi(y_\Sigma) \,d\sigma(y_\Sigma)\biggr|^2 \\
		&\leq \int_{\Sigma} \abs{\Delta G_ z(x_\Sigma,y_\Sigma)} \, d\sigma(y_\Sigma) \int_\Sigma \abs{\Delta G_ z(x_\Sigma,y_\Sigma)}  \abs{\psi(y_\Sigma)}^2 \, d\sigma(y_\Sigma) \\
		&\leq C \varepsilon^{2} \int_\Sigma (L(x_\Sigma,y_\Sigma)+L(x_\Sigma,y_\Sigma)^{{1-\theta}}) e^{-\omega L(x_\Sigma,y_\Sigma)} \, d\sigma(y_\Sigma) \\
		&\qquad  \cdot\int_\Sigma (L(x_\Sigma,y_\Sigma)+L(x_\Sigma,y_\Sigma)^{{1-\theta}}) e^{-\omega L(x_\Sigma,y_\Sigma)}\abs{\psi(y_\Sigma)}^2 \, d\sigma(y_\Sigma)\\
		&\leq C  \varepsilon^{2} (1+\abs{\log(\varepsilon \abs{t-s})})  \int_\Sigma (L(x_\Sigma,y_\Sigma)+L(x_\Sigma,y_\Sigma)^{{1-\theta}}) \\
		&\qquad \qquad \qquad\qquad\qquad\qquad \cdot e^{-\omega L(x_\Sigma,y_\Sigma)}\abs{\psi(y_\Sigma)}^2 \, d\sigma(y_\Sigma).
		\end{split}
	\end{equation}
	Now, Fubini's theorem and \eqref{eq_diff_b_1} show
	\begin{equation*}
		\begin{split}
		 \int_\Sigma&\abs{ b_{t,s,\varepsilon}( z)\psi (x_\Sigma)}^2 \, d\sigma(x_\Sigma)\\
		 &\leq C \varepsilon^{2} (1+\abs{\log(\varepsilon\abs{t-s})}) \\
		 &\quad \cdot\int_\Sigma  \int_\Sigma (L(x_\Sigma,y_\Sigma)+L(x_\Sigma,y_\Sigma)^{{1-\theta}}) e^{-\omega L(x_\Sigma,y_\Sigma)}\abs{\psi(y_\Sigma)}^2 \, d\sigma(y_\Sigma) \, d\sigma(x_\Sigma)\\
		&=C \varepsilon^{2} (1+\abs{\log(\varepsilon\abs{t-s})}) \\
		&\quad \cdot\int_\Sigma\int_\Sigma  (L(x_\Sigma,y_\Sigma)+L(x_\Sigma,y_\Sigma)^{{1-\theta}}) e^{-\omega L(x_\Sigma,y_\Sigma)}\, d\sigma(x_\Sigma) \abs{\psi(y_\Sigma)}^2 \, d\sigma(y_\Sigma) \\
		&\leq C \varepsilon^{2} (1+\abs{\log(\varepsilon\abs{t-s})})^2   \int_\Sigma \abs{\psi(y_\Sigma)}^2 \, d\sigma(y_\Sigma),
		\end{split}
	\end{equation*}
	which yields \eqref{l2bound}.
	
	\textit{Step~3.} Next we prove the estimate
	\begin{equation}\label{h1bound}
	 \Vert b_{t,s,\varepsilon}( z)\psi\Vert_{H^1(\Sigma;\C^N)}\leq 
	 C \frac{1}{\abs{t-s}}\Vert \psi \Vert_{L^2(\Sigma;\C^N)},
	 \quad \psi \in L^2(\Sigma;\C^N).
	\end{equation}
	Let $\psi \in L^2(\Sigma;\C^N)$ and $ x_{\Sigma_l}(x') = x_\Sigma \in \Sigma$ with $x' \in \R^{\theta -1}$. By Lemma~\ref{lem_part_of_unity} the function $\varphi_l$ and its derivatives  are bounded. Thus, with $\widetilde{\varphi}_l := \varphi_l \circ x_{\Sigma_l}$ we have
	\begin{equation*}%\label{eq_diff_b_3}
	\begin{split}
			\biggl|\frac{d}{dx_k'}& (\widetilde{\varphi}_l(x') (b_{t,s,\varepsilon}( z)\psi )(x_{\Sigma_l}(x')))\biggr|^2 \\
			%\leq 2\abs{ \widetilde{\varphi}_l(x')\frac{d}{dx_k'}(b_{t,s,\varepsilon}( z)\psi )(x_{\Sigma_l}(x')) }^2
			%+ 2\abs{ \left(\frac{d}{dx_k'} \widetilde{\varphi}_l(x') \right)b_{t,s,\varepsilon}( z)\psi(x_{\Sigma_l}(x')) }^2 \\
			 &\leq  C \Bigl(\bigl| \frac{d}{dx_k'} (b_{t,s,\varepsilon}( z)\psi )(x_{\Sigma_l}(x')) \bigr|^2 + \abs{  b_{t,s,\varepsilon}( z)\psi (x_{\Sigma_l}(x'))}^2\Bigr).
			 \end{split}
	\end{equation*}
	 Using the dominated convergence theorem and the properties of $\Delta G_z$ stated in Lemma~\ref{lem_diff_G_z} one obtains
	\begin{equation*}
		\frac{d}{dx_k'}(b_{t,s,\varepsilon}( z)\psi )(x_{\Sigma_l}(x'))= \int_{\Sigma} \frac{d}{dx_k'} \Delta G_ z(x_{\Sigma_l}(x'),y_\Sigma) \psi (y_\Sigma)\, d\sigma(y_\Sigma).
	\end{equation*}
	 Hence, we get with the Cauchy Schwarz inequality, Lemma \ref{lem_diff_G_z}~(ii), $x_\Sigma =  x_{\Sigma_l}(x')$, and \eqref{eq_diff_b_1} 
	\begin{equation*}
		\begin{split}
		&\biggl| \frac{d}{dx_k'} (b_{t,s,\varepsilon}( z)\psi )(x_{\Sigma_l}(x')) \biggr|^2  = \biggl|\int_{\Sigma} \frac{d}{dx_k'} \Delta G_ z(x_{\Sigma_l}(x'),y_\Sigma) \psi (y_\Sigma)\, d\sigma(y_\Sigma) \biggr|^2 \\ 
		&\quad\leq \int_{\Sigma} \bigl|\frac{d}{dx_k'}\Delta G_ z(x_{\Sigma_l}(x'),y_\Sigma)\bigr| \, d\sigma(y_\Sigma) \int_\Sigma \bigl|\frac{d}{dx_k'}\Delta G_ z(x_{\Sigma_l}(x'),y_\Sigma)\bigr|  \abs{\psi(y_\Sigma)}^2  d\sigma(y_\Sigma) \\
		&\quad\leq C \varepsilon^{2} \int_\Sigma (L(x_\Sigma,y_\Sigma)+L(x_\Sigma,y_\Sigma)^{-\theta}) e^{-\omega L(x_\Sigma,y_\Sigma)} \, d\sigma(y_\Sigma) \\
		&\qquad \qquad \cdot \int_\Sigma (L(x_\Sigma,y_\Sigma)+L(x_\Sigma,y_\Sigma)^{{-\theta}}) e^{-\omega L(x_\Sigma,y_\Sigma)}\abs{\psi(y_\Sigma)}^2 \, d\sigma(y_\Sigma)\\
		&\quad\leq C \frac{\varepsilon}{\abs{t-s}} \int_\Sigma (L(x_\Sigma,y_\Sigma)+L(x_\Sigma,y_\Sigma)^{{-\theta}}) e^{-\omega L(x_\Sigma,y_\Sigma)}\abs{\psi(y_\Sigma)}^2 \, d\sigma(y_\Sigma).
		\end{split}
	\end{equation*}
 	According to \eqref{eq_diff_b_2} we can estimate
 	\begin{equation*}
 	\begin{split}
  			\abs{ b_{t,s,\varepsilon}( z)\psi (x_{\Sigma_l}(x')) }^2  &\leq C  \varepsilon^{2} (1+\abs{\log(\varepsilon \abs{t-s})})  \\
  			&\cdot\int_\Sigma (L(x_\Sigma,y_\Sigma)+L(x_\Sigma,y_\Sigma)^{{1-\theta}}) e^{-\omega L(x_\Sigma,y_\Sigma)}\abs{\psi(y_\Sigma)}^2\, d\sigma(y_\Sigma), 
  			\end{split}
 	\end{equation*}
 	where $x_\Sigma = x_{\Sigma_l}(x')$.
  	Moreover,  $a+a^{1-\theta} \leq 2(a + a^{-\theta})$ for $a>0$ and $1+\abs{\log(b)} \leq C\tfrac{1}{b}$ for $b \in (0,2\varepsilon_2)$ yields
  		\begin{equation*}
  		\begin{split}
  			| b_{t,s,\varepsilon}( z)&\psi (x_{\Sigma_l}(x')) |^2  \\
  			 &\leq C \frac{\varepsilon}{\abs{t-s}} \int_\Sigma (L(x_\Sigma,y_\Sigma)+L(x_\Sigma,y_\Sigma)^{{-\theta}}) e^{-\omega L(x_\Sigma,y_\Sigma)}\abs{\psi(y_\Sigma)}^2 \, d\sigma(y_\Sigma).
  			 \end{split}
  	\end{equation*}
	Thus, 
	\begin{equation*}
		\begin{split}
		& \int_{x_{\Sigma_l}^{-1}(\Sigma)} \bigl|\frac{d}{dx_k'} \Big( \varphi_l(x_{\Sigma_l}(x')) b_{t,s,\varepsilon}( z)\psi\Big)(x')\bigr|^2 \, dx' \\
		&\quad\leq C \frac{\varepsilon}{\abs{t-s}}  \int_{x_{\Sigma_l}^{-1}(\Sigma)} \int_\Sigma  (L(x_{\Sigma_l}(x'),y_\Sigma)+L(x_{\Sigma_l}(x'),y_\Sigma)^{{-\theta}})\\
		&\qquad \qquad\qquad\qquad\qquad\qquad \cdot e^{-\omega L(x_{\Sigma_l}(x'),y_\Sigma)}\abs{\psi(y_\Sigma)}^2 \, d\sigma(y_\Sigma)  dx'\\
%		&\quad\leq C  \frac{\varepsilon}{\abs{t-s}}  \int_{x_{\Sigma_l}^{-1}(\Sigma)} \int_\Sigma  (L(x_{\Sigma_l}(x'),y_\Sigma)+L(x_{\Sigma_l}(x'),y_\Sigma)^{{-\theta}})  e^{-\omega L(x_{\Sigma_l}(x'),y_\Sigma)}\abs{\psi(y_\Sigma)}^2 \, d\sigma(y_\Sigma) \sqrt{1 +\abs{\nabla \zeta_l(x')}^2}\, dx'\\
		&\quad\leq C \frac{\varepsilon}{\abs{t-s}} \int_\Sigma   \int_{\Sigma}  (L(x_\Sigma,y_\Sigma)+L(x_\Sigma,y_\Sigma)^{{-\theta}}) e^{-\omega L(x_\Sigma,y_\Sigma)}\, \abs{\psi(y_\Sigma)}^2 \, d\sigma(y_\Sigma) d \sigma(x_\Sigma).
		\end{split}
	\end{equation*}
	Therefore, Fubini's theorem and~\eqref{eq_diff_b_1}  yield
	\begin{equation*}
		\int_{x_{\Sigma_l}^{-1}(\Sigma)} \bigl|\frac{d}{dx_k'} \Big( \varphi_l(x_{\Sigma_l}(x')) b_{t,s,\varepsilon}( z)\psi\Big)(x')\bigr|^2 \, dx' \leq C \frac{1}{\abs{t-s}^2}   \norm{\psi}_{L^2(\Sigma;\C^N)}^2.
	\end{equation*}
	 This estimate, the definition of the norm in $H^1(\Sigma;\C^N)$, see
	 Section~\ref{sec_geometry}, and \eqref{l2bound} imply \eqref{h1bound}.
	
	\textit{Step~4.}
	 By combining $H^{1/2}(\Sigma;\C^N) = [L^2(\Sigma;\C^N),H^{1}(\Sigma;\C^N) ]_{1/2}$, see \eqref{eq_interpolation}, and   $L^2(\Sigma;\C^N) = [L^2(\Sigma;\C^N),L^2(\Sigma;\C^N)]_{1/2}$
	 we conclude from the bounds
	 \eqref{l2bound} and \eqref{h1bound} together with
	 \eqref{eq_interpolation_inequality} that
	\begin{equation*}
	\begin{split}
	\|b_{t,s,\varepsilon}( z) &\|_{L^2(\Sigma;\C^N) \to H^{1/2}(\Sigma;\C^N)} \\
	&\leq C \norm{b_{t,s,\varepsilon}( z)}_{L^2(\Sigma;\C^N) \to L^2(\Sigma;\C^N)}^{1/2}  \norm{b_{t,s,\varepsilon}( z)}_{L^2(\Sigma;\C^N) \to H^1(\Sigma;\C^N)}^{1/2}\\
	&\leq  C (\varepsilon +\varepsilon\abs{\log(\varepsilon \abs{t-s })})^{1/2} \frac{1}{\abs{t-s}^{1/2}}.
	\end{split}
	\end{equation*}
	This completes the proof of Proposition~\ref{prop_diff_b}.
\end{proof}

%\begin{proposition}\label{prop_diff_B}
%	Let $r \in [0,1)$ and $\overline{B}_\varepsilon(z)$ as well as $\widetilde{B}_\varepsilon(z)$ be defined by \eqref{def_B_bar} and \eqref{def_B_tilde}, respectively. Then, $\widetilde{B}_\varepsilon( z) - \overline{B}_\varepsilon(z)$ can be extended to an bounded operator from $L^2\big((-1,1);L^2(\Sigma;\C^N)\big)$ to $L^2\big((-1,1);H^r(\Sigma;\C^N)\big)$ and there holds
%	\begin{equation*}
%		\norm{\widetilde{B}_\varepsilon( z) - \overline{B}_\varepsilon(z)}_{0 \to r} \leq C (\varepsilon \abs{\log(\varepsilon)})^{1-r}.
%	\end{equation*}
%\end{proposition}

\begin{lemma}\label{lem_Bochner_measuarble}
	Let  $\varepsilon \in (0,\varepsilon_A)$. Then, the operator-valued function   
	\begin{equation*}
			F: (-1,1)^2 \to \mathcal{L}(L^2(\Sigma;\C^N),H^{1/2}(\Sigma;\C^N)), \qquad
			F(t,s) =\begin{cases}
				b_{t,s,\varepsilon}(z),		& \textup{if } t \neq s,\\
				0, 	&\textup{if } t=s,		
			\end{cases} 
	\end{equation*}
	is measurable.
\end{lemma}
\begin{proof}
	It suffices to prove that $\spv{F(\cdot,\cdot)\varphi}{\psi}_{H^{1/2}(\Sigma;\C^N)}$ is measurable on $(-1,1)^2$ for all $\varphi \in L^2(\Sigma;\C^N)$ and $\psi \in H^{1/2}(\Sigma;\C^N)$; cf. Definition~\ref{def_Bochner_measurable}. For this, we prove that the function $\spv{F(\cdot,\cdot)\varphi}{\psi}_{H^{1/2}(\Sigma;\C^N)}$ is continuous on $\mathcal{O} := (-1,1)^2 \setminus \{(t,t):  t \in (-1,1) \}$. Let $(t,s) \in \mathcal{O}$ be fixed and consider the case $t > s $.
	Choose a  sequence $((t_n,s_n))_{n \in \N} \subset \mathcal{O}$ which converges to $(t,s)$. It is no restriction to  assume that $\tfrac{3}{2} (t_n -s_n) > t-s > \frac{1}{2} (t_n -s_n)$ holds for all $n \in \N$. Then, 
	\begin{equation*}
	  \begin{split}
  	  L(x_\Sigma,y_\Sigma,t_n,s_n) &= |x_\Sigma - y_\Sigma| + \varepsilon |t_n - s_n| < |x_\Sigma - y_\Sigma| + 2 \varepsilon |t - s| 
  	  \leq 2 L(x_\Sigma,y_\Sigma,t,s)
  	  \end{split}
	\end{equation*}
	and in a similar way
	\begin{equation*}
	  L(x_\Sigma,y_\Sigma,t_n,s_n)^{1 - \theta} \leq \left( \frac{2}{3} \right)^{1-\theta} L(x_\Sigma,y_\Sigma,t,s)^{1 - \theta}.
	\end{equation*}
	Moreover, as $|t_n - s_n| \geq 0 = |t - s| - |t - s| \geq |t - s| - 2$, one has 
	\begin{equation*}
	  e^{-\omega L(x_\Sigma,y_\Sigma,t_n,s_n)} \leq e^{-\omega (|x_\Sigma - y_\Sigma| + \varepsilon |t - s|) + 2 \omega \varepsilon} \leq e^{2 \omega \varepsilon_A} e^{-\omega L(x_\Sigma,y_\Sigma,t,s)}.
	\end{equation*}
	Combining Lemma~\ref{lem_diff_G_z}~(i) with the latter three displayed formulas yields the existence of a constant $C >0$ which is independent of $x_\Sigma, y_\Sigma, t, s, t_n, s_n$, and $\varepsilon$ such that 
	\begin{equation} \label{equation_Delta_s_n_t_n}
	  \begin{split}
  	  | \Delta G_z(x_\Sigma,y_\Sigma,t_n,s_n) | &\leq C\varepsilon \bigl(L(x_\Sigma,y_\Sigma,t_n,s_n)+L(x_\Sigma,y_\Sigma,t_n,s_n)^{{1-\theta}}\bigr) e^{-\omega L(x_\Sigma,y_\Sigma,t_n,s_n)} \\
  	  &\leq C\varepsilon \bigl(L(x_\Sigma,y_\Sigma,t,s)+L(x_\Sigma,y_\Sigma,t,s)^{{1-\theta}}\bigr) e^{-\omega L(x_\Sigma,y_\Sigma,t,s)}\\
  	  &=: 	M(x_\Sigma,y_\Sigma).
  	  \end{split}
	\end{equation}
	We claim that $(b_{t_n,s_n,\varepsilon}(z)\varphi)_{n \in \N}$ converges weakly to $b_{t,s,\varepsilon}(z)\varphi$ in $L^2(\Sigma;\C^N)$. Let $\gamma \in L^2(\Sigma;\C^N)$ be fixed, then
	\begin{equation*}
	\begin{split}
		&\spv{(b_{t_n,s_n,\varepsilon}(z)- b_{t,s,\varepsilon}(z))\varphi}{\gamma}_{L^2(\Sigma;\C^N)} \\
		&\quad= \int_\Sigma \int_\Sigma \spf{(\Delta G_z(x_\Sigma,y_\Sigma,t_n,s_n)-\Delta G_z(x_\Sigma,y_\Sigma,t,s)) \varphi(y_\Sigma )}{\gamma(x_\Sigma)} \,d\sigma(y_\Sigma) \,d \sigma(x_\Sigma).
		\end{split}
	\end{equation*}
	The integrand on the right hand side converges pointwise almost everywhere to zero, as $n \to \infty$. Moreover,~\eqref{equation_Delta_s_n_t_n} shows that the integrand is bounded by $M(x_\Sigma,y_\Sigma) \abs{\varphi(y_\Sigma)} \abs{\gamma(x_\Sigma)}$.
	Applying the Cauchy-Schwarz inequality twice, Fubini's theorem, and the symmetry relation $M(x_\Sigma,y_\Sigma) = M(y_\Sigma, x_\Sigma)$ yields
	\begin{equation*}
		\begin{split}
			\bigg(\int_\Sigma &\int _\Sigma  M(x_\Sigma,y_\Sigma) \abs{\varphi(y_\Sigma)} \abs{\gamma(x_\Sigma)} \, d\sigma(y_\Sigma) \, d\sigma(x_\Sigma)\bigg)^2 \\
			&\leq\int_\Sigma \bigg(\int _\Sigma  M(x_\Sigma,y_\Sigma) \abs{\varphi(y_\Sigma)} \, d\sigma(y_\Sigma) \bigg)^2 \, d \sigma(x_\Sigma)\norm{\gamma}_{L^2(\Sigma;\C^N)}^2\\
			&\leq\int_\Sigma \int _\Sigma  M(x_\Sigma,y_\Sigma) \abs{\varphi(y_\Sigma)}^2 \, d\sigma(y_\Sigma) \int _\Sigma  M(x_\Sigma,y_\Sigma)  \, d \sigma(y_\Sigma) \, d \sigma(x_\Sigma)\norm{\gamma}_{L^2(\Sigma;\C^N)}^2 \\
			&\leq C \bigg(\sup_{x_\Sigma \in \Sigma} \int_\Sigma M(x_\Sigma,y_\Sigma) \, d\sigma(y_\Sigma) \bigg)^2 \norm{\varphi}_{L^2(\Sigma;\C^N)}^2  \norm{\gamma}_{L^2(\Sigma;\C^N)}^2 .
		\end{split}
	\end{equation*}
	Furthermore, with \eqref{eq_diff_b_1} we see that $\sup_{x_\Sigma \in \Sigma} \int_\Sigma M(x_\Sigma,y_\Sigma) \, d\sigma(y_\Sigma) \leq C \varepsilon (1+\abs{\log (\varepsilon \abs{t-s})}) < \infty$.
	Hence, dominated convergence yields $$\spv{(b_{t_n,s_n,\varepsilon}(z)- b_{t,s,\varepsilon}(z))\varphi}{\gamma}_{L^2(\Sigma;\C^N)} \to 0$$ for $n \to \infty$.  Since, $\gamma \in L^2(\Sigma;\C^N)$ was arbitrary, we conclude that $(b_{t_n,s_n,\varepsilon}(z) \varphi)_{n \in \N}$ converges weakly to $b_{t,s,\varepsilon}(z) \varphi$ in $L^2(\Sigma;\C^N)$. 
	
	Next, we show that $(b_{t_n,s_n,\varepsilon}(z)\varphi)_{n \in \N}$ converges weakly to $b_{t,s,\varepsilon}(z)\varphi$ in the space $H^{1/2}(\Sigma; \mathbb{C}^N)$, which proves the claimed continuity.	
	For this, we note that Lemma~\ref{prop_diff_b} and $\tfrac{3}{2} (t_n -s_n) > t-s >0$ imply that $(b_{t_n,s_n,\varepsilon}(z)\varphi)_{n \in \N}$ is a bounded sequence in $H^{1/2}(\Sigma;\C^N)$. Let us assume that $(b_{t_n,s_n,\varepsilon}(z) \varphi)_{n \in \N}$ does not converge weakly to $b_{t,s,\varepsilon}(z) \varphi$ in $H^{1/2}(\Sigma;\C^N)$. Then, the $H^{1/2}$-boundedness implies that there exists a weakly convergent subsequence $(b_{t_{n_k},s_{n_k},\varepsilon}(z) \varphi )_{k \in \N}$ which converges to some $\varphi' \in H^{1/2}(\Sigma;\C^N)$ with $\varphi' \neq b_{t,s,\varepsilon}(z) \varphi$. However, in this case $(b_{t_{n_k},s_{n_k},\varepsilon} (z)\varphi)_{k \in \N}$ would also converge weakly to $\varphi'$ in $L^2(\Sigma;\C^N)$ which contradicts the first part of the proof. Hence, $(b_{t_n,s_n,\varepsilon}(z) \varphi)_{n \in \N}$ converges weakly to $b_{t,s,\varepsilon}(z) \varphi$ in $H^{1/2}(\Sigma; \mathbb{C}^N)$ and therefore, $(\spv{b_{t_n,s_n,\varepsilon}(z) \varphi}{\psi}_{H^{1/2}(\Sigma;\C^N)})_{n \in \N}$ converges to $\spv{b_{t,s,\varepsilon}(z)\varphi}{\psi}_{H^{1/2}(\Sigma;\C^N)}$. 
\end{proof}

After all these preliminary considerations we are prepared to prove \eqref{ESTIMATE_B_BAR_B_TILDE}.
\begin{proof}[Proof of \eqref{ESTIMATE_B_BAR_B_TILDE}]
	Let $f \in L^2((-1,1);L^2(\Sigma;\C^N))$ be fixed. Using Proposition \ref{prop_diff_b}, the Cauchy-Schwarz inequality, and Fubini's theorem we obtain   
	\begin{equation*}
		\begin{split}
			\int_{-1}^1& \biggl(\int_{-1}^1 \norm{b_{t,s,\varepsilon}( z) f(s)}_{H^{1/2}
				(\Sigma;\C^N)} \,ds\biggr)^2 \, dt \\
				&\leq    C  \int_{-1}^1\biggl(\int_{-1}^1(\varepsilon + \varepsilon\abs{ \log(\varepsilon\abs{t-s})})^{1/2} \frac{1}{\abs{t-s}^{1/2}}\norm{f(s)}_{L^2(\Sigma;\C^N)} \,ds \biggr)^2  \, dt\\
			&\leq    C  \int_{-1}^1\biggl(\int_{-1}^1(\varepsilon + \varepsilon\abs{ \log(\varepsilon\abs{t-s})})^{1/2} \frac{1}{\abs{t-s}^{1/2}}  \,ds  \\
			&\qquad \cdot\int_{-1}^1 (\varepsilon + \varepsilon\abs{ \log(\varepsilon\abs{t-s})})^{1/2} \frac{1}{\abs{t-s}^{1/2}} \norm{f(s)}_{L^2(\Sigma;\C^N)}^2 \,ds   \biggr) dt\\
			&\leq    C  \sup_{s \in (-1, 1)} \biggl(\int_{-1}^1(\varepsilon + \varepsilon\abs{ \log(\varepsilon\abs{t-s})})^{1/2} \frac{1}{\abs{t-s}^{1/2}} \,dt \biggr)^2 \int_{-1}^{1}\norm{f(s)}_{L^2(\Sigma;\C^N)}^2 \,ds  \\
			&\leq    C \varepsilon (1+\abs{\log(\varepsilon)}) \norm{f}_0^2.
		\end{split}
	\end{equation*}
%	where we used in the last estimate that for $\varepsilon>0$ sufficiently small one has
%	\begin{equation*}
%		\begin{split}
%			\sup_{s \in (-1, 1)} &\int_{-1}^1(\varepsilon + \varepsilon \abs{\log(\varepsilon\abs{t-s})})^{1/2} \frac{1}{\abs{t-s}^{1/2}} \,dt= \sup_{t \in (-1, 1)} \int_{-1}^1(\varepsilon + \varepsilon \abs{\log(\varepsilon\abs{t-s})})^{1/2}  \frac{1}{\abs{t-s}^{1/2}} \,ds  \\
%			&\leq \int_{-2}^2 (\varepsilon + \varepsilon \abs{\log(\varepsilon\abs{\tau})})^{1/2} \frac{1}{\abs{\tau}^{1/2}} \,d\tau \leq C (\varepsilon + \varepsilon \abs{\log(\varepsilon)})^{1/2}   \int_{-2}^2(1 + \abs{ \log (\abs{\tau  })})^{1/2} \frac{1}{\abs{\tau}^{1/2}} \,d\tau\\
%			&\leq C (\varepsilon + \varepsilon \abs{\log(\varepsilon)})^{1/2} .
%		\end{split}
%	\end{equation*}
	Combined with Lemma \ref{lem_Bochner_measuarble} and the discussion below Definition \ref{def_Bochner_measurable} this shows that 
	$\int_{-1}^1 b_{t,s,\varepsilon}(z) f(s)\,ds \in H^{1/2}(\Sigma;\C^N)$ exists for a.e. $t \in (-1,1)$ and  that the function $t\mapsto \int_{-1}^1 b_{t,s,\varepsilon}(z)f(s) \,ds \in H^{1/2}(\Sigma;\C^N) $ is measurable. Hence, the mapping 
    \begin{equation*}
    \begin{split}
        \mathcal{B}&: L^2((-1,1);L^2(\Sigma;\C^N)) \to L^2((-1,1);H^{1/2}(\Sigma;\C^N)),\\
        \mathcal{B} f(t) &:= \int_{-1}^1 b_{t,s,\varepsilon}(z)f(s) \, ds,
        \end{split}
    \end{equation*}
    is well-defined, bounded, and $\| \mathcal{B} \|_{0 \to 1/2} \leq C (\varepsilon + \varepsilon \abs{\log(\varepsilon)})^{1/2}$. By  \eqref{eq_diff_B_tilde_bar}, \eqref{eq_def_b_t,s,eps}, and Proposition~\ref{prop_Bochner}~(iii) we also  have  
	$$(\widetilde{B}_\varepsilon( z) - \overline{B}_\varepsilon(z))f(t) = \int_{-1}^1 b_{t,s,\varepsilon}(z)f(s) \, ds = \mathcal{B} f(t) $$
	for all $f \in L^2((-1,1);H^{1/2}(\Sigma;\C^N))$. Hence, $\widetilde{B}_\varepsilon( z) - \overline{B}_\varepsilon(z)$ can be extended to a bounded operator from $L^2((-1,1);L^2(\Sigma;\C^N))$ to $L^2((-1,1);H^{1/2}(\Sigma;\C^N))$ and~\eqref{ESTIMATE_B_BAR_B_TILDE} is true, i.e. all claims are shown.
\end{proof}

\section{Properties of $\Phi_{\lowercase{z}}$ and $\mathcal{C}_{\lowercase{z}}$} \label{appendix_Phi_C}

This appendix is devoted to the proofs of Proposition~\ref{prop_Phi_z} and Proposition~\ref{proposition_C_z}, which are inspired by the abstract notion of boundary triples, their $\gamma$-fields and Weyl functions from extension theory of symmetric operators in Hilbert spaces;
cf. \cite{BHS20,BL07,BL12,BGP08,DHMS06,DM91,DM95}. Here
we follow a similar strategy as in \cite{BHSS22}, where similar results for bounded $\Sigma$ were shown.  We also refer to \cite{BH20, BHM20, BHOP20, Ben21, OV16} for related considerations  in the context of two and three-dimensional Dirac operators  and to \cite{BMP17, CMP13} for one-dimensional Dirac operators,

Throughout this appendix, let $m \in \R$ and let $\Sigma \subset \mathbb{R}^\theta$ satisfy Hypothesis~\ref{hypothesis_Sigma}. We define in $L^2(\mathbb{R}^\theta; \mathbb{C}^N)$ the operator $T$ by
	\begin{equation*} %\label{def_T}
	\begin{split}
	Tu &:= (- i (\alpha \cdot \nabla) + m \beta ) u_+ \oplus ( -i(\alpha \cdot \nabla) + m \beta ) u_-, \\
	\dom T&:= H^1(\R^\theta \setminus \Sigma;\C^N)= H^1(\Omega_+;\C^N)  \oplus H^1(\Omega_-;\C^N),
	\end{split}
	\end{equation*}
and the mappings $\Gamma_0, \Gamma_1: \dom T \to H^{1/2}(\Sigma;\C^N)$ by
\begin{equation*} %\label{def_Gamma}
	\begin{aligned}
  	\Gamma_0 u &:= i (\alpha \cdot \nu)( \tr^+ u_+ - \tr^- u_- ) \quad \text{and} \quad & \Gamma_1 u &:= \frac{1}{2}( \tr^+ u_+  + \tr^- u_-  ).
	\end{aligned}
\end{equation*}
Recall that $\nu$ is pointing outwards of $\Omega_+$. For $u_\pm, v_\pm \in H^1(\Omega_\pm; \mathbb{C}^N)$ 
integration by parts implies 
\begin{equation*}
  ((\alpha \cdot \nabla) u_\pm, v_\pm)_{L^2(\Omega_\pm; \mathbb{C}^N)} = -(u_\pm, (\alpha \cdot \nabla) v_\pm)_{L^2(\Omega_\pm; \mathbb{C}^N)} \pm ((\alpha \cdot \nu) u_\pm, v_\pm)_{L^2(\Sigma; \mathbb{C}^N)}
\end{equation*}
and one finds in the same way as, e.g., in the proof of \cite[Theorem 4.3 (i)]{BHSS22} that
\begin{equation}\label{eq_Greens_formula}
	\spv{Tu}{v}_{L^2(\R^\theta;\C^N)} - \spv{u}{Tv}_{L^2(\R^\theta;\C^N)}  = \spv{\Gamma_1u}{ \Gamma_0v}_{L^2(\Sigma;\C^N)}  - \spv{\Gamma_0u}{\Gamma_1v}_{L^2(\Sigma;\C^N)} 
\end{equation}
holds for all $u,v \in \dom T$. 
Note that~\eqref{eq_Greens_formula} implies that $T \upharpoonright \ker \Gamma_0$ is symmetric. Furthermore, $H^1(\mathbb{R}^\theta; \mathbb{C}^N) \subset \ker \Gamma_0$. Since the free Dirac operator $H$ in~\eqref{def_free_op} is self-adjoint, this implies that
\begin{equation} \label{equation_ker_Gamma_0}
  T \upharpoonright \ker \Gamma_0  = H.
\end{equation}
Now, we are prepared to prove Proposition~\ref{prop_Phi_z} and Proposition~\ref{proposition_C_z}.

\begin{proof}[Proof of Proposition~\ref{prop_Phi_z}]
	First, Fubini's theorem shows the representation in~\eqref{equation_Phi_star}. Hence, the mapping properties of $\tr$ and $R_z$ prove assertion~(iii).  
	
	To verify item~(i), we note first that by~(iii) and anti-duality $\Phi_z$ has the bounded extension
	\begin{equation} \label{def_Phi_z_tilde}
	  \widetilde{\Phi}_z := (\Phi_z^*)': H^{-1/2}(\Sigma; \mathbb{C}^N) \rightarrow L^2(\mathbb{R}^\theta; \mathbb{C}^N) = H^0(\Omega_+; \mathbb{C}^N) \oplus H^0(\Omega_-; \mathbb{C}^N).
	\end{equation}
	Next, we show the statement for $r = \frac{1}{2}$. If we manage to do that, then the claim for $r \in [0, \frac{1}{2})$ follows from~\eqref{def_Phi_z_tilde} and interpolation. 
	
	To prove the claim for $r = \frac{1}{2}$ we note that with~\eqref{equation_ker_Gamma_0} one can show for $z \in \rho(H)$ the direct sum decomposition
	\begin{equation*}
		\dom T  = \dom H  \dot{+}  \ker (T -  z ) =   \ker \Gamma_0 \dot{+}  \ker (T -  z ),
	\end{equation*}
	which allows us to define the auxiliary operator
	\begin{equation}\label{eq_gamma}
		\widehat{\Phi}_z := (\Gamma_0 \upharpoonright \ker (T-  z))^{-1}.
	\end{equation}
	We remark that this is the usual formula for the $\gamma$-field corresponding to a (quasi or generalized) boundary triple.
	Note that the properties of the trace operator in Proposition~\ref{proposition_trace_theorem} imply that $\ran \Gamma_0 =H^{1/2}(\Sigma;\C^N)$ and we also have $\dom T = H^{1}(\R^\theta \setminus \Sigma;\C^N)$. Thus, $\widehat{\Phi}_z$ is a linear operator from $H^{1/2}(\Sigma;\C^N)$ to $H^1(\R^\theta \setminus \Sigma;\C^N)$. Next, we show that $\widehat{\Phi}_z$ is a restriction of $\Phi_z$. To see this, we observe for $v \in  L^2(\R^\theta;\C^N)$, $\varphi \in H^{1/2}(\Sigma;\C^N)$, and  $u = R_{\overline{z}}v = (H-\overline{z})^{-1} v \in \dom H = \ker \Gamma_0$ with the help of \eqref{eq_Greens_formula} that
	\begin{equation*}%\label{eq_gamma_adj}
		\begin{split}
			(\widehat{\Phi}_z \varphi, v)_{L^2(\R^\theta;\C^N)} &= (\widehat{\Phi}_z \varphi, (H-\overline{ z})u )_{L^2(\R^\theta;\C^N)} \\
			&= (\widehat{\Phi}_z \varphi, H u)_{L^2(\R^\theta;\C^N)} -  ( z \widehat{\Phi}_z \varphi, u)_{L^2(\R^\theta;\C^N)}\\
			&= (\widehat{\Phi}_z \varphi, T u)_{L^2(\R^\theta;\C^N)} -  (T \widehat{\Phi}_z \varphi, u)_{L^2(\R^\theta;\C^N)}\\
			&=- (\Gamma_1\widehat{\Phi}_z \varphi, \Gamma_0 u)_{L^2(\Sigma;\C^N)}  + (\Gamma_0\widehat{\Phi}_z \varphi, \Gamma_1 u)_{L^2(\Sigma;\C^N)} \\
			&= \spv{\varphi}{\Gamma_1 R_{\overline{ z}}v}_{L^2(\Sigma;\C^N)} \\
			&=  \spv{\varphi}{\Phi_z^*v}_{L^2(\Sigma;\C^N)}=\spv{\Phi_z \varphi}{v}_{L^2(\R^{\theta};\C^N)} . 
		\end{split}
	\end{equation*}
	Since this holds for all $v \in L^2(\mathbb{R}^\theta; \mathbb{C}^N)$ we conclude  $\widehat{\Phi}_z \varphi = \Phi_z \varphi$, i.e. $\widehat{\Phi}_z = \Phi_z \upharpoonright H^{1/2}(\Sigma; \mathbb{C}^N)$. In particular, $\Phi_z \varphi \in \ker (T-z)$ by~\eqref{eq_gamma}, which yields item~(ii).
	Eventually, we show that this restriction of $\Phi_z$ is bounded from $H^{1/2}(\Sigma; \mathbb{C}^N)$ to $H^1(\R^\theta \setminus \Sigma;\C^N)$. To see this, we show that $\widehat{\Phi}_z$ is closed with respect to these spaces.
	But this follows from the $L^2$ boundedness of $\Phi_z$ and the fact that $H^{1/2}(\Sigma;\C^N) $  and $H^1(\R^\theta \setminus \Sigma;\C^N)$ are continuously embedded in $L^2(\Sigma;\C^N)$ and $L^2(\R^\theta;\C^N)$, respectively.  Thus, the closed graph theorem shows that
	\begin{equation*}
	  \widehat{\Phi}_z = \Phi_z \upharpoonright H^{1/2}(\Sigma; \mathbb{C}^N): H^{1/2}(\Sigma; \mathbb{C}^N) \rightarrow H^1(\Omega_+; \mathbb{C}^N) \oplus H^1(\Omega_-; \mathbb{C}^N)
	\end{equation*}
	is bounded, which finishes the proof.
\end{proof}

\begin{proof}[Proof of Proposition~\ref{proposition_C_z}]
	(i) First, it follows from Proposition~\ref{prop_Phi_z}~(i) and~\eqref{def_C_z} that $\mathcal{C}_z$ is a bounded operator in $H^{1/2}(\Sigma;\C^N)$. Next, we show that the anti-dual $\mathcal{C}_{\overline{ z}}'$ of $\mathcal{C}_{\overline{ z}}$, which is a bounded map in $H^{-1/2}(\Sigma; \mathbb{C}^N)$, is an extension of $\mathcal{C}_z$. To see this, let $\varphi, \psi \in H^{1/2}(\Sigma;\C^N)$. We use~\eqref{eq_Greens_formula}, Proposition~\ref{prop_Phi_z}~(ii),~\eqref{eq_gamma}, and the definition of $\mathcal{C}_z$ to obtain
	\begin{equation*}
	\begin{split}
	0&=(T\widehat{\Phi}_z \varphi, \widehat{\Phi}_{\overline{ z}} \psi)_{L^2(\R^\theta;\C^N)} - (\widehat{\Phi}_z \varphi, T \widehat{\Phi}_{\overline{ z}} \psi )_{L^2(\R^\theta;\C^N)} \\
	&= \spv{\mathcal{C}_z \varphi}{\psi}_{L^2(\Sigma;\C^N)}  - \spv{\varphi}{\mathcal{C}_{\overline{ z}} \psi}_{L^2(\Sigma;\C^N)} \\
	&= \spf{\mathcal{C}_z \varphi}{\psi}_{H^{-1/2}(\Sigma;\C^N) \times H^{1/2}(\Sigma;\C^N) }  - \spf{\varphi}{\mathcal{C}_{\overline{ z}} \psi}_{H^{-1/2}(\Sigma;\C^N) \times H^{1/2}(\Sigma;\C^N) } \\
	&= \spf{(\mathcal{C}_z - \mathcal{C}_{\overline{ z}}')\varphi}{\psi}_{H^{-1/2}(\Sigma;\C^N) \times H^{1/2}(\Sigma;\C^N) },
	\end{split}
	\end{equation*}
	where $\spf{\cdot}{\cdot}_{H^{-1/2}(\Sigma;\C^N) \times H^{1/2}(\Sigma;\C^N) }$ denotes the sesquilinear duality product, which is anti-linear in the second argument, on  $H^{-1/2}(\Sigma;\C^N) \times H^{1/2}(\Sigma;\C^N)$.
	Hence, $\mathcal{C}_{\overline{ z}}'$ is an extension of $\mathcal{C}_z$ which is bounded in $H^{-{1/2}}(\Sigma;\C^N)$, i.e. the claim is true for $r = -\frac{1}{2}$. By interpolation, we conclude that $\mathcal{C}_z$ gives rise to a bounded map in $H^r(\Sigma; \mathbb{C}^N)$ for any $r \in [-\frac{1}{2}, \frac{1}{2}]$. 
	
	(ii) First, for $\varphi \in H^{1/2}(\Sigma;\C^N)$ the definition of $\mathcal{C}_z$ in~\eqref{def_C_z} and the relation~\eqref{eq_gamma} imply
	\begin{equation} \label{Plemelj_Sokhotsky}
	\begin{split}
	\mathcal{C}_ z \varphi &= \frac{1}{2} \tr^+ (	\Phi_z \varphi)_+ + \frac{1}{2} \tr^- (	\Phi_z \varphi)_- =  \mp \frac{1}{2} (\tr^+ (	\widehat{\Phi}_z \varphi)_+ - \tr^- (	\widehat{\Phi}_z \varphi )_- )   +\tr^\pm (	\Phi_z \varphi)_\pm \\
	&=\pm  \frac{i}{2} (\alpha \cdot \nu ) \Gamma_0	\widehat{\Phi}_z \varphi  +\tr^\pm (	\Phi_z \varphi)_\pm =\pm  \frac{i}{2} (\alpha \cdot \nu )\varphi  +\tr^\pm (	\Phi_z \varphi)_\pm,
	\end{split}
	\end{equation}
	which is the claimed identity for $\varphi \in H^{1/2}(\Sigma;\C^N)$.
	If $\varphi  \in H^r(\Sigma;\C^N)$ and $r \in (0,\frac{1}{2})$, then the assertion follows from~\eqref{Plemelj_Sokhotsky} by continuity and density. 
\end{proof}

\begin{remark}
 We note that the results stated in Proposition~\ref{prop_Phi_z} and Proposition~\ref{proposition_C_z} are not optimal 
 in the sense of the maximal possible range of Sobolev indices $r$; the stated mapping properties can be proved for a wider range of Sobolev indices in a similar manner as in \cite[Proposition~4.4 and Corollary~4.5]{BHSS22}, but the present formulation is sufficient for the proof of Theorem~\ref{THEO_MAIN}.
 More precisely, Proposition~\ref{prop_Phi_z}~(i) can be extended to $r\in [-\tfrac{1}{2},\frac{1}{2}]$ and (ii) also remains valid for 
 $\varphi\in H^{r}(\Sigma; \mathbb{C}^N)$ with $r\in[-\tfrac{1}{2},\frac{1}{2}]$. Furthermore, the Plemelj-Sokhotsky formula in Proposition~\ref{proposition_C_z}~(ii) can be generalized to all $r\in [-\frac{1}{2},\frac{1}{2}]$.
\end{remark}}

\subsection*{Declarations.} 
The authors have no competing interests to declare that are relevant to the content of this article.

\subsection*{Acknowledgement.}
This research was funded in whole by the Austrian Science Fund (FWF) 10.55776/P 33568-N. For the purpose of open access, the author has applied a CC BY public copyright licence to any Author Accepted Manuscript version arising from this submission. This publication is based upon work from COST Action CA 18232 MAT-DYN-NET, supported by COST (European Cooperation in Science and Technology), www.cost.eu.

\bibliographystyle{abbrv}

\end{document}